\documentclass[a4paper,10pt]{amsart}
\usepackage{amsmath}
\usepackage{amssymb}
\usepackage{amsthm}
\usepackage[arrow, matrix, curve]{xy}
\usepackage{lmodern,dsfont}
\usepackage{enumerate}
\usepackage{enumitem}
\newtheorem{theorem}{Theorem}[section]
\newtheorem{proposition}[theorem]{Proposition}
\newtheorem{corollary}[theorem]{Corollary}
\newtheorem{lemma}[theorem]{Lemma}
\newtheorem{definition}[theorem]{Definition}
\newtheorem{definition and notation}[theorem]{Notation and Definition}
\newtheorem{remark}[theorem]{\bf Remark}
\newtheorem{condition}[theorem]{Condition}

\title{Root Parametrized Differential Equations for the classical groups}
\author{Matthias Seiss}
\email{matthias.seiss@mathematik.uni-kassel.de}
\address{Universit\"at Kassel\\ 
Fachbereich 10 \\ 
Heinrich Plett Str. 40 \\ 34132 Kassel \\ Germany.}
\begin{document}

\begin{abstract}
Let $C \langle t_1, \dots t_l\rangle$ be the differential field generated by $l$ differential indeterminates $\textit{\textbf{t}}=(t_1, \dots ,t_l)$ 
over an algebraically closed field $C$ of characteristic zero. We develop a lower bound criterion for the differential Galois group $G(C)$ 
of a matrix parameter differential equation $\partial(\textit{\textbf{y}})=A(\textit{\textbf{t}})\textit{\textbf{y}}$ over 
$C \langle t_1, \dots t_l\rangle$ and we prove that every connected linear algebraic group is the Galois group of  
a linear parameter differential equation over $C\langle t_1 \rangle$. As a second application we compute explicit and nice linear parameter 
differential equations over $C\langle t_1, \dots, t_l \rangle$ for the 
groups $\mathrm{SL}_{l+1}(C)$, $\mathrm{SP}_{2l}(C)$, $\mathrm{SO}_{2l+1}(C)$, $\mathrm{SO}_{2l}(C)$, i.e.~for the classical groups of type $A_l$, $B_l$, $C_l$, $D_l$, 
and for $\mathrm{G}_2$ (here $l=2$).
\end{abstract}

\maketitle

\section*{Introduction}
Let $C$ be an algebraically closed field of characteristic zero and consider a linear algebraic group $G$ over $C$. In differential Galois theory one of the classical
questions asks 
whether we can realize $G(C)$ as a differential Galois group over some differential field $F$ with constants $C$ and if the answer is positive can we construct an explicit 
linear differential equation with differential Galois group $G(C)$? Answers are only partially known for some groups and fields.\\ 
In some settings bound criteria for the differential Galois group play an essential role.
A well-established upper bound criterion is given in Proposition~\ref{Prop_upper}. It states that if the defining matrix $A$ of a linear differential equation 
$\partial(\textit{\textbf{y}})=A\textit{\textbf{y}}$ is contained in the Lie algebra $\mathfrak{g}(F)$ of $G(F)$, then the differential Galois group is a subgroup 
of $G(C)$. In the literature there is also a well-known lower bound criterion. 
Roughly-speaking Proposition~\ref{thm_lower} says that a suitable differential conjugate of the defining matrix $A$ lies in the Lie algebra of the differential Galois group.
But for a successful application of Proposition~\ref{thm_lower} we need to guarantee that a geometric condition is satisfied. 
This condition is automatically fulfilled if the differential base field is the field of rational functions $C(z)$ with standard derivation $\frac{d}{dz}$.
Using these bounds
C. Mitschi and M. Singer developed in \cite{MitschiSinger} a constructive method for the realization of connected linear algebraic groups over $C(z)$ and the general 
inverse problem over the same differential base field could be solved by J. Hartmann in \cite{Hart}.\\ 
Let $C \langle \textit{\textbf{t}} \rangle $ be the differential field generated by $l$ differential indeterminates $\textit{\textbf{t}}=(t_1, \dots ,t_l)$ over $C$ and 
for a differential indeterminate $y$ denote by $C\langle \textit{\textbf{t}} \rangle \{y\}$
the differential ring generated by $y$ over $C\langle \textit{\textbf{t}} \rangle$. In the following we mean 
by a matrix parameter differential equation a matrix differential equation 
$\partial(\textit{\textbf{y}})=A(\textit{\textbf{t}})\textit{\textbf{y}}$ 
with $A(\textit{\textbf{t}}) \in C \langle \textit{\textbf{t}} \rangle^{n \times n}$ and by a linear parameter differential equation  
an equation of the form
\begin{equation*}
 L(y,\textit{\textbf{t}})= y^{(n)} - \sum_{i=0}^{n-1} a_i(\textit{\textbf{t}})y^{(i)} \in C\langle t_1, \dots, t_l\rangle \{y\}. 
\end{equation*}
It is well-known that two such types of equations can be converted into each other (see for instance \cite{KovacicCyclicVector}).\\ 
Considering the structure of the Lie algebras of the classical groups we found matrix parameter differential equations   
which define naturally very nice and at the same time very general linear parameter differential equations (see Theorem~\ref{main_thm_intro} below).
Motivated by this observation we want to determine the differential Galois group of a matrix parameter differential equation. 
As in \cite{MitschiSinger} we intend to apply bound criteria for the differential Galois group to the defining matrix $A(\textit{\textbf{t}})$.
As an upper bound we can use again Proposition~\ref{Prop_upper}. But 
unfortunately, we cannot use Proposition~\ref{thm_lower} as a lower bound criterion, 
since we have no information
whether the assumption of Proposition~\ref{thm_lower} is satisfied or not.\\
In the first part of this article we develop a lower bound criterion for the differential Galois group of a matrix parameter differential equation.  
For a defining matrix $A(\textit{\textbf{t}}) \in C \langle \textit{\textbf{t}} \rangle^{n \times n}$, let $R_1$ be a localization of 
$C\{ \textit{\textbf{t}} \}$ by a finitely generated multiplicative subset 
of $C\{ \textit{\textbf{t}} \}$ such that $A(\textit{\textbf{t}}) \in R_1^{n \times n}$, where $C \{ \textit{\textbf{t}} \}$ denotes the differential ring generated by the 
indeterminates $\textit{\textbf{t}}$ over $C$. 
We denote by $G(C)$ the differential Galois group of 
$\partial(\textit{\textbf{y}})=A(\textit{\textbf{t}})\textit{\textbf{y}}$ over $C \langle \textit{\textbf{t}} \rangle $. Further let
$\sigma:C\{ \textit{\textbf{t}} \} \rightarrow C[z]$ be a specialization of the differential indeterminates to the differential subring $C[z]$ of $C(z)$ such that 
it extends to a specialization $\sigma : R_1 \rightarrow R_2$ where $R_2$ is a suitable finitely generated localization of $C[z]$.  
One result of this article is the following theorem (see also Theorem~\ref{infornmal_spez_bound} and~\ref{specialization_bound}).
\begin{theorem}(Specialization Bound)\\
The differential Galois group $H(C)$ of the specialized equation $\partial(\textit{\textbf{y}})=A(\sigma(\textit{\textbf{t}}))\textit{\textbf{y}}$ over $C(z)$ 
is a subgroup of $G(C)$.
\end{theorem}
The proof uses differential maps from the coordinate rings  
for $\mathrm{GL}_n$ over $R_1$ and $R_2$ to suitable rings of power series, where the differential structure 
on these coordinate rings is defined by $A(\textit{\textbf{t}})$ and $A(\sigma(\textit{\textbf{t}}))$ accordingly.  
These mappings are obtained from Taylor maps, which we will construct in Chapter~\ref{chap3}. The images of these Taylor maps will then define Picard-Vessiot rings.
In Chapter~\ref{chap4} we will show that we can specialize the coefficients of the corresponding power series 
appropriately such that $\sigma$ extends to a specialization of Picard-Vessiot rings. 
We prove then our specialization bound. We want to note that L. Goldman presented in \cite{Gold} a similar approach using so-called 
{\it{analytic specializations}}.\\
As a first application we show in Chapter~\ref{chap5} that for every connected linear algebraic group there exists a linear parameter differential equation in one parameter.
In the special case of a semisimple linear algebraic group of Lie rank $l$, we prove that there exists a parameter equation in $l$ parameters 
which specializes to all differential equations of a specific type over $C(z)$ and that we cannot remove a parameter without loosing some of these equations.
\\
In the second part of this article we present a more interesting application of our specialization bound. We will prove Theorem~\ref{main_thm_intro} below. 
It gives explicit linear parameter differential equations for the 
groups $\mathrm{SL}_{l+1}(C)$, $\mathrm{SP}_{2l}(C)$, $\mathrm{SO}_{2l+1}(C)$, $\mathrm{SO}_{2l}(C)$, i.e. for the classical groups of type $A_l$, $B_l$, $C_l$, $D_l$, 
and for $\mathrm{G}_2$ (here $l=2$).
\begin{theorem}
\label{main_thm_intro}
The linear parameter differential equation
\begin{enumerate}
\item 
 $\begin{aligned}[t] 
L(y,\textbf{t})= y^{(l+1)} - \sum\nolimits_{i=1}^{l} t_i y^{(i-1)}=0
\end{aligned}$ 
has $\mathrm{SL}_{l+1}(C)$ as its differential Galois group over $C\langle t_1, \dots , t_l \rangle$.
\vspace{2mm}
\item $\begin{aligned}[t] 
L(y,\textbf{t})= y^{(2l)} - \sum\nolimits_{i=1}^{l} (-1)^{i-1} (t_i y^{(l-i)})^{(l-i)}=0
\end{aligned}$  
has $\mathrm{SP}_{2l}(C)$ as differential Galois group over $C\langle t_1, \dots , t_l \rangle$.
\vspace{2mm}
\item $\begin{aligned}[t] 
L(y,\textbf{t})= y^{(2l+1)} - \sum\nolimits_{i=1}^{l} (-1)^{i-1} ((t_i y^{(l+1-i)})^{(l-i)}+(t_i y^{(l-i)})^{(l+1-i)})= 0
\end{aligned}$\\
has $\mathrm{SO}_{2l+1}(C)$ as differential Galois group over $C\langle t_1, \dots , t_l \rangle$.
\vspace{2mm}
\item \label{main_thm_intro_point_so2l}  $\begin{aligned}[t] 
L(y,\textbf{t})=  y^{(2l)} -2 \sum\nolimits_{i=3}^{l} (-1)^{i} ((t_i y^{(l-i)})^{(l+2-i)}+(t_i y^{(l+1-i)})^{(l+1-i)})-
\end{aligned}$ \\
$\begin{aligned}[t]
(t_2 y^{(l-2)} + t_1 y)^{(l)} - ((-1)^l t_1 z_1 + z_2) -\sum\nolimits_{i=0}^{l-2} (t_2^{l-2-i} z_1)^{(i)} =0
\end{aligned}$ 
has $\mathrm{SO}_{2l}(C)$ as differential Galois group over $C\langle t_1, \dots , t_l \rangle$. 
\vspace{2mm}
\item $\begin{aligned}[t] 
L(y,\textbf{t})= y^{(7)} - 2t_2 y' - 2 (t_2 y)' - (t_1 y^{(4)})' - (t_1y')^{(4)} + (t_1(t_1y')')'= 0
\end{aligned}$ \\
has $\mathrm{G}_{2}(C)$ as differential Galois group over $C\langle t_1, t_2 \rangle$.
\end{enumerate} 
The substitutions $z_1$ and $z_2$ in (\ref{main_thm_intro_point_so2l}) 
can be found in Lemma~\ref{Lemma_SO_2l_equation}.
\end{theorem}
We would like to point out that the differential equations for $\mathrm{SP}_{2l}$ are self-adjoint and the ones for $\mathrm{SO}_{2l+1}$ and $\mathrm{G}_2$ are anti-self-adjoint. 
It is also worth mentioning that the method, which leads to the equations of the theorem, can also be applied to the remaining groups of exceptional type, that is to $F_4$, 
$E_6$, $E_7$ and $E_8$.\\
The proof of Theorem~\ref{main_thm_intro} is organized in the following way: 
In Chapter~\ref{chap6} we prove our so-called {\it Transformation Lemma} 
which is an essential tool for the proof of the theorem. It will be used in combination with the 
{\it Specialization Bound} in the subsequent chapters where we prove 
Theorem~\ref{main_thm_intro} for each group separately.\\
In \cite[Theorem 2.10.6]{Katz} N.~Katz presented an explicit linear differential equation over $C(z)$ with differential Galois group $G_2$. It is easy to check that 
his equation can be obtained from our equation for $G_2$ by a suitable specialization of the parameters. 
For further differential equations with group $G_2$ we refer to \cite{Dett_Reiter}.
In \cite{Frenkel/Gross}, E.~Frenkel and B. Gross introduced a uniform 
construction of a rigid irregular differential equation over $C(z)$ which has a given simple linear algebraic group as its differential Galois group.
Our construction of the defining matrices is a generalization of their construction   
and so their explicit equations for $\mathrm{SL}_{l+1}$, $\mathrm{SP}_{2l}$, $\mathrm{SO}_{2l+1}$ and $G_2$ (see \cite[Section 6]{Frenkel/Gross}) can be obtained 
from the corresponding differential equations of Theorem~\ref{main_thm_intro} by a suitable specialization of the parameters. \\
The equations in Theorem~\ref{main_thm_intro} have an easy and nice shape and at the same time define a large family of differential 
equations.
It is therefore natural to ask for the generic properties of these parameter differential equations. 
For instance, it would be interesting to know whether the defining matrix of every $G$-primitive extension $E/F$ lies in an orbit of a specialization of $A(\textit{\textbf{t}})$. 
We refer to 
\cite{Seiss} for a first result in case of $\mathrm{SL}_{l+1}(C)$ and for a more detailed explanation regarding this question. Some examples of generic equations can be found in \cite{Gold}
and \cite{LJAL}. 
 
\section*{PART I\\ The Specialization Bound}
\section{Bounds for the differential Galois group}\label{chap1}
In this section we present bound criteria for the differential Galois group. More precisely, besides  upper and lower bound criteria known from the literature, we introduce our 
specialization bound and sketch the main ideas. A detailed proof then follows in the subsequent sections.\\
Let $F$ denote a differential field with field of constants $C$.
An upper bound criterion for the differential Galois group is given by the following proposition (see \cite[Proposition 1.31]{P/S}).
\begin{proposition}\label{Prop_upper}
 Let $H$ be a connected linear algebraic group over $C$ and let $A \in \mathfrak{h}(F)$ in the Lie algebra of $H(F)$. Then
 the differential Galois group $G(C)$ of the differential equation $\partial(\textbf{y})=A \textbf{y}$ is contained (up to 
 conjugation) in $H(C)$.
\end{proposition}
Let $\partial(\textit{\textbf{y}})=A \textit{\textbf{y}}$  be a differential equation with differential Galois group $G$ and denote by 
$L$ a Picard-Vessiot ring for $\partial(\textit{\textbf{y}})=A \textit{\textbf{y}}$. 
By Kolchin's Structure Theorem (see \cite[Theorem 1.28]{P/S}) we know that the 
affine group scheme 
$\mathcal{Z}=\mathrm{Spec}(L)$ over $F$ is a $G$-torsor. In the situation when 
$\mathcal{Z}$ has a $F$-rational point or equivalently when $\mathcal{Z}$ is the 
trivial torsor, the following proposition (for a proof see \cite[Corollary 1.32]{P/S}) is a lower bound criteria for the differential Galois group. 
\begin{proposition}\label{thm_lower}
 Let $L$ be a Picard-Vessiot ring for $\partial(\textit{\textbf{y}}) =A \textit{\textbf{y}}$ over $F$ with connected differential Galois group $G(C)$ and
 let $\mathcal{Z}$ be the associated torsor. Let $A \in \mathfrak{h}(F)$ for a connected 
 linear algebraic group $H(C) \supset G(C)$. If $\mathcal{Z}$ is the trivial torsor, then
 there exists $B \in H(F)$ such that $\partial(B)B^{-1} + B A B^{-1} \in \mathfrak{g}(F)$.
\end{proposition}
For a differential ground field of cohomological dimension at most one, the 
associated torsor is automatically the trivial torsor (see \cite[Chapter 2.4]{SerreGaloisCohomology}). 
Thus over such differential fields, e.g.~over 
the rational function field $C(z)$ with standard derivation $\frac{d}{dz}$, Proposition~\ref{thm_lower} can be considered as a lower bound criterion. 
In \cite{MitschiSinger} C. Mitschi and 
M.F. Singer applied Proposition~\ref{Prop_upper} and~\ref{thm_lower} successfully to connected semisimple linear algebraic groups 
for the differential ground field $C(z)$. 
Unfortunately in our situation, i.e.~for the differential base field $C\langle t_1 , \dots , t_l \rangle$, we have no information if the associated torsor is trivial or not.
Therefore we cannot use Proposition~\ref{thm_lower} directly as a lower bound criterion. But we can consider a specialization of a differential equation 
$\partial(\textit{\textbf{y}})=A(\textit{\textbf{t}})\textit{\textbf{y}}$ 
over $C\langle \textit{\textbf{t}} \rangle$ to a differential equation $\partial(\textit{\textbf{y}})=A(\textit{\textbf{r}})\textit{\textbf{y}}$ over $C(z)$ which we obtain from evaluating 
the differential indeterminates $\textit{\textbf{t}}=(t_1,\hdots,t_l)$ at elements $\textit{\textbf{r}}=(r_1,\hdots,r_l) \in C[z]^l$ with the property that 
$A(\textit{\textbf{r}})$ is well-defined. 
\begin{definition and notation}
We denote in the following by $F_1$ the differential field 
$C\langle \boldsymbol{t} \rangle$ and by $F_2$ the standard differential field $C(z)$. 
Further, we denote by $R_1$ a localization of $C\{ \boldsymbol{t} \}$ by 
a finitely generated multiplicative subset of $C\{ \boldsymbol{t} \}$.
If we evaluate the elements of this 
multiplicative subset at some $\boldsymbol{r} \in C[z]^l$, which we choose  
such that the generators do not vanish, we obtain 
a multiplicative subset of $C[z]$ and we denote by $R_2$ the localization of $C[z]$ by this 
subset. This construction then yields a differential $C$-algebra homomorphism 
$\sigma: R_1 \rightarrow R_2$, $\boldsymbol{t}\mapsto \boldsymbol{r}$ which we call a
\textbf{$\boldsymbol{R}_1$-specialization}.\\
For a ring $R$ and a matrix of indeterminates $X=(X_{ij})$ we denote by $R[\mathrm{GL}_n]$ the 
ring $R[X_{ij},\mathrm{det}(X_{ij})^{-1}]$. For an ideal $I$ of a ring $R$ and a ring extension $R \subset S$
we mean by $(I)$ the ideal $I \cdot S$ of $S$.
\end{definition and notation}
We sketch the main ideas of our specialization bound: 
Let $A(\boldsymbol{t}) \in R_1^{n \times n}$ and so we obtain  
$A(\sigma(\boldsymbol{t})) \in R_2^{n \times n}$ for a $R_1$-specialization $\sigma$. 
We equip the rings $R_1[\mathrm{GL}_n]$ and $R_2[\mathrm{GL}_n]$  
with a well-defined differential structure which is given by the matrix equations
$\partial(X) = A(\boldsymbol{t}) X$ and 
$\partial(X) = A(\sigma(\boldsymbol{t})) X$.
It is then possible to extend $\sigma$ to a differential ring homomorphism 
$\sigma: R_1[\mathrm{GL}_n] \rightarrow  R_2[\mathrm{GL}_n]$. 
We choose now a differential ideal $I_1$ of $R_1[\mathrm{GL}_n]$ maximal with
the property that $I_1 \cap R_1 = (0)$. Then, since $\sigma$ is a 
differential homomorphism, $\sigma(I_1)$ is a differential ideal of $R_2[\mathrm{GL}_n]$.
Assuming that $\sigma(I_1)$ is a differential ideal 
with $\sigma(I_1)\cap R_2=(0)$, 
we can choose a differential ideal $I_2$ of $R_2 [\mathrm{GL}_n]$ 
maximal with the property $I_2 \cap R_2=(0)$ 
and which satisfies $\sigma(I_1) \subset I_2$. 
By construction the quotient ring $S_i$ of $R_i[\mathrm{GL}_n]$ by $I_i$ 
does not have any non-trivial differential ideals $Q_i$ with $Q_i \cap R_i=(0)$. 
Further, the ideal $(I_i)$
is a maximal differential ideal of $F_i[\mathrm{GL}_n]$ by the 
maximality of $I_i$ and so 
the differential ring $S_i$ injects into the Picard-Vessiot ring
\begin{equation*}
 L_i  = F_i[\mathrm{GL}_n] / (I_i).
\end{equation*} 
Now, since the differential Galois group $G_i(C)$ of $L_i$ has to stabilize 
the ideal $I_i$ and $\sigma(I_1) \subset I_2$, we expect intuitively that 
the differential Galois group $G_2(C)$ of $L_2$ has to satisfy the same 
and even more conditions than the differential Galois group $G_1(C)$ of $L_1$. 
Therefore $G_2(C)$ should be a subgroup of $G_1(C)$. 
\begin{theorem}\label{infornmal_spez_bound}
 Let $\partial(\textbf{y})=A(\textbf{t})\textbf{y}$ be a matrix parameter differential equation over $F_1$ with differential 
 Galois group $G(C)$ and suppose
 $A(\textbf{t}) \in R_{1}^{n \times n}$. Let $\sigma$ 
 be a surjective $R_1$-specialization. Then the differential Galois group of $\partial(\textbf{y})= A(\sigma(\textbf{t}))\textbf{y}$ 
 over $F_2$ is a subgroup of $G(C)$.
\end{theorem}
For a formalization and completion of the proof of Theorem~\ref{infornmal_spez_bound} 
we will consider in Section~\ref{chap3} Taylor maps from the
differential rings $R_i[\mathrm{GL}_n]$ to suitable rings of power series.
The purpose of these Taylor maps is to show 
that there exists a differential ideal 
$I_1$ maximal with the property $I_1 \cap R_1=(0)$ such that 
its specialization satisfies $\sigma(I_1) \cap R_2=(0)$. 
\section{Differential rings}\label{chap2}  
Let $R$ be a differential ring with field of constants $C$. 
Suppose that $R$ is an integral domain and denote $\mathrm{Quot}(R)$ by $F$.  
For $\partial(\textit{\textbf{y}})=A\textit{\textbf{y}}$ 
with $A \in R^{n \times n}$ let a differential structure on 
$R[\mathrm{GL}_n]$ be defined by $\partial(X)=AX$.
We consider in the following differential ideals $I$ in $R[\mathrm{GL}_n]$ which satisfy the following condition:
\begin{condition} $\,$ \label{condition}
\begin{enumerate}[label=(\alph*),ref=\thetheorem(\alph*)]
 \item \label{condition_a} We have $I \cap R=(0)$.
 \item \label{condition_b} If for $q \in I$ there exists $r \in R$ such that 
 $r^{-1} q \in R[\mathrm{GL}_n]$, then $r^{-1} q  \in I$.
\end{enumerate}
\end{condition}

\begin{definition}$\,$ 
\begin{enumerate}
\item A differential ideal $I$ of $R[\mathrm{GL}_n]$ which satisfies 
Condition~\ref{condition_a} and is maximal with this 
property is called a {\bf relatively} maximal differential ideal.
\item Let $S$ be a differential ring and $R$ a differential subring of $S$. Then $S$ is called a {\bf $R$-simple} differential ring if 
$S$ contains no non-trivial differential ideals $I$ with the property $I\cap R = (0)$.
\end{enumerate}
\end{definition}
\begin{remark}
 If $I\subset R[\mathrm{GL}_n]$ is a relatively maximal differential ideal, 
 then $I$ satisfies automatically Condition~\ref{condition_b}, 
 since for any differential ideal $I\subset R[\mathrm{GL}_n] $ which satisfies 
 Condition~\ref{condition_a} the expanded ideal
 $I_{exp}:=(I) \cap R[\mathrm{GL}_n]$, where 
 $(I) \subset \mathrm{Quot}(R)[\mathrm{GL}_n]$, is a differential ideal and 
 satisfies Condition~\ref{condition_a} and \ref{condition_b}.
\end{remark}
A Picard-Vessiot ring is usually defined over a differential field. We give a definition of a Picard-Vessiot ring over the differential ring $R$. 
\begin{definition}
A differential ring $S$ over $R$ is called a Picard-Vessiot ring for the differential equation $\partial(\textit{\textbf{y}})=A\textit{\textbf{y}}$ 
with $A \in R^{n \times n}$ if
\begin{enumerate}
 \item $S$ is a $R$-simple differential ring.
 \item $S$ contains a fundamental solution matrix $Z \in \mathrm{GL}_n(S)$ and $S$ is generated as a ring by the entries of $Z$ and the inverse of the 
 determinant.
 \item the field of constants of $S$ is $C$.
\end{enumerate}
\end{definition}
If $S$ is a Picard-Vessiot ring over $R$, then the differential ring $L:=S \otimes_R F$ is a simple differential ring,  
since if $L$ would have a non-trivial differential ideal, then its non-trivial intersection with $S \otimes_R R$ 
would yield a differential ideal of $S$. 
It has therefore $C$ as its field of
constants and since it contains a fundamental solution matrix $Z$ and it is generated by the entries of $Z$ over $F$, 
it follows that $L$ is a Picard-Vessiot ring for $\partial(\textit{\textbf{y}})=A\textit{\textbf{y}}$ over $F$. 
Conversely, if $L$ is a Picard-Vessiot ring 
for $\partial(\textit{\textbf{y}})=A\textit{\textbf{y}}$ 
with $A \in R^{n \times n}$ over $F$, then the $R$-algebra generated by the entries of a fundamental solution matrix and the inverse of the determinant is a 
Picard-Vessiot ring $S$ over $R$ for $\partial(\textit{\textbf{y}})=A\textit{\textbf{y}}$. 
In fact, $S$ is $R$-simple, since if $S$ would have a non-trivial differential ideal $I$ with 
$I \cap R=(0)$, then the ideal $(I)$ would be a non-trivial differential ideal of $L$ and so $(I)=L$. 
But then $1 \in (I)$ and so there exists nonzero $r \in R$ with $r \in I$ which contradicts to $I \cap R=(0)$.
\begin{definition}
 The differential Galois group $\mathrm{Gal}_{\partial}(S/R)$ of a differential equation $\partial(\textit{\textbf{y}})=A\textit{\textbf{y}}$ 
with $A \in R^{n \times n}$ is defined as the group of differential $R$-algebra automorphisms of a Picard-Vessiot ring $S$ over $R$ for the equation. 
\end{definition}
Let $S$ be a Picard-Vessiot ring over $R$ and let $L = S \otimes_R F$ be the Picard-Vessiot ring over $F$ obtained from $S$ for 
$\partial(\textit{\textbf{y}})=A\textit{\textbf{y}}$. Then there is an obvious bijection 
\begin{equation*}
 \mathrm{Gal}(S/R) \leftrightarrow \mathrm{Gal}(L/F).
\end{equation*}
Lemma \ref{lemma_bijection1} and Lemma \ref{lemma_bijection2} below
are well known in the case of Picard-Vessiot rings over differential fields.
\begin{lemma}\label{lemma_bijection1}
 Let $R$ be a differential ring with the properties that $R$ is an integral domain, 
 the constants of $R$ and $\mathrm{Quot}(R)$ are $C$. 
 We extend the derivation $\partial$ of $R$ to a derivation on 
 $R[Y,\mathrm{det}(Y)^{-1}]$ by $\partial(Y)=0$ where $Y:=(Y_{ij})$ are indeterminates. 
 Then the map $ \delta: I \mapsto (I) $
 from the set of ideals in the subring $C[Y, \mathrm{det}(Y)^{-1}]$ to the set of 
 differential ideals in $R[Y,\mathrm{det}(Y)^{-1}]$ which satisfy 
 Condition~\ref{condition} is a bijection with inverse map 
 \begin{equation*}
   \delta^{-1}: J \mapsto J \cap C[Y, \mathrm{det}(Y)^{-1}].
 \end{equation*}
\end{lemma}
\begin{proof}
The map $\delta^{-1}$ is obviously well-defined. To prove that $\delta$ is also well-defined, 
we need to show that the image of an ideal $I$ under $\delta$ satisfies 
Condition~\ref{condition}.
We fix a $C $-basis $\{ e_i \mid i \in N \}$ of $R $ 
and we extend it to a $C $-basis $\{ e_i \mid i \in M \}$
of $\mathrm{Quot}(R) $. 
Then the first basis is also a basis of the free $C[Y, \mathrm{det}(Y)^{-1}]$-module 
$R[Y, \mathrm{det}(Y)^{-1}]$ and the second one of the free $C[Y, \mathrm{det}(Y)^{-1}]$-module 
$\mathrm{Quot}(R)[Y, \mathrm{det}(Y)^{-1}]$.
Let $q \in (I)$ and assume that $r^{-1} q \in R[Y, \mathrm{det}(Y)^{-1}]$ for  $r \in R$. It is clear that 
$q$ has a unique expression $q= \sum_{i \in N} q_i e_i$ with $q_i$ in $I$.
We have  
\begin{equation*}
r^{-1}q = r^{-1} (\sum_{i \in N} q_i e_i )= \sum_{i \in N} q_i  (r^{-1} e_i) =  \sum_{i \in N} q_i  ( \sum_{j \in M }c_{ij} e_j)
= \sum_{i \in M} \tilde{q}_i    e_i
\end{equation*}
with $r^{-1} e_i = \sum_{j \in M }c_{ij} e_j$ for $c_{ij} \in C$ and $\tilde{q}_i \in I$ .
Since $r^{-1} q \in R[Y, \mathrm{det}(Y)^{-1}]$, we conclude that $\tilde{q}_i$ can only be nonzero for indices of $N$
and so $r^{-1}q \in (I)$.\\   
We need to show that the two maps are inverse to each other.\\
(a) We have to show that for an ideal $I$ of $C[Y, \mathrm{det}(Y)^{-1}]$ it follows that 
$I=(I) \cap C[Y, \mathrm{det}(Y)^{-1}]$. It is clear that $I$ is contained in the right side. 
For the other inclusion let $\{ e_i \mid i \in N \}$ be a $C$-basis of $I$. 
This basis is also a basis of the free $R$ module $(I)$. It is easy to see that an element 
$q=\sum_{i \in N} r_i e_i $ of $(I)$ is a constant if and only if all $r_i $ are elements of $ C$.\\
(b) We have to show that an ideal $J$ of  $R[Y, \mathrm{det}(Y)^{-1}]$ which 
satisfies Condition~\ref{condition}
is generated by the ideal $I:= J \cap C[Y, \mathrm{det}(Y)^{-1}]$. To this purpose we fix a $C$-basis 
$\{ e_i \mid i \in N \}$ of $C[Y, \mathrm{det}(Y)^{-1}]$ which is also a basis of the free $R$-module 
$R[Y, \mathrm{det}(Y)^{-1}]$.  Then $q \in J$ writes for $r_i \in R$ uniquely as $q= \sum_{i \in N} r_i e_i$.  
We show by induction on the length $l(q)$, that is the number of nonzero $r_i$ in the expression for $q$, 
that $q \in (I)$. Let $l(q)=1$, 
that is, $q=r_i e_i$. Then Condition~\ref{condition_b} forces 
$q \in (I)$. Let $l(q)>1$. If all $r_i$ in the expression
of $q$ are elements of $C$ or all $r_i$ are of shape $r_i = r c_i$ with $r \in R$
and $c_i \in C$, then Condition~\ref{condition_b} yields $q \in (I)$. 
Without loss of generality let $r_{i_1} $ be not a constant and let $r_{i_1}$ and 
$r_{i_2}$ be $C$-linearly independent, that is, there exists no $c \in C$ 
such that $r_{i_1}=c r_{i_2}$. Note that 
$r_{i_1} \partial(r_{i_2}) - r_{i_2} \partial(r_{i_1}) = 0$ if and only if $r_{i_1}$, $r_{i_2}$ are $C$-linearly dependent.
Then 
 \begin{equation*}
\tilde{q}:=   r_{i_1} \partial(q) - \partial(r_{i_1})q = 
   \sum_{i \in M_q \setminus \{ i_1 \}} \tilde{r}_{i}  e_i \neq 0.
 \end{equation*} 
 The induction assumption implies 
 that $\tilde{q} \in (I)$. The same holds for $\tilde{r}_{i_2} q - r_{i_2} \tilde{q}$.
 We conclude that $\tilde{r}_{i_2} q  \in (I)$. Since $(I)$ satisfies 
 Condition~\ref{condition_b}, it follows that $q \in (I)$.
\end{proof}
Before we prove Lemma \ref{lemma_bijection2} below, we show the following lemma.
\begin{lemma} \label{lemma_about_torsion}
 Let $I_1$ and $I_2$ be two differential ideals of $R[\mathrm{GL}_n]$ which satisfy 
Condition~\ref{condition_a} and \ref{condition_b}. 
\begin{enumerate}[label=(\arabic*), ref=\thetheorem(\arabic*)]
\item \label{lemma_about_torsion_point1} Then the $R$-module 
$R[\mathrm{GL}_n]/I_1 \otimes_{R} R[\mathrm{GL}_n]/I_2$
is torsion-free. 
\item \label{lemma_about_torsion_point2} If $I_1$ is a relatively maximal differential ideal and $S=R[\mathrm{GL}_n]/I_1$,
then the $S$-module $S \otimes_R R[\mathrm{GL}_n]/I_2 \cong S[\mathrm{GL}_n]/(I_2)$
is torsion-free. 
\end{enumerate}
\end{lemma}
\begin{proof} 
(1) Since the ideals $I_1$ and $I_2$ satisfy 
Condition~\ref{condition_b}, the ideal 
\begin{equation*}
J:=(I_1 \otimes_R  R[\mathrm{GL}_n]) + (R[\mathrm{GL}_n] \otimes_R I_2 )
\end{equation*}
of $R[\mathrm{GL}_n] \otimes_{R} R[\mathrm{GL}_n] $ 
also satisfies Condition~\ref{condition_b}. Now, since 
$R[\mathrm{GL}_n] \otimes_{R} R[\mathrm{GL}_n]$ has no $R$-torsion, the $R$-module
$R[\mathrm{GL}_n] \otimes_{R} R[\mathrm{GL}_n] /J$
is torsion-free if and only if $J$ satisfies 
Condition~\ref{condition_b}. 
Then the statement in (1) follows from isomorphism 
of $R$-algebras 
\begin{equation*}
R[\mathrm{GL}_n]/I_1 \otimes_{R} R[\mathrm{GL}_n]/I_2 \cong
R[\mathrm{GL}_n] \otimes_{R} R[\mathrm{GL}_n] /J.
\end{equation*}
(2) It follows from (1) and 
\begin{equation*}
S[\mathrm{GL}_n]/(I_2) \cong S \otimes_R R[\mathrm{GL}_n]/I_2 \cong 
R[\mathrm{GL}_n]/I_1 \otimes_R R[\mathrm{GL}_n]/I_2
\end{equation*}
that $S[\mathrm{GL}_n]/(I_2)$ is a torsion-free $R$-module. 
We consider the $S$-module homomorphism $\phi: S[\mathrm{GL}_n]/(I_2) \rightarrow 
S[\mathrm{GL}_n]/(I_2) \otimes_R F$. Since $S[\mathrm{GL}_n]/(I_2)$ has no 
$R$-torsion, $\phi$ is an injective $R$-module homomorphism and so it is also 
injective as a $S$-module homomorphism. 
The differential ring $(S \otimes_R F)$ is $\partial$-simple and so 
\cite[Theorem 4.3]{Maurischat_simple} yields that the differential 
$(S \otimes_R F)$-module 
$(S \otimes_R F) \otimes_R R[\mathrm{GL}_n]/I_2 $ is projective. 
Thus, it is a torsion-free $(S \otimes_R F)$-module and so 
it is also a torsion-free $S$-module. We conclude with
\begin{equation*}
S[\mathrm{GL}_n]/(I_2) \otimes_R F \cong (S \otimes_R R[\mathrm{GL}_n]/I_2 )
\otimes_R F \cong (S \otimes_R F) \otimes_R R[\mathrm{GL}_n]/I_2 
\end{equation*} 
that $S[\mathrm{GL}_n]/(I_2) \otimes_R F$ is a torsion-free $S$-module and so, since 
$\phi$ is an injective $S$-module homomorphism, $S[\mathrm{GL}_n]/(I_2)$ 
is a torsion-free $S$-module.  
\end{proof}
\begin{lemma}\label{lemma_bijection2}
Let $S$ be a Picard-Vessiot ring over $R$. 
 Then the map 
 \begin{equation*}
  \iota: I \mapsto (I)
 \end{equation*}
 from the set of ideals in $R[\mathrm{GL}_n]$ which satisfy Condition~\ref{condition}  
 to the set of $\mathrm{Gal}(S/R)$ invariant ideals in $S[\mathrm{GL}_n]$ which 
 fulfill Condition~\ref{condition} 
 is a bijection where $(I)$ means the ideal $I \cdot S[\mathrm{GL}_n]$
 of $S[\mathrm{GL}_n]$. Its inverse map is  
 \begin{equation*}
  \iota^{-1}: J \mapsto J \cap R[\mathrm{GL}_n].
 \end{equation*}
\end{lemma}
\begin{proof}
It is easy to verify that the map $\iota^{-1}$ is well-defined.
Since  $S[\mathrm{GL}_n]/I \cdot S[\mathrm{GL}_n]$ has no $S$-torsion by 
Lemma~\ref{lemma_about_torsion_point2}, the ideal $I \cdot S[\mathrm{GL}_n]$ satisfies 
Condition~\ref{condition_b}. 
Hence, the map $\iota$ is also well-defined. \\
(a) For an ideal $I$ of $R[\mathrm{GL}_n]$ which satisfies Condition~\ref{condition} 
we show that $I=I \cdot S[\mathrm{GL}_n] \cap R[\mathrm{GL}_n]$. It is clear that 
$I \subset I\cdot S[\mathrm{GL}_n] \cap R[\mathrm{GL}_n]$. For the other inclusion we consider the 
extension $S_F := F \otimes_R S $ and the ideal $I_F:=I\cdot F[\mathrm{GL}_n]$ where 
$F$ denotes $\mathrm{Quot}(R)$. 
One chooses a $F$-basis of the module $I_F $. This basis is also a basis 
of the free $S_F$-module $I_F \cdot S_F[\mathrm{GL}_n]$. 
We conclude similar as in the proof of Lemma~\ref{lemma_bijection1} part (a) that 
$I_F = I_F \cdot S_F[\mathrm{GL}_n] \cap F[\mathrm{GL}_n]$. 
Now let $f \in I \cdot S[\mathrm{GL}_n] \cap R[\mathrm{GL}_n]$. Trivially it holds  
$f \in I_F \cdot S_F[\mathrm{GL}_n] \cap F[\mathrm{GL}_n]$ and so
$f \in I_F$. Since $f \in R[\mathrm{GL}_n]$ it follows that 
$f \in I_F \cap R[\mathrm{GL}_n]$ which is equal to $I$ since $I$ satisfies 
Condition~\ref{condition}. 
\\
(b) For a $\mathrm{Gal}(S/R)$-invariant ideal $J$ of $S[\mathrm{GL}_n]$ which satisfies
Condition~\ref{condition} it is clear that $I:=J \cap R[\mathrm{GL}_n]$ satisfies 
also Condition~\ref{condition}.
Recall that then the ideal $I \cdot S[\mathrm{GL}_n]$ of $S[\mathrm{GL}_n]$ 
also satisfies Condition~\ref{condition_b} by 
Lemma~\ref{lemma_about_torsion_point2}.
We have to show that $J$ is equal to $I \cdot S[\mathrm{GL}_n]$.
Let $\{ e_i \mid i \in N \}$ be a $R$-basis of $R[\mathrm{GL}_n]$ which is also 
a $S$-module basis of $S[\mathrm{GL}_n]$. Then every $f \in J$ writes uniquely as 
$\sum_{i \in N} s_i e_i$ with $s_i \in S$. We proceed by induction on the length $l(f)$, 
that is, on the number of nonzero $s_i$ in the expression of $f$. 
If $f=s_i e_i$, then $e_i \in J \cap R[\mathrm{GL}_n]$ and so in $I\cdot S[\mathrm{GL}_n]$.
Now let $l(f)>1$. If all $s_i$ in the expression of $f$ lie in $R$ or if $f= s \sum r_i e_i$ 
with $r_i \in R$, then clearly $f$ is in $I \cdot S[\mathrm{GL}_n]$. 
Without loss of generality let $s_{i_1}$ and $s_{i_2}$ be $R$-linearly 
independent and let $s_{i_1}$ be not in $R$. 
Then there is $\gamma \in \mathrm{Gal}(S/R)$ such that 
$\gamma(s_{i_1})s_{i_2} \neq s_{i_1}\gamma (s_{i_2})$. The length of 
$g:=s_{i_1}\gamma(f)-\gamma(s_{i_1})f \neq 0$ is smaller than $l(f)$ 
and so $g$ lies by induction assumption in $I \cdot S[\mathrm{GL}_n]$. 
Further, the length of $(s_{i_1}\gamma(s_{i_2})-\gamma(s_{i_1})s_{i_2})f-s_{i_2}g$ 
is smaller than $l(f)$ and so it is also in $I \cdot S[\mathrm{GL}_n]$. 
We conclude that $(s_{i_1}\gamma(s_{i_2})-\gamma(s_{i_1})s_{i_2})f$ lies in 
$I \cdot S[\mathrm{GL}_n]$. Since $I\cdot S[\mathrm{GL}_n]$ 
satisfies Condition~\ref{condition}, we have that $f$ is an element of $I\cdot S[\mathrm{GL}_n]$. 
\end{proof}

\section{Formal Taylor maps}\label{chap3}

As we have seen in Chapter~\ref{chap1} the proof of the specialization bound is based on the idea 
to specialize a relatively maximal differential ideal of $R_1[\mathrm{GL}_n]$ to a 
differential ideal of $R_2[\mathrm{GL}_n]$ and then to choose a relatively maximal 
differential ideal  above it. 
But to make this idea work, we need to guarantee that there is a relatively maximal 
differential ideal of $R_1[\mathrm{GL}_n]$ such that its specialization is a 
differential ideal of $R_2[\mathrm{GL}_n]$ satisfying 
Condition~\ref{condition_a}.
To ensure that such an ideal exists, we will construct in this section 
Taylor maps from $R_1[\mathrm{GL}_n]$ and $R_2[\mathrm{GL}_n]$ in appropriate 
rings of power series.\\ 
It will turn out that the images of the two maps are Picard-Vessiot rings and so their 
kernels are relatively maximal differential ideals. The construction will be done 
in such a way that we can specialize the coefficients of the power series solutions 
representing $R_1[\mathrm{GL}_n]$ to the coefficients of the ones representing 
$R_2[\mathrm{GL}_n]$ and that it commutes with the specialization of 
$R_1[\mathrm{GL}_n]$ to  $R_2[\mathrm{GL}_n]$.
From this it will follow that 
the specialization of the kernel of the Taylor map for $R_1[\mathrm{GL}_n]$ is 
contained in the one for $R_2[\mathrm{GL}_n]$, that is, it is a relatively maximal 
differential ideal of $R_1[\mathrm{GL}_n]$ with the desired property.\\

{\bf(a)} In a first part we construct a Taylor map for $R_1[\mathrm{GL}_n]$. 
Since $R_1$ is generated by differential indeterminates and since later on we are
going to specialize them, a Taylor map for $R_1$ has to map the indeterminates to
power series which are also differentially algebraically independent over $C$ and 
whose coefficients can be specialized. 
In order to construct a differential embedding of $R_1$ in a ring of power series 
and to have enough freedom for specializations, we need to introduce new 
indeterminates, that is, 
we consider a field extension $C(\boldsymbol{\beta})/C$ generated 
over $C$ by infinitely many algebraically independent elements
$\boldsymbol{\beta}= \{ \beta_{ik} \mid 1 \leq i \leq l, \ k \in \mathbb{N}_0 \}$.
It will turn out that the Taylor map 
\begin{equation*}
R_1 \rightarrow C(\boldsymbol{\beta})[[T]], \ t_i \mapsto 
p_i:= \sum_{k \in \mathbb{N}_0}  1/k! \ \beta_{ik} T^k 
\end{equation*}
is a differential monomorphism (see Corollary~\ref{corollary_diff_transcendental}) 
where $C(\boldsymbol{\beta})[[T]]$ denotes 
the ring of power series in $T$ over $C(\boldsymbol{\beta})$ 
with derivation $\partial=\frac{d}{dT}$. Since the $\beta_{ik}$ are indeterminates, 
the power series $p_i$ can obviously be 
specialized to every power series in $C [[T]]$.\\
The Taylor map on $R_1[\mathrm{GL}_n]$ mentioned above will be the unique extension of 
the Taylor map for $R_1$ 
to a differential homomorphism such that the image of $X$ is congruent to $1_n$ 
modulo $T$. This Taylor map has the nice property that after arbitrary specializations 
of the coefficients the image of $X$ still is invertible, hence a fundamental solution 
matrix for a specialized equation. The hard task, however, is to show that the image of 
$R_1[\mathrm{GL}_n]$ under this Taylor map is indeed a 
Picard-Vessiot ring with constants $C$.\\  
In order to prove this statement, we need to extend in a first step for
technical reasons the constants $C$ of $R_1$ to $C(\boldsymbol{\beta})$, 
that is, we will construct a Taylor map for the differential ring
 \begin{equation*}
 R_{ \boldsymbol{\beta}} :=  R_1 \otimes_C C(\boldsymbol{\beta}).
\end{equation*}
Once we have seen that the kernel of this Taylor map is a relatively maximal differential ideal of 
$R_{ \boldsymbol{\beta}} [\mathrm{GL}_n]$, we show in a second step 
that the kernel of the restriction to $R_{1}[\mathrm{GL}_n]$
is a relatively maximal differential ideal 
(Theorem~\ref{lemma_taylor_emb_transcendental_without beta}).\\
We start with the construction of the Taylor map for $R_{ \boldsymbol{\beta}}[\mathrm{GL}_n]$.
Since the differential ring  $R_{\boldsymbol{\beta}}$
is differentially generated by the differential indeterminates
$\boldsymbol{t}$ over $C(\boldsymbol{\beta})$, we can consider $R_{ \boldsymbol{\beta}}$ as a 
polynomial ring generated by infinitely 
many transcendental elements $\partial^k(t_i)$ over $C(\boldsymbol{\beta})$.
In this polynomial ring the ideal 
\begin{equation*}
 P_1 := \langle \partial^k(t_i)  - \beta_{ik} \mid 1 \leq i \leq l, \ k \in \mathbb{N}_0 \rangle
\end{equation*}
is clearly a maximal ideal and so the image of an element $r$ 
of $R_{ \boldsymbol{\beta}}$ under the quotient map 
$R_{ \boldsymbol{\beta}} \rightarrow  R_{ \boldsymbol{\beta}}/P_1$ 
is the evaluation of $r$ at $\boldsymbol{\beta}$. We denote 
this image by $r(P_1)$.
 \begin{lemma}\label{lemma_taylor_emb_transcendental} 
  For $A \in R_{1}^{n \times n}$ let a derivation on $R_{\boldsymbol{\beta}}[\mathrm{GL}_n]$ 
 be defined by the matrix differential equation  
 $\partial(X)=AX$.
 Let $P_{1}$ be the maximal ideal of $R_{ \boldsymbol{\beta}}$ 
 as above and let $X(P_{1}) \in \mathrm{GL}_n(C(\boldsymbol{\beta}))$ be initial values.   
 Then the kernel of the differential ring homomorphism
 \begin{gather*}
 \tau: R_{\boldsymbol{\beta}}[\mathrm{GL}_n] \rightarrow C(\boldsymbol{\beta})[[T]], \\
  t_i \mapsto \sum_{k \in \mathbb{N}_0} 1/k! \  \partial^{k}(t_i)(P_{1}), \ 
  X_{ij} \mapsto \sum_{k \in \mathbb{N}_0} 1/k! \ \partial^{k}(X_{ij})(P_{1})
 \end{gather*}
 is a relatively maximal differential ideal of $R_{\boldsymbol{\beta}}[\mathrm{GL}_n]$. 
\end{lemma}
\begin{proof}
The map 
 \begin{equation*}
  \tau: R_{ \boldsymbol{\beta}} \rightarrow C(\boldsymbol{\beta})[[T]], \ r \mapsto \sum_{k \in \mathbb{N}_0} 
  1/k! \ \partial^{k} (r)(P_1)T^k
 \end{equation*}  
is clearly a differential ring homomorphism and by
Corollary~\ref{corollary_diff_transcendental} 
its image is contained in a purely differential transcendental extension of $C(\boldsymbol{\beta})$.
Hence the map $\tau$ is a differential monomorphism.\\
We extend $\tau$ to $R_{ \boldsymbol{\beta}}[\mathrm{GL}_n]$. The recursion 
\begin{equation*}
 A_1 := A \ \mathrm{and} \ A_k: = \partial(A_{k-1})+A_{k-1}A \ \mathrm{for} \ \mathrm{all} \ k \geq 2 
\end{equation*}
provides us with all higher derivatives $\partial^k (X)= A_k X$ and if we choose initial values 
$X(P_1) \in \mathrm{GL}_n(C(\boldsymbol{\beta}))$, we obtain values 
$\partial^{k} (X)(P_1) \in C(\boldsymbol{\beta})^{n \times n}$. 
So the Taylor map $\tau$ extends to a differential ring homomorphism
\begin{equation*}
 \tau: R_{ \boldsymbol{\beta}}[\mathrm{GL}_n] \rightarrow C(\boldsymbol{\beta})[[T]], \ 
 X_{ij} \mapsto  \sum_{k \in \mathbb{N}_0} 1/k! \ \partial^{k} (X_{ij})(P_1)T^k.
\end{equation*}
It is left to show that the kernel of $\tau$ is a relatively maximal differential ideal. 
Since $\tau(R_{ \boldsymbol{\beta}}) \subset C(\boldsymbol{\beta})[[T]]$ is an integral domain, 
we can take its quotient field 
$\mathrm{Quot}(\tau(R_{ \boldsymbol{\beta}})) \subset C(\boldsymbol{\beta})((T))$ 
which is a differential field with field of constants $C(\boldsymbol{\beta})$. We prove that 
$S := \mathrm{Quot}(\tau(R_{ \boldsymbol{\beta}})) [ \tau(X), \mathrm{det} (\tau(X))^{-1} ]$ 
is a $\partial_T$-simple differential ring, that is, $S$ has no proper, non-trivial differential ideals. 
By construction the matrix $\tau(X)$ is a fundamental solution matrix for the differential equation 
$\partial(\textit{\textbf{y}})= \tau (A) \textit{\textbf{y}}$ and $S$ is generated 
by the entries of $\tau(X)$ and the inverse of its determinant. 
Further, since $S$ is a subring of $C(\boldsymbol{\beta})((T))$, it is an integral domain 
and the constants of its field of fractions $\mathrm{Quot}(S) \subset C(\boldsymbol{\beta})((T))$
coincide with the field of constants of $\mathrm{Quot}(\tau(R_{ \boldsymbol{\beta}}))$. 
We apply now \cite[Corollary 2.7]{DyckerhoffInverseProb} and obtain that $S$ is a 
$\partial_T$-simple differential ring. Hence, the differential ring 
$\tau(R_{ \boldsymbol{\beta}}[\mathrm{GL}_n])$ is 
$\tau(R_{ \boldsymbol{\beta}})$-simple and so $\mathrm{kern}(\tau)$ 
is a relatively maximal differential ideal of $R_{ \boldsymbol{\beta}}[\mathrm{GL}_n]$. 
\end{proof}

To complete the proof of Lemma~\ref{lemma_taylor_emb_transcendental} we need to  
show that the power series $p_i$ generated a purely differential transcendental extension of $C(\boldsymbol{\beta})$. 
\begin{lemma}\label{lemmadifftranspowerseries}
For every $1 \leq i \leq l$ the power series $p_i$ is differentially algebraically independent
over the differential ring $C(\boldsymbol{\beta})[T]$ with derivation $\partial=\frac{d}{dT}$.
\end{lemma}
\begin{proof} 
Assume $p_i$ is differentially algebraic. This means that the transcendence degree of 
$C(\boldsymbol{\beta})(T,p_i, \partial(p_i), \dots)$ over $C(\boldsymbol{\beta})(T)$ is finite.
Then the transcendence degree of 
$C(\boldsymbol{\beta})(T,p_i, \partial(p_i), \dots)$
is finite over $C(\boldsymbol{\beta})$ and so the same holds for the field $C(\boldsymbol{\beta})(p_i, \partial(p_i), \dots)$.
Thus $p_i$ satisfies a differential equation coming from $C(\boldsymbol{\beta})\{ y\}$. Since the coefficients of this 
equation are contained in a finitely generated subfield of $C(\boldsymbol{\beta})$, 
we obtain by the same argument that $p_i$ satisfies a relation from $C \{y \}$. 
By \cite[Theorem 4.1]{Lipshitz&Rubel} there exist power series in
$C[[T]]$, for example $\sum_k T^{2^k}$, which are not differentially algebraic over $C$. 
Thus if we specialize the coefficients of $p_i$ to the coefficients of such a power series, we get a contradiction.
\end{proof}
 
\begin{corollary}\label{lemma_diff_transcendental}
 For every $1 \leq i \leq l$ the power series $p_i$ is differentially 
 algebraically independent over the differential ring 
 $C(\boldsymbol{\beta})\{p_1,\dots ,\check{p}_i,\dots,p_l \}$.
\end{corollary}
\begin{proof}  
Without loss of generality assume that the power series $p_l(T)$ is differentially algebraically dependent over 
$C(\boldsymbol{\beta})\{ p_1(T),\dots,p_{l-1}(T) \}$, i.e.~there exists a non-trivial differential polynomial
\begin{equation*}
  h(y) \in C(\boldsymbol{\beta})\{ p_1(T),\dots,p_{l-1}(T) \} \{ y\} 
\end{equation*}
such that $h(p_l(T)) =0$. Then the coefficients of 
$h(y)$ are power series with infinitely many coefficients in 
$C[\beta_{1j},\dots , \beta_{l-1,j} \mid j \in \mathbb{N}_0 ]$ and finitely many coefficients in $C(\boldsymbol{\beta})$. Denote by 
$C(\boldsymbol{\beta}_{l}) $ the subfield of $C(\boldsymbol{\beta})$ which is generated 
by $\boldsymbol{\beta}_{l}=\{ \beta_{lj} \mid j \in \mathbb{N}_0 \}$ over $C$. We define
a specialization $\sigma : C(\boldsymbol{\beta}) \rightarrow C(\boldsymbol{\beta}_{l}) $ 
by $\beta_{lj} \mapsto \beta_{lj}$ for $j \in \mathbb{N}_0$ and $\beta_{ij} \mapsto c_{ij}$ 
for $i \in \{1, \dots , l-1 \}$ and $j \in \mathbb{N}_0$, where we choose finitely many 
$c_{ij} \in C^{\times}$ and infinitely many $c_{ij} = 0$ such that 
$\sigma(h)(y) \in C(\boldsymbol{\beta}_{l})[T]$ has no pole and does not vanish. 
Thus we obtain a non-trivial differential algebraic relation with 
\begin{equation*}
\sigma(h(p_l(T)))= \sigma(h)(p_l(T))=0.
\end{equation*}
 But $p_l(T)$ is differentially transcendental 
 over the differential field $C(\boldsymbol{\beta}_{l})[T]$ by Lemma~\ref{lemmadifftranspowerseries}.
\end{proof}
\begin{corollary}\label{corollary_diff_transcendental}
 The differential field extension 
 $C(\boldsymbol{\beta})\langle p_1,\dots,p_l \rangle / C(\boldsymbol{\beta})$ 
 is a purely differentially transcendental extension.
\end{corollary}

In Lemma~\ref{lemma_taylor_emb_transcendental} we constructed a Taylor map 
$\tau$ for $R_{ \boldsymbol{\beta}}[\mathrm{GL}_n]$ 
such that its kernel is a relatively maximal differential ideal of 
$R_{ \boldsymbol{\beta}}[\mathrm{GL}_n]$. 
In order to prove that the kernel of the restriction of $\tau$ to $R_1[\mathrm{GL}_n]$ 
is also a relatively maximal differential ideal,
we construct an auxiliary Taylor map. The difference to the previous Taylor map is
that we choose finitely many coefficients in the Taylor series representing
the differential indeterminates of $R_{ \boldsymbol{\beta}}$ to be elements of $C$. 
More precisely, for $r \in \mathbb{N}_0$ and $c_{ij} \in C$ with $0 \leq j < r$ we 
consider the ideal 
\begin{equation*}
  \tilde{P}_1 := \langle \partial^j (t_{i}) - c_{ij}, \ \partial^k (t_{i}) - \beta_{ik} 
  \mid 1 \leq i \leq l, \ 0 \leq j < r , \ k \in \mathbb{N} , \ k \geq r \rangle
\end{equation*}
of $R_{ \boldsymbol{\beta}}$.
It is clearly a maximal ideal of $R_{ \boldsymbol{\beta}}$ and so the image under 
the quotient map $R_{ \boldsymbol{\beta}} \rightarrow R_{ \boldsymbol{\beta}}/\tilde{P}_1$
of an element $r \in R_{ \boldsymbol{\beta}}$ is the evaluation of $r$ at
$c_{ij}$ and $ \beta_{i,k}$. We denote its image by $r(\tilde{P}_1)$.
 
 \begin{lemma}\label{lemma_taylor_emb_semi_transcendental} 
 For $A \in R_{1}^{n \times n}$ let a derivation on $R_{ \boldsymbol{\beta}}[\mathrm{GL}_n]$ be defined by 
 the differential equation $\partial(X)=AX$ and let $\tilde{P}_{1}$ be the maximal 
 ideal of $R_{ \boldsymbol{\beta}}$ as above.  Let $X(\tilde{P}_{1}) \in \mathrm{GL}_n(C(\boldsymbol{\beta}))$
 be arbitrary initial values.
 Then the kernel of the differential ring homomorphism 
 \begin{gather*}
  \tilde{\tau} : R_{\boldsymbol{\beta}}[\mathrm{GL}_n] \rightarrow C(\boldsymbol{\beta})[[T]], \\ 
  t_i \mapsto \sum_{k \in \mathbb{N}_{0}} 1/k! \ \partial^{k}(t_i)(\tilde{P}_{1}), \ 
  X_{ij} \mapsto \sum_{k \in \mathbb{N}_{0}} 1/k! \ \partial^{k}(X_{ij})(\tilde{P}_{1}) 
 \end{gather*}
 is a relatively maximal differential ideal of $R_{\boldsymbol{\beta}}[\mathrm{GL}_n]$. 
\end{lemma}
\begin{proof}
The proof just works as the proof of Lemma~\ref{lemma_taylor_emb_transcendental} where we use this time  
the ideal $\tilde{P}_1$ instead of $P_1$. That the elements $\tilde{\tau}(t_1), \dots , \tilde{\tau}(t_l)$ generate 
a purely differential transcendental extension of $C(\boldsymbol{\beta})$ is shown in
Lemma~\ref{semi_transcendental_power series} below. 
\end{proof}
\begin{lemma}\label{semi_transcendental_power series}
Let $\tilde{\tau}$ be as in Lemma~\ref{lemma_taylor_emb_semi_transcendental}. 
Then the differential field extension 
$C(\boldsymbol{\beta}) \langle \tilde{\tau}(t_1), \dots , \tilde{\tau}(t_l) \rangle/ C(\boldsymbol{\beta}) $
is purely differential transcendental.
\end{lemma}
\begin{proof}
We need to show that $\tilde{\tau}(t_1)$ is 
differentially algebraically independent over $C(\boldsymbol{\beta})[T]$. 
It then follows as in Corollary~\ref{lemma_diff_transcendental} that 
$C(\boldsymbol{\beta})\langle \tilde{\tau}(t_1), \dots , \tilde{\tau}(t_l) \rangle$ is a purely 
differential transcendental extension of $C(\boldsymbol{\beta})$. We assume that $\tilde{\tau}(t_1)$ is
differentially algebraic over $C(\boldsymbol{\beta})[T]$. Then the same arguments as in 
Lemma~\ref{lemmadifftranspowerseries} yield that the transcendence degree of 
$C(\tilde{\tau}(t_1), \partial(\tilde{\tau}(t_1)), \dots)$ is finite over $C$. We conclude that the transcendence degree
of $C( \partial^m(\tilde{\tau}(t_1)), \partial^{m+1}(\tilde{\tau}(t_1)), \dots)$ over $C$ 
must also be finite for every $m \in \mathbb{N}_0$. 
But since only finitely many coefficients of $\tilde{\tau}(t_1)$ are elements of $C$ there is a $m$
such that $\partial^m(\tilde{\tau}(t_1))$ is equal to $\partial^m(p_1)$. 
We obtain a contradiction to Lemma~\ref{lemmadifftranspowerseries}.
\end{proof}
We show now that if we choose the same initial values the images of 
$R_{ \boldsymbol{\beta}}[\mathrm{GL}_n]$ under the two Taylor maps $\tau$ and $\tilde{\tau}$ 
become differentially isomorphic by sending the fundamental matrix $\tau(X)$ of $\mathrm{im}(\tau)$ 
to the fundamental matrix $\tilde{\tau}(X)$ of $\mathrm{im}(\tilde{\tau})$. 
It will then turn out in Corollary~\ref{corollary Kernels are equal} that the two maps have the same kernels.
\begin{lemma}\label{Lemma_isomorphic}
 For a matrix differential equation 
 $\partial(\textit{\textbf{y}})=A\textit{\textbf{y}}$ 
 with $A \in R_{1}^{n \times n}$ let $\tau$ and $\tilde{\tau}$ 
 be the Taylor maps of Lemma~\ref{lemma_taylor_emb_transcendental} 
 and~\ref{lemma_taylor_emb_semi_transcendental} 
 with initial values $X(P_{1})=X(\tilde{P}_{1})= 1_n$.  
 Then the map  
 \begin{equation*}
  \Psi : \mathrm{im}(\tau) \rightarrow \mathrm{im}(\tilde{\tau}) , \
  \tau(\boldsymbol{t}) \mapsto \tilde{\tau}(\boldsymbol{t}) \  and \ 
  \tau(X)  \mapsto  \tilde{\tau}(X) 
 \end{equation*}
 is a differential isomorphism.  
 \end{lemma}
\begin{proof} 
Note that the constants of the Picard-Vessiot rings $\mathrm{im}(\tau)$ and 
$\mathrm{im}(\tilde{\tau})$ are in both cases $C(\boldsymbol{\beta})$. 
The map  
\begin{gather*}
\psi :  \tau(R_{ \boldsymbol{\beta}}) \rightarrow \tilde{\tau}(R_{ \boldsymbol{\beta}}), \ 
\tau(\boldsymbol{t}) \mapsto \tilde{\tau}(\boldsymbol{t})  
\end{gather*}
is obviously a differential isomorphism, since $\tau(R_{ \boldsymbol{\beta}}) $ and 
$\tilde{\tau}(R_{ \boldsymbol{\beta}})$ are subrings of purely differentially 
transcendental extensions of $C(\boldsymbol{\beta})$ of the same degree by 
Corollary~\ref{corollary_diff_transcendental} and Lemma~\ref{semi_transcendental_power series}.  
Let  $\mathrm{im}(\tau) \otimes_{R_{\boldsymbol{\beta}}} \mathrm{im}(\tilde{\tau})$ be the tensor product of 
$\mathrm{im}(\tau)$ and $\mathrm{im}(\tilde{\tau})$ defined via $\psi$, that is, 
for two elements $p \in \tau(R_{ \boldsymbol{\beta}})$ and 
$\tilde{p} \in \tilde{\tau}(R_{ \boldsymbol{\beta}})$ we have the rule 
\begin{equation*}
p \otimes \tilde{p} = p \psi^{-1}(\tilde{p}) \otimes 1 = 1 \otimes \psi(p)\tilde{p}.
\end{equation*} 
The ring $\mathrm{im}(\tau) \otimes_{R_{\boldsymbol{\beta}}} \mathrm{im}(\tilde{\tau})$ becomes a differential ring 
in the obvious way, that is, for $s \otimes \tilde{s} \in 
\mathrm{im}(\tau) \otimes_{R_{\boldsymbol{\beta}}} \mathrm{im}(\tilde{\tau})$ the derivation is defined by  
$\partial(s \otimes \tilde{s}) = \partial(s) \otimes \tilde{s} + s \otimes \partial(\tilde{s})$ and it 
has no $R_{\boldsymbol{\beta}}$-torsion by 
Lemma~\ref{lemma_about_torsion_point1}.
Let $Z_{ij}$ be the entries of the matrix
\begin{equation*}
Z:=\tau(X)^{-1} \otimes \tilde{\tau}(X) .
\end{equation*}
One can show now as in \cite[Lemma 2.4]{DyckerhoffInverseProb} that  
$U:= C(\beta)[Z_{ij},\mathrm{det}(Z_{ij})^{-1}]$ is the algebra of constants of the differential ring 
$\mathrm{im}(\tau) \otimes_{R_{\boldsymbol{\beta}}} \mathrm{im}(\tilde{\tau})$ and that the map
\begin{equation*}
\Phi: \mathrm{im}(\tau) \otimes_{C(\boldsymbol{\beta})}  U \rightarrow 
\mathrm{im}(\tau) \otimes_{R_{\boldsymbol{\beta}}} \mathrm{im}(\tilde{\tau}),  \
(s \otimes u) \mapsto (s \otimes 1) u
\end{equation*}
is a differential isomorphism. In the last step of the proof one uses that 
$\mathrm{im}(\tau) \otimes_{R_{\boldsymbol{\beta}}} \mathrm{im}(\tilde{\tau})$ has no 
$\mathrm{im}(\tau)$-torsion by 
Lemma~\ref{lemma_about_torsion_point2}. 
The map $\Phi$ defines a bijection between the maximal relatively 
differential ideals of $\mathrm{im}(\tau) \otimes_{C(\boldsymbol{\beta})}  U$ and 
$\mathrm{im}(\tau) \otimes_{R_{\boldsymbol{\beta}}} \mathrm{im}(\tilde{\tau})$.
By \cite[Lemma 10.7]{Maurischat} the relatively maximal differential ideals of 
$\mathrm{im}(\tau) \otimes_{C(\boldsymbol{\beta})} U$ correspond to the maximal ideals of $U$, 
since $\mathrm{im}(\tau)$ is relatively $\partial$-simple. 
We show that 
\begin{equation*}
Q = \langle Z_{ij} - \delta_{ij} \otimes \delta_{ij}  \rangle 
\end{equation*}
is a maximal ideal of $U$. 
For this purpose let us consider the ring homomorphism 
\begin{equation*}
 \phi : \mathrm{im}(\tau) \otimes_{C(\boldsymbol{\beta})} \mathrm{im}(\tilde{\tau}) 
 \rightarrow \mathrm{im}(\tau) \cdot \mathrm{im}(\tilde{\tau}) 
 \subset C(\boldsymbol{\beta})[[T]] , \ s \otimes \tilde{s} \mapsto s \tilde{s}.
\end{equation*}
The subset of all power series with constant term zero of 
$\mathrm{im}(\tau) \cdot \mathrm{im}(\tilde{\tau})$ is a proper ideal.
Therefore the preimage of this ideal under $\phi$ is a proper ideal of the ring 
$ \mathrm{im}(\tau) \otimes_{C(\boldsymbol{\beta})} \mathrm{im}(\tilde{\tau})$ 
and due to the choice of the initial values it contains the elements
$Z_{ij} - \delta_{ij} \otimes \delta_{ij}$. The inclusion 
$C(\boldsymbol{\beta}) \hookrightarrow R_{\boldsymbol{\beta}}$ 
yields a surjective ring homomorphism 
$ \mathrm{im}(\tau) \otimes_{C(\boldsymbol{\beta})} \mathrm{im}(\tilde{\tau}) \rightarrow
 \mathrm{im}(\tau) \otimes_{R_{\boldsymbol{\beta}}} \mathrm{im}(\tilde{\tau})$
and so the image of the above ideal under this map is a proper ideal of 
$\mathrm{im}(\tau) \otimes_{R_{\boldsymbol{\beta}}} \mathrm{im}(\tilde{\tau})$
and it contains $Q$. We conclude that $Q$ is a proper ideal of $U$ and 
since $U/Q \cong C(\boldsymbol{\beta})$, it is a maximal ideal. 
Thus $J:=\Phi(\mathrm{im}(\tau) \otimes_{C(\boldsymbol{\beta})} Q)$ 
is a relatively maximal differential ideal of 
$\mathrm{im}(\tau) \otimes_{R_{\boldsymbol{\beta}}} \mathrm{im}(\tilde{\tau})$. 
It is easy to see that the image of $Z$ in 
$\mathrm{im}(\tau) \otimes_{R_{\boldsymbol{\beta}}} \mathrm{im}(\tilde{\tau}) /J $ 
under the quotient map is $1_n\otimes 1_n$. So the differential isomorphisms 
$\varphi: \mathrm{im}(\tau) \rightarrow \mathrm{im}(\tau) \otimes_{R_{\boldsymbol{\beta}}}
 \mathrm{im}(\tilde{\tau}) / J$ 
and $\tilde{\varphi}: \mathrm{im}(\tilde{\tau}) \rightarrow \mathrm{im}(\tau) \otimes_{R_{\boldsymbol{\beta}}}
 \mathrm{im}(\tilde{\tau}) / J$ 
satisfy $\varphi (\tau(X))= \tilde{\varphi}(\tilde{\tau}(X))$ and consequently 
$\varphi(\mathrm{im}(\tau))=\tilde{\varphi}(\mathrm{im}(\tilde{\tau}))$.
We conclude that 
\begin{equation*}
\Psi: \mathrm{im}(\tau) \rightarrow   \mathrm{im}(\tilde{\tau}), \ \tau(X) \mapsto  \tilde{\tau}(X)
\end{equation*}
is a differential isomorphism.
\end{proof}
\begin{corollary}\label{corollary Kernels are equal}
 For $A \in R_{1}^{n \times n}$ let a derivation on $R_{\boldsymbol{\beta}}[\mathrm{GL}_n]$ be defined by 
 $\partial(X)=A X$. Then for the Taylor maps $\tau$ and $\tilde{\tau}$  
 of Lemma~\ref{lemma_taylor_emb_transcendental} and~\ref{lemma_taylor_emb_semi_transcendental} 
 with initial values $X(P_{1})=X(\tilde{P}_{1})= 1_n$, we have 
 \begin{equation*}
  \mathrm{kern}(\tau) = \mathrm{kern}(\tilde{\tau}).
 \end{equation*}
\end{corollary}
\begin{proof}
 From Lemma~\ref{Lemma_isomorphic} we obtain a differential isomorphism 
 $\Psi: \mathrm{im}(\tau) \rightarrow \mathrm{im}(\tilde{\tau})$ which satisfies $\Psi \circ \tau = \tilde{\tau}$,
 since $\Psi(\tau(X))= \tilde{\tau}(X)$. From this it follows that 
 $\mathrm{kern}(\tau) \subset \mathrm{kern}(\tilde{\tau})$ and 
 $\mathrm{kern}(\tilde{\tau}) \subset \mathrm{kern}(\tau)$.
\end{proof}

\begin{lemma}\label{Lemma Taylor map with algebraic closure}
Let $\tau$ be the Taylor map of Lemma~\ref{lemma_taylor_emb_transcendental} 
for differential equation $\partial(\textit{\textbf{y}})=A\textit{\textbf{y}}$ 
with initial values $X(P_1)=1_n$ and let 
$\overline{C(\boldsymbol{\beta})}$ be an algebraic closure of $C(\boldsymbol{\beta})$.
Then there exists a matrix $M \in \mathrm{GL}_n(\overline{C(\boldsymbol{\beta})})$ such that
the kernel of the Taylor map
\begin{equation*}
\tau': R_1[\mathrm{GL}_n] 
 \rightarrow \overline{C(\boldsymbol{\beta})}[[T]], \ X \mapsto \tau(X)M
\end{equation*}
is a relatively maximal differential ideal of $R_1[\mathrm{GL}_n]$, that is, $\mathrm{im}(\tau')$
is a Picard Vessiot ring with constants $C$.
\end{lemma}
\begin{proof}
For a relatively maximal differential ideal $I$ of $R_1[\mathrm{GL}_n]$  
the ring $R_1[\mathrm{GL}_n]/I $ is a Picard-Vessiot ring with constants $C$.
The tensor product of $R_1[\mathrm{GL}_n]/I$ with an arbitrary $C$-algebra $\mathcal{A}$ 
produces a new differential ring $R_1[\mathrm{GL}_n]/I \otimes_C \mathcal{A}$
where the derivation on the second factor is trivial.
By \cite[Lemma 10.7]{Maurischat} this differential ring is a  
$R_1 \otimes_C \mathcal{A}$-simple differential ring, 
that is, it is a Picard-Vessiot ring over $R_1 \otimes_C \mathcal{A}$
for the differential equation $\partial(\textit{\textbf{y}})=A \textit{\textbf{y}}$ 
with constants $C \otimes_C \mathcal{A} \cong \mathcal{A}$. 
We conclude that 
\begin{equation*}
S_1 :=R_1[\mathrm{GL}_n]/I \otimes_C \overline{C(\boldsymbol{\beta})} 
\end{equation*}
is a Picard-Vessiot
ring with an algebraically closed field of constants $\overline{C(\boldsymbol{\beta})}$. 
By Lemma~\ref{lemma_taylor_emb_transcendental} there is a Taylor map 
\begin{equation*}
 \tau : R_{\boldsymbol{\beta}}[\mathrm{GL}_n] \rightarrow C(\boldsymbol{\beta})[[T]]
\end{equation*}
such that $\mathrm{im}(\tau)$ is a Picard-Vessiot ring with field of constants $C(\boldsymbol{\beta})$.
The same argumentation as above yields then a Picard-Vessiot ring 
\begin{equation*}
S_2:=\mathrm{im}(\tau) \otimes_{C(\boldsymbol{\beta})} \overline{C(\boldsymbol{\beta})}
\end{equation*}
with constants $\overline{C(\boldsymbol{\beta})}$. 
Now by \cite[Proposition 1.20.2]{P/S} the two differential rings $S_1$ and $S_2$
for the differential equation $\partial(\textit{\textbf{y}})=A \textit{\textbf{y}}$ are differentially 
isomorphic over the constants $\overline{C(\boldsymbol{\beta})}$. 
More precisely, if $Y \otimes 1_n$ is 
the fundamental matrix of $S_1$, then the isomorphism is given by 
\begin{equation*}
 \varphi: S_1 \rightarrow S_2 , \ Y \otimes 1_n \mapsto \tau(X) \otimes M
\end{equation*}
where $M \in \mathrm{GL}_n(\overline{C(\boldsymbol{\beta})})$.  
We conclude that for the initial values 
$M \in \mathrm{GL}_n(\overline{C(\boldsymbol{\beta})})$ we obtain a new Taylor map 
\begin{equation*}
\tau' : R_1 [\mathrm{GL}_n] \otimes_C \overline{C(\boldsymbol{\beta})} \rightarrow 
\overline{C(\boldsymbol{\beta})}[[T]], \ X \otimes 1_n \mapsto \tau(X)M 
\end{equation*}
with the property that
$\mathrm{kern}(\tau')= I \otimes_C \overline{C(\boldsymbol{\beta})}$. Since $I$ 
is a relatively maximal differential ideal of $R_1[\mathrm{GL}_n]$, 
the image of $R_1[\mathrm{GL}_n]\otimes 1$ under $\tau'$ is a Picard-Vessiot ring
with constants $C$.
\end{proof}

\begin{theorem}
\label{lemma_taylor_emb_transcendental_without beta}
Let $\tau$ be the Taylor map of Lemma~\ref{lemma_taylor_emb_transcendental} with initial values $X(P_{1})=1_n$ 
for a matrix differential equation
$\partial(\textit{\textbf{y}})=A(\textit{\textbf{t}})\textit{\textbf{y}}$ with $A \in R_1^{n \times n}$ 
and denote by $\tau_1$ the restriction
of $\tau$ to $R_1 [\mathrm{GL}_n]$.
Then the kernel of 
 \begin{equation*}
\tau_1 : R_1 [\mathrm{GL}_n] \rightarrow C(\boldsymbol{\beta})[[T]]
 \end{equation*}
 is a relatively maximal differential ideal.
\end{theorem}
\begin{proof} 
By Lemma~\ref{Lemma Taylor map with algebraic closure} there is a matrix 
$M \in \mathrm{GL}_n(\overline{C(\boldsymbol{\beta})})$ such that the kernel of the Taylor map 
\begin{equation*}
\tau': R_1[\mathrm{GL}_n] 
 \rightarrow \overline{C(\boldsymbol{\beta})}[[T]], \ X \mapsto \tau(X)M
\end{equation*}
is a relatively maximal differential ideal of $R_1[\mathrm{GL}_n]$.
The matrix $M \in \mathrm{GL}_n(\overline{C(\boldsymbol{\beta})})$ depends only on finitely many $\beta_{ij}$'s. 
This means that for some $r \in \mathbb{N}$ there is a subring 
\begin{equation*}
C[\beta_{ij}\mid 1\leq i \leq l, \ 0 \leq j \leq r]
\end{equation*}
of $C[\boldsymbol{\beta}]$ and a finite algebraic extension $\mathcal{A}$ of a 
finitely generated localization of this ring such that 
$M \in \mathrm{GL}_n(\mathcal{A})$. For $1 \leq i \leq l$ and $0 \leq j \leq r$ it is possible to choose 
elements $c_{ij}$ of $C$ such that the specialization
\begin{equation*}
\varphi : C[\boldsymbol{\beta}] \rightarrow C[\beta_{ij}\mid j > r]
\end{equation*}
defined by $\beta_{ij} \mapsto  c_{ij}$  for $0 \leq j \leq r$ and  
$\beta_{ij} \mapsto \beta_{ij}$ for $j > r$ 
induces a well defined specialization 
\begin{equation*}
\varphi : \mathcal{A} \rightarrow C[\beta_{ij}\mid j > r] 
\end{equation*}
with the property that $\varphi(M) \in \mathrm{GL}_n(C)$. 
Let $\tilde{\tau}'$ be the Taylor map from $R_{\boldsymbol{\beta}}[\mathrm{GL}_n]$ to 
$C(\boldsymbol{\beta})[[T]]$ of Lemma~\ref{lemma_taylor_emb_semi_transcendental} where 
we choose in this situation the ideal $\tilde{P}_1$ of $R_{\boldsymbol{\beta}}$ as   
\begin{equation*}
 \tilde{P}_1 = \langle \partial^j(t_i) -c_{ij}, \ \partial^k(t_i) - \beta_{ik} \mid 
  \ 0 \leq j \leq  r , \ k \in \mathbb{N}, \ k > r \rangle 
\end{equation*}
and the initial values $X(\tilde{P}_1)= \varphi(M)$. Obviously we have by construction  
\begin{equation*}
\tilde{\tau}' (X)= \varphi(\tau(X))\varphi(M).
\end{equation*}
Every multivariate polynomial $f(X,\boldsymbol{t}) \in \mathrm{kern}(\tau')$ satisfies
\begin{equation*}
f(\tilde{\tau}'(X) , \tilde{\tau}'(\boldsymbol{t})) = 
f(\varphi(\tau(X))\varphi(M), \varphi(\tau(\boldsymbol{t})))= 
\varphi (f(\tau(X)M, \tau(\boldsymbol{t})))=0
\end{equation*}
and so $\mathrm{kern}(\tau') \subset \mathrm{kern}(\tilde{\tau}')$.
Since $\mathrm{kern}(\tau') \otimes_C C(\boldsymbol{\beta})$ 
is a relatively maximal differential ideal, 
we conclude that $\mathrm{kern}(\tilde{\tau}') =\mathrm{kern}(\tau') \otimes_C C(\boldsymbol{\beta})$.\\
Lemma~\ref{lemma_taylor_emb_semi_transcendental} provides us with a Taylor map
\begin{equation*}
\tilde{\tau}: R_{\boldsymbol{\beta}}[\mathrm{GL}_n] \rightarrow C(\boldsymbol{\beta})[[T]]  
\end{equation*}
where we choose the ideal $\tilde{P}_1$ of $R_{\boldsymbol{\beta}}[\mathrm{GL}_n]$ as above, but this time 
the initial values $X(\tilde{P}_1)= 1_n$.
The fundamental matrices of $\mathrm{im}(\tilde{\tau}')$ and $\mathrm{im}(\tilde{\tau})$ 
satisfy obviously the relation
\begin{equation*}
\tilde{\tau}'(X)= \tilde{\tau}(X) \varphi(M)
\end{equation*} 
and so, since $\mathrm{kern}(\tau')$ is an ideal of $R_1[\mathrm{GL}_n]$ and $\varphi(M) \in \mathrm{GL}_n(C)$, 
it follows from $\mathrm{kern}(\tilde{\tau}') =\mathrm{kern}(\tau') \otimes_C C(\boldsymbol{\beta})$ 
that $\mathrm{kern}(\tilde{\tau})= I \otimes_C C(\boldsymbol{\beta})$ 
where $I $ is a relatively maximal differential ideal of $ R_1 [\mathrm{GL}_n]$. 
Finally, we apply Corollary~\ref{corollary Kernels are equal} to $\tilde{\tau}$ and to the 
Taylor map $\tau$ of Lemma~\ref{lemma_taylor_emb_transcendental} with initial values $X(P_1)=1_n$.
The restriction of $\tau$ to $R_1[\mathrm{GL}_n]$ is then the desired Taylor map.
\end{proof}

 {\bf(b)} We construct now in a second part a Taylor map for $R_2[\mathrm{GL}_n]$. 
 Since $R_2$ is a localization 
 of the polynomial ring $C[z]$ by a finitely generated multiplicative subset,
 there are  
 infinitely many $c \in C$ such that 
 \begin{equation*}
  P_2 := \langle z-c \rangle
 \end{equation*}
 is a maximal ideal in $R_2$, that is, $c$ is not a zero 
 of any generator of the multiplicative subset.  
 We choose now such a $P_2$.
 Similar as in the previous constructions of Taylor maps, for an element $r \in R_2$,
 we denote by $r(P_2)$ the image of $r$ under the quotient map from $R_2$ to $R_2/P_2$.
\begin{lemma}\label{lemma_taylor_emb_ratfunc}
 Let a derivation on $R_2[\mathrm{GL}_n]$ be defined 
 by $\partial(\textbf{y})=A \textbf{y}$ with $A \in R_2^{n \times n}$ and let  
 $P_2$ be a maximal ideal of $R_2$ as above.
 Then for arbitrary initial values $X(P_2) \in \mathrm{GL}_n(C)$ the kernel of the 
 differential ring homomorphism 
 \begin{equation*}
  \tau_2 : R_2[\mathrm{GL}_n] \rightarrow C[[T]], \
  r \mapsto \sum_{k \in \mathbb{N}_0 } 1/k! \ \partial^{k}(r)(P_2) T^k, \
  X_{ij} \mapsto \sum_{k \in \mathbb{N}_0 } 1/k! \ \partial^{k}(X_{ij})(P_2) T^k
 \end{equation*}
 is a relatively maximal differential ideal of $R_2[\mathrm{GL}_n]$. 
\end{lemma}
\begin{proof}
The construction of the Taylor map works just as in Lemma~\ref{lemma_taylor_emb_transcendental} but now for the 
matrix $A \in R_2^{n \times n}$ and the ideal $P_2$.
One shows with same arguments as in Lemma~\ref{lemma_taylor_emb_transcendental} 
that $\mathrm{kern}(\tau_2)$ is a relatively maximal
differential ideal of $R_2[\mathrm{GL}_n]$.  
\end{proof}

\section{The specialization bound}\label{chap4}
In this section we prove the specialization bound.
We use the results of the previous section to prove that there exists a relatively maximal differential ideal of 
$R_1[\mathrm{GL}_n]$ such that its image under a $R_1$-specialization is contained in 
a relatively maximal differential ideal of $R_2[\mathrm{GL}_n]$. 
\begin{proposition}\label{Prop_specialization}
 Let $\partial(\textbf{y})=A\textbf{y}$ be a matrix differential equation 
 with $A \in R_{1}^{n \times n}$ and let $\sigma: R_1 \rightarrow R_2$ 
be a surjective $R_1$-specialization.  
\begin{enumerate}
 \item Then there exist Taylor maps
\begin{equation*}  
 \quad \quad \quad \tau_1: R_1[\mathrm{GL}_n]  \rightarrow \mathrm{im}(\tau_1) \subset C(\boldsymbol{\beta})[[T]] \ and \
 \tau_2: R_2[\mathrm{GL}_n]  \rightarrow \mathrm{im}(\tau_2) \subset C[[T]]
\end{equation*}
such that $\mathrm{im}(\tau_1)$ is a Picard-Vessiot ring for $\partial(\textbf{y})=\tau_1 (A)\textbf{y}$
 with constants $C$ 
and $\mathrm{im}(\tau_2)$ is a Picard-Vessiot ring for $\partial(\textbf{y})=\tau_2(\sigma(A))\textbf{y}$.
\item In the situation of (1). There exists a surjective differential
homomorphism $\hat{\sigma}$ such that the following diagram commutes:
 \[
 \begin{xy}
 \xymatrix{
 R_1 [\mathrm{GL}_n] \ar[d]^{\tau_1} \ar[r]^{\sigma} & R_2 [\mathrm{GL}_n] \ar[d]^{\tau_2} \\
 \mathrm{im}(\tau_1) \ar@{-->}[r]^{\hat{\sigma}}& \mathrm{im}(\tau_2)
 }
 \end{xy}
 \]
 \end{enumerate}
\end{proposition}
\begin{proof} 
For the ideal $P_1 \subset R_{ \boldsymbol{\beta}}$ of Section~\ref{chap3} and initial values $X(P_{1}) = 1_n$ 
we obtain from Theorem~\ref{lemma_taylor_emb_transcendental_without beta} a Taylor map
\begin{equation*}
\tau_1: R_1[\mathrm{GL}_n] \rightarrow C(\boldsymbol{\beta})[[T]]
\end{equation*}
such that $\mathrm{im}(\tau_1)$ is a Picard-Vessiot ring for $\partial(\boldsymbol{y})=\tau_1
(A)\boldsymbol{y}$ with constants $C$.
Further if we choose a maximal ideal $P_2=\langle z-c \rangle $ of $R_2$ as in 
Section~\ref{chap3}, that is, 
$c$ is not a zero of any element in the multiplicative subset generating the localization $R_2$, and apply 
Lemma~\ref{lemma_taylor_emb_ratfunc} to 
$\partial(\textit{\textbf{y}})=\sigma(A)\textit{\textbf{y}}$
with initial values $X(P_{2}) = 1_n$, we get a second Taylor map
\begin{equation*}
 \tau_2: R_2[\mathrm{GL}_n] \rightarrow C[[T]] 
\end{equation*}
such that $\mathrm{im}(\tau_2)$ is a Picard-Vessiot ring for 
$\partial(\boldsymbol{y})=\tau_2(\sigma(A))\boldsymbol{y}$. This proves the first part. \\ 
Let the $R_1$-specialization be given by  
\begin{equation*}
\sigma(\boldsymbol{t})= \boldsymbol{r}=(r_1(z),\dots, r_l(z)) \in C[z]^l 
\end{equation*}
and let $c_{ij}:=(\partial^j (r_i))(c)$ be the evaluation of the polynomial 
$(\partial^j (r_i))(z)$ at $c$.  Since $R_1$ is a localization of $C \{ \boldsymbol{t} \}$
by a finitely generated multiplicative subset, the coefficients of the power series 
in $\mathrm{im}(\tau_1)$ are elements of a localization $\mathcal{A}$ of $C[\boldsymbol{\beta}]$
by a finitely generated multiplicative subset of $C[\boldsymbol{\beta}]$.
We extend $\sigma$ to a specialization
\begin{equation*}
 \sigma: R_1 \otimes_C \mathcal{A} \rightarrow R_2 
\end{equation*}
by $\sigma: 1 \otimes \beta_{ij} \mapsto c_{ij}$. Note that the elements in the multiplicative 
subset for the localization 
$\mathcal{A}$ do not vanish under $\sigma$, because $c$ is not a zero of the elements in the multiplicative subset
generating the localization $R_2$. We denote in the following the smaller ideal 
$P_1 \cap (R_1 \otimes_C \mathcal{A} )$ also by $P_1$. 
Since for $r(z) \in R_2$ the polynomial $r(z)-r(c)$ has obviously a zero at
$c$ and since $r(z)-\tilde{c}$ is the trivial polynomial, if $r(z)= \tilde{c} \in C$ is a constant,  we obtain that 
\begin{equation*}
\sigma(t_{ij} \otimes 1 - 1 \otimes \beta_{ij})= \partial^j(r_i)- c_{ij} \in P_2 = \langle z-c \rangle 
\end{equation*}
and so $\sigma(P_{1}) \subset P_2$. From the surjectivity of $\sigma$ which means that there is at least one 
$r_i(z)$ of degree grater than one, it follows that $\sigma(P_{1}) = P_2$. 
For the initial values it trivially holds $\sigma(X)(P_{1})=X(P_2)$.\\
We define now the map $\hat{\sigma}$ by
\begin{equation*}
 \hat{\sigma}: \mathrm{im}(\tau_1) \rightarrow \mathrm{im}(\tau_2), \ 
 \sum_{k \in \mathbb{N}_0} 1/k! \ \partial^{k}(s)(P_{1})T^k \mapsto 
 \sum_{k \in \mathbb{N}_0} 1/k! \ \partial^{k}(\sigma(s))
 (\sigma(P_{1}))T^k
\end{equation*}
where $s \in R_1[\mathrm{GL}_n]$.
We have to show that $\hat{\sigma}$ is well defined, i.e., for $s_1, s_2 \in R_1[\mathrm{GL}_n]$ with 
$\partial^{k}(s_1)(P_{1})=\partial^{k}(s_2)(P_{1})$ for all $k \in \mathbb{N}_0$ it has to follow that 
$\partial^{k}(\sigma(s_1))(P_{2})=\partial^{k}(\sigma(s_2))(P_{2})$ for all $k \in \mathbb{N}_0$. Since 
$\sigma(P_{1})=P_2$, the fundamental theorem of homomorphisms yields that 
there exists a homomorphism $\varphi$ such that
the following diagram commutes
 \[
\begin{xy}
\xymatrix{
R_1[\mathrm{GL}_n] \otimes \mathcal{A}  \ar[d]^{\pi_1} \ar[r]^{\sigma} & R_2[\mathrm{GL}_n] \ar[d]^{\pi_2} \\
(R_1[\mathrm{GL}_n] \otimes \mathcal{A})/P_{1} \ar@{-->}[r]^{\varphi}& R_2[\mathrm{GL}_n]/P_2
}
\end{xy}
\]
We conclude that $\hat{\sigma}$ is well defined. Since $\hat{\sigma}$ is induced by $\tau_2 \circ \sigma$, it is 
a differential homomorphism. From the definition of $\hat{\sigma}$ it follows that the diagram commutes.
\end{proof}
\begin{corollary}\label{Cor_spec_ideals}
 Let $\partial(\textbf{y})=A\textbf{y}$ be a matrix differential equation over $F_1$ 
 with $A \in R_{1}^{n \times n}$ and  let $\sigma: R_1 \rightarrow R_2$ be a surjective $R_1$-specialization.  
One extends $\sigma$ to a surjective differential homomorphism from $R_1[\mathrm{GL}_n]$ to $R_2[\mathrm{GL}_n]$ 
where the derivation on $R_2[\mathrm{GL}_n]$ is defined by $\partial(X)=\sigma(A )(X)$. 
Then there exist relatively maximal differential ideals 
$I_1 \subset R_1[\mathrm{GL}_n]$ and $I_2 \subset R_2[\mathrm{GL}_n]$ with the property that 
\begin{equation*}
 \sigma(I_1) \subseteq I_2
\end{equation*}
and $\sigma$ can be extended to a surjective differential homomorphism
\begin{equation*}
\sigma : R_1[\mathrm{GL}_n]/I_1 \rightarrow R_2[\mathrm{GL}_n]/ I_2.
\end{equation*}
\end{corollary}
\begin{proof}
Using the notation of Proposition~\ref{Prop_specialization}, 
take $I_1 := \mathrm{kern}(\tau_1)$ and $I_2 := \mathrm{kern}(\tau_2)$.
Then $I_1$ is a relatively maximal differential ideal of $R_1[\mathrm{GL}_n]$ by 
Theorem~\ref{lemma_taylor_emb_transcendental_without beta} and 
$I_2$ is a relatively maximal differential ideal of $R_2[\mathrm{GL}_n]$ by 
Lemma~\ref{lemma_taylor_emb_ratfunc}. Since the diagram in  
Proposition~\ref{Prop_specialization} commutes, it follows for 
$f \in \mathrm{kern}(\tau_1)$ that $\sigma(f) \in \mathrm{kern}(\tau_2)$ and so $\sigma(I_1) \subseteq I_2$. 
The second statement follows from the fundamental theorem of homomorphisms.
\end{proof}

\begin{theorem}\label{specialization_bound}
 Let $\partial(\textbf{y})=A\textbf{y}$ be a differential equation over $F_1$ with 
 $A \in R_{1}^{n \times n}$ and let $\sigma: R_1 \rightarrow R_2$ 
 be a surjective $R_1$-specialization.  
 Then the differential Galois group $G_2(C)$ of 
 the specialized equation $\partial(\textbf{y})=\sigma(A)\textbf{y}$ over $F_2$ 
 is a subgroup of the differential Galois group 
 $G_1(C)$ of the original equation $\partial(\textbf{y})=A\textbf{y}$ over $F_1$.
\end{theorem}
\begin{proof}
We define a differential structure on the rings
$R_1[\mathrm{GL}_n]$ and $R_2 [\mathrm{GL}_n]$ by 
the matrix differential equations $\partial(\boldsymbol{y})=A\boldsymbol{y}$ and 
$\partial(\boldsymbol{y})=\sigma (A) \boldsymbol{y}$.  
Corollary~\ref{Cor_spec_ideals} then yields that there exist a relatively maximal 
differential ideal $I_1$ in $R_1[\mathrm{GL}_n]$ 
and a relatively maximal differential ideal $I_2$ in $R_2[\mathrm{GL}_n]$ such that 
$\sigma(I_1) \subseteq I_2$.  
The rings $S_k = R_k [\mathrm{GL}_n]/I_k$ are then Picard-Vessiot rings 
and the specialization $\sigma$ extends to a surjective specialization
\begin{equation*}
\sigma : S_1 \rightarrow S_2.
\end{equation*}
Let $Z:=(Z_{ij}) \in \mathrm{GL}_n(S_1)$  be the matrix whose entries $Z_{ij}$ are the images 
of the variables $X_{ij}$  in $S_1$ under the quotient map 
$R_1 [\mathrm{GL}_n] \rightarrow R_1 [\mathrm{GL}_n]/I_1$.
The specialization $\sigma$ then maps the fundamental matrix $Z$ to the 
image of $X$ under the quotient map $R_2 [\mathrm{GL}_n] \rightarrow R_2 [\mathrm{GL}_n]/I_2$ 
which is a fundamental solution matrix of $S_2$. We write $\sigma(Z)$ for it. 
For $k=1,2$ we consider the rings
\begin{equation*}
S_k[Y, \mathrm{det}(Y)^{-1}] = S_k[X, \mathrm{det}(X)^{-1}] \supset
R_k[X, \mathrm{det}(X)^{-1}]
\end{equation*}
and 
\begin{equation*}
 C[Y, \mathrm{det}(Y)^{-1}] \subset  S_k[Y, \mathrm{det}(Y)^{-1}].  
\end{equation*}
The relations between the new variables $X:=(X_{ij})$ and $Y:=(Y_{ij})$ in these rings are defined by  
\begin{equation*} 
 X= Z \cdot Y   \ \mathrm{and} \   X= \sigma(Z) \cdot Y 
\end{equation*}
respectively.
We specify the derivations and Galois actions on these rings for $k=1$. 
The derivations and Galois actions on the rings with $k=2$ are defined analogously 
where one uses the matrix $\sigma (A)$ and the fundamental matrix $\sigma(Z)$ instead.
The derivation on $R_1[X, \mathrm{det}(X)^{-1}]$ and $S_1[X, \mathrm{det}(X)^{-1}]$ is 
defined by $\partial(X)=A X$ and by the derivation on $S_1$. 
Computing $\partial(Z Y)$ shows that the derivation on $Y$ is trivial. 
Therefore, $\partial$ acts also trivial on $C[Y, \mathrm{det}(Y)^{-1}]$ and the derivation 
on $S_1[Y, \mathrm{det}(Y)^{-1}]$ is defined by the derivation on $S_1$. The action of the Galois group 
$\mathrm{Gal}(S_1/R_1)$ on the above rings is induced by the action on $S_1$. More precisely, 
$\mathrm{Gal}(S_1/R_1)$ acts trivially on $R_1[X, \mathrm{det}(X)^{-1}]$ and, since for 
$\gamma_1 \in \mathrm{Gal}(S_1 /R_1)$ the Galois action is given by $\gamma_1(Z)= Z M_1$ with 
$M_1 \in \mathrm{GL}_n(C)$, 
it acts via  $\gamma_1(Y)=M_1^{-1} \cdot Y$ on the rings $S_1[Y, \mathrm{det}(Y)^{-1}]$ 
and $C[Y, \mathrm{det}(Y)^{-1}]$. \\
Lemma~\ref{lemma_bijection2} yields a bijection between the set of differential ideals of 
$R_k [X, \mathrm{det}(X)^{-1}]$ which satisfy Condition~\ref{condition} and the set of 
$\mathrm{Gal}(S_k/R_k)$-invariant differential ideals of
\begin{equation*}
S_k [X, \mathrm{det}(X)^{-1}]=S_k [Y, \mathrm{det}(Y)^{-1}]
\end{equation*}
which also satisfy Condition~\ref{condition}. 
Thus, since $I_1$ is a relatively maximal differential ideal, the ideal
\begin{equation*}
 \tilde{I}_1= \{ f(Z  \cdot Y ) \mid f \in (I_1) \} \trianglelefteq S_1 [Y, \mathrm{det}(Y)^{-1}]
\end{equation*}
is a maximal $\mathrm{Gal}(S_1/R_1)$-invariant differential ideal which satisfies 
Condition~\ref{condition}. Analogously we obtain that 
\begin{equation*}
 \tilde{I}_2= \{ f(\sigma(Z)  \cdot Y ) \mid f \in (I_2) \} \trianglelefteq S_2 [Y, \mathrm{det}(Y)^{-1}].
\end{equation*}
is a maximal $\mathrm{Gal}(S_2/R_2)$-invariant differential ideal satisfying 
Condition~\ref{condition}.
By Lemma~\ref{lemma_bijection1} there is a bijection 
between the differential ideals of $S_k [Y, \mathrm{det}(Y)^{-1}]$ which satisfy 
Condition~\ref{condition} and the ideals of $C [Y, \mathrm{det}(Y)^{-1}]$. 
Hence, the ideal 
\begin{equation*}
Q_k = \tilde{I}_{k } \cap C [Y, \mathrm{det}(Y)^{-1}]
\end{equation*}
is a maximal $\mathrm{Gal}(S_k/R_k)$-invariant ideal of $C [Y, \mathrm{det}(Y)^{-1}]$. 
By its maximality $Q_k$ defines the 
differential Galois group $G_k$ of $S_k$. For a detailed explanation we refer to the proof of Theorem 1.28 in \cite{P/S}.
The specialization $\sigma$ extends trivially to a surjective specialization 
\begin{equation*}
\sigma: S_1 [Y, \mathrm{det}(Y)^{-1}] \rightarrow S_2 [Y, \mathrm{det}(Y)^{-1}], \ Y_{ij} \mapsto Y_{ij}.
\end{equation*}
With $\sigma(I_1) \subseteq I_2$ we conclude that 
$\sigma(\tilde{I}_1) \subseteq \tilde{I}_2$ and  
since $\sigma$ acts trivially on the ring $C [Y, \mathrm{det}(Y)^{-1}]$, we obtain that 
$Q_1 \subseteq Q_2$ and so $G_1 \supseteq G_2$.
\end{proof}
 \section{Parameter equations for connected linear algebraic groups}\label{chap5}
 Let $G$ be a connected linear algebraic group defined over $C$.
 In a first part of this section we use the specialization bound to prove that there exists 
a linear parameter differential equation over $C\langle t_1 \rangle$ such that its differential
Galois group is $G$. \\
For a successful application of the specialization bound we need a differential equation over $F_2$ with group $G$ 
to which we can specialize a suitable equation over $C\langle t_1 \rangle$.
In \cite{MitschiSinger} C.~Mitschi and M.~Singer proved that every connected linear algebraic group 
can be realized as a differential 
Galois group over a differential field which has $C$ as its field of constants and is  
of finite, non-zero transcendence degree over $C$.   
For the differential field $F_2$ the following proposition is a special case of 
\cite[Theorem 1.1]{MitschiSinger}. 
 \begin{proposition} \label{M/S_connected_groups}
  Let $G$ be a connected linear algebraic group defined over $C$. Then 
  there exists a Picard-Vessiot extension of $F_2$ with differential Galois group $G$.
 \end{proposition}
We can prove now that every connected linear algebraic group occurs as a differential Galois group over $C\langle t_1 \rangle$.  
\begin{theorem}\label{Seiss_connected_groups}
 Let $G$ be a connected linear algebraic group defined over $C$. Then 
  there exists a linear parameter differential equation
  \begin{equation*}
   L(y,t_1)=0 
  \end{equation*}
over $C \langle t_1 \rangle$ with differential Galois group $G$.
\end{theorem}
\begin{proof}
Proposition~\ref{M/S_connected_groups} yields that there exists a matrix 
$A(z) \in F_2^{n \times n}$ such that the differential Galois group of 
\begin{equation*}
 \partial(\boldsymbol{y}) =  A(z)  \boldsymbol{y}
\end{equation*} 
is $G(C)$. Since the cohomological dimension of $F_2$ is 
at most one, we may assume without loss of generality that $A(z)$ is an element of the 
Lie algebra $\mathfrak{g} ( F_2 )$ of $G$. 
If we now substitute the indeterminate $z$ in the entries of $A(z)$ by the differential indeterminate $t_1$, 
we get a matrix $A(t_1) $ 
whose entries lie in $C \langle t_1 \rangle$ and a new matrix differential equation 
\begin{equation*}
\partial(\boldsymbol{y}) =  A(t_1)  \boldsymbol{y}
\end{equation*}
over $C \langle t_1 \rangle$. We determine the differential Galois group $H(C)$ of 
$\partial(\boldsymbol{y}) =  A(t_1)  \boldsymbol{y}$. 
By construction we have that $A(t_1) \in \mathfrak{g} ( C \langle t_1 \rangle )$ and so 
$H(C) \subseteq G(C)$ by Proposition~\ref{Prop_upper}. We choose now a finitely generated multiplicative subset
of $C \{ t_1 \}$ which contains all denominators 
appearing in the entries of $A(t_1)$. Then the localization $R_1 $ of $C \{ t_1 \}$ by this set satisfies
$A(t_1) \in \mathfrak{g} ( R_1 )$ and by construction the map 
\begin{equation*}
\sigma: R_1 \rightarrow R_2 , \ t_1 \mapsto z
\end{equation*}
is surjective $R_1$-specialization with $\sigma(A(t_1))=A(z) \in R_2^{n \times n}$. 
We apply now Theorem~\ref{specialization_bound} to $\partial(\boldsymbol{y}) =  A(t_1)  \boldsymbol{y}$, $R_1$ 
and $\sigma$. 
By construction the differential Galois group of $\partial(\boldsymbol{y}) =  A(\sigma(t_1))  \boldsymbol{y} $ over $F_2$ is $G(C)$. 
Hence, we get from Theorem~\ref{specialization_bound} that $G(C) \subseteq H(C)$ and so we deduce with $H(C) \subseteq G(C)$ that $H(C)=G(C)$. 
Finally we apply the Cyclic Vector Theorem (see for instance \cite{KovacicCyclicVector}) and obtain a linear parameter equation $L(y,t_1)=0$ with 
differential Galois group $G$.
\end{proof}
In a second part of this section 
we show that for every connected 
semisimple linear algebraic group $G$ of Lie rank $l$ there exists a parameter differential equation in 
$l$ parameters with the property 
that we can specialize the parameters such that we get 
all differential equations with group $G$ of a specific type over $F_2$.\\ 
We begin to recall some structure theory of a semisimple linear algebraic group $G$. 
Let $\Phi$ be the root system of $G$ and denote by $T$ a maximal torus
of $G$. Then from the adjoint action of $T$ on the Lie algebra $\mathfrak{g}$, we 
obtain a root space decomposition 
\begin{equation*}
  \mathfrak{g}(C) = \mathfrak{h}(C) \oplus \bigoplus_{\alpha \in \Phi} \mathfrak{g}_{\alpha}(C)
\end{equation*}
where $\mathfrak{h}(C) = \mathrm{Lie}(T)$ is a Cartan subalgebra and for 
$\alpha \in \Phi$ we denote by $\mathfrak{g}_{\alpha}(C)$ the one dimensional root space of 
$\mathfrak{g}(C)$. Let $\Delta$ be a basis of the root system $\Phi$ with simple roots 
$\alpha_i \in \Delta$. According to the above root space decomposition we can choose a Chevalley basis 
\begin{equation*}
 \{ H_{\alpha_i} \mid \alpha_i \in \Delta\} \cup \{ X_{\alpha} \mid \alpha \in \Phi\}
\end{equation*}
of $\mathfrak{g}$ where $\mathfrak{h}(C)= \langle H_{\alpha_1}, \dots, H_{\alpha_l} \rangle$ and $\mathfrak{g}_{\alpha}(C)= \langle X_{\alpha} \rangle$.\\
Using the structure of a connected semisimple linear algebraic group $G$, C.~Mitschi and M.~Singer constructed in \cite{MitschiSinger} specific matrix differential
equations over $F_2$ with group $G$. The following proposition is a modification of \cite[Proposition 3.5]{MitschiSinger}.
\begin{proposition}\label{mod_Mitch_Singer}
 Let $G$ be a connected semisimple linear algebraic group and define $A_0 = \sum_{\alpha_i \in \Delta} 
 (X_{\alpha_i} + X_{- \alpha_i})$. Then there exists $A_1 \in \mathfrak{h}(C)$ such that the differential equation
 $\partial(\textbf{y})=(A_0 + A_1 z)\textbf{y}$ over $F_2$ has $G(C)$ as its differential Galois group.
\end{proposition}
A proof can be found in \cite[Proposition 3.1]{Seiss}.

\begin{theorem}\label{general_result}
 Let $G$ be a connected semisimple linear algebraic group of Lie rank $l$ over $C$. 
 Then there exists a parameter differential equation 
 \begin{equation*}
 L(y,\textbf{t})=0  
 \end{equation*}
 over $F_1$ with differential Galois group $G(C)$ such that we obtain from a suitable specialization 
 of the parameters
 every Picard-Vessiot extension which is defined by an equation of Proposition~\ref{mod_Mitch_Singer}.
\end{theorem}
\begin{proof}
 We define a parameter differential equation $\partial(\textit{\textbf{y}})=A(\textit{\textbf{t}})\textit{\textbf{y}}$ by
 \begin{equation*}
  A(\textit{\textbf{t}})=A_0 + \sum_ {i=1}^l t_i H_{\alpha_i} \in \mathfrak{g}(C\{\textit{\textbf{t}}\}).
 \end{equation*}
Proposition~\ref{Prop_upper} yields that the Galois group $H(C)$ of $A(\textit{\textbf{t}})$ is a subgroup of $G(C)$.
On the other hand there exists $A_1 \in \mathfrak{h}(C)$ by Proposition~\ref{mod_Mitch_Singer} such that
the differential Galois group of 
$\partial(\textit{\textbf{y}})=(A_0 + z A_1)\textit{\textbf{y}}$ is $G(C)$. Let $a_1, \dots, a_l \in C$ such that $A_1 = \sum_{i=1}^l a_i H_{\alpha_i}$.
Then we obtain for the specialization 
\begin{equation*}
 \sigma:C\{t_1, \dots,t_l\} \rightarrow C[z], \ (t_1, \dots,t_l) \mapsto (a_1 z, \dots , a_l z)
\end{equation*}
that the differential Galois group of $\sigma (A(\textit{\textbf{t}}))$ is $G(C)$. By 
Theorem~\ref{specialization_bound} we have that $G(C) \subseteq H(C) \subseteq G(C)$, that is 
$H(C) = G(C)$. 
 We apply now the Cyclic Vector Theorem (see for instance \cite{KovacicCyclicVector}) and obtain a linear parameter differential equation 
 $L(y,\textit{\textbf{t}})$. By construction we can specialize the parameters of $A(\textit{\textbf{t}})$ 
 such that we obtain all equations of Proposition~\ref{mod_Mitch_Singer} and 
 therefore a suitable specialization of $L(y,\textit{\textbf{t}})$ yields every Picard-Vessiot extension defined by such an equation. 
\end{proof}
\section*{PART II\\Parameter differential equations for the classical groups}
 \section{The Transformation Lemma}\label{chap6}
In this and in the following chapters we will prove Theorem~\ref{main_thm_intro} from the introduction. To this purpose 
let $G$ be one of the groups occurring in Theorem~\ref{main_thm_intro} and keep the notations of 
the preceding section.  
With respect to a Cartan decomposition of $\mathfrak{g}$ we consider in the following the 
maximal nilpotent subalgebra $\mathfrak{u}^+= \sum_{\alpha \in \Phi^+} \mathfrak{g}_{\alpha}$ 
(respectively $\mathfrak{u}^-$) defined by all positive (respectively negative) roots of $\Phi$. 
Further, we denote by $\mathfrak{b}^+= \mathfrak{h} + \mathfrak{u}^+$ (respectively $\mathfrak{b}^-$) the maximal solvable 
subalgebra of $\mathfrak{g}$ which contains the maximal nilpotent subalgebra $\mathfrak{u}^+$ 
(respectively $\mathfrak{u}^-$) and the Cartan subalgebra $\mathfrak{h}$. Let $X \in \mathfrak{g}$
and denote by $\mathfrak{s}$ a subspace of $\mathfrak{g}$. Then we call the subset 
$X+\mathfrak{s}$ a plane of $\mathfrak{g}$. We denote by $A_0^+ = \sum_{i=1}^l 
X_{\alpha_i}$ (respective $A_0^- = \sum_{i=1}^l X_{- \alpha_i}$) the sum of all basis 
elements belonging to the positive (respective negative) simple roots. For a root 
$\alpha = \sum_{\alpha_i \in \Delta} n_{\alpha_i}(\alpha) \alpha_i \in \Phi$, where 
$n_{\alpha_i}(\alpha) \in \mathbb{Z}$ are all negative or positive, we denote by 
$\mathrm{ht}(\alpha) = \sum_{\alpha_i \in \Delta} n_{\alpha_i}(\alpha) \in \mathbb{Z}$ 
the height of $\alpha$.

The proof of Theorem~\ref{main_thm_intro} is organized in the following way: In this chapter we show  
that for every group $G$ in Theorem~\ref{main_thm_intro} there are $l$ negative roots $\gamma_i \in \Phi^-$ of specific heights 
such that every matrix of the plane $ A_0^+ + \mathfrak{b}^-$ is gauge equivalent to a matrix of the plane 
\begin{equation*}
A_0^+  + \sum_{i=1}^l \mathfrak{g}_{\gamma_i}.
\end{equation*}
 Afterwards, i.e.~in the Chapters~\ref{chap7}-\ref{chap11}, we prove 
 Theorem~\ref{main_thm_intro} for each group separately.
 We will determine the roots $\gamma_i$ and the explicit shape of the matrix 
 $A_{G}(\textit{\textbf{t}})$, where $A_{G}(\textit{\textbf{t}})$ is the 
 parametrization of the above plane. Finally, we will show that 
  $\partial(\textit{\textbf{y}})= A_{G}(\textit{\textbf{t}})\textit{\textbf{y}}$ 
  is equivalent to the corresponding linear parameter differential equation in 
  Theorem~\ref{main_thm_intro} and has $G(C)$ as differential Galois group over 
  $C \langle \boldsymbol{t} \rangle$. 

Let $X, \ Y \in \mathfrak{g}$. Then we write $[X,Y]$ for the usual bracket product 
and $\mathrm{ad}(X)$ for the endomorphism $\mathrm{ad}(X):Y \mapsto [X,Y]$.
The adjoint action for an element $g \in G \subset \mathrm{GL}_n$ on $\mathfrak{g}$ 
is denoted by $\mathrm{Ad}(g): \mathfrak{g} \rightarrow  \mathfrak{g}, 
X \mapsto gXg^{-1}$. 
For $X \in \mathfrak{g}$ nilpotent, let 
\begin{equation*}
\mathrm{exp}(\mathrm{ad}(X))= \sum_{j\geq 0} \frac{1}{j!} \mathrm{ad}^j (X)
\end{equation*}
be the exponential of $\mathrm{ad}(X)$, which is an automorphism of $\mathfrak{g}$, and for a root $\beta \in \Phi$ 
the exponential map from $\mathfrak{g}_{\beta}$ to the root group $U_{\beta}$ is defined by 
\begin{equation*}
 \mathrm{exp}: \mathfrak{g}_{\beta} \rightarrow U_{\beta}, \ X_{\beta} \mapsto \sum_{j\geq 0} \frac{1}{j!} X_{\beta}^j.
\end{equation*}
For a parameter $\rho$ we denote by $u_{\beta}(\rho)$ the root group element 
$\mathrm{exp}(\rho X_{\beta})$. We have the well known identity 
$\mathrm{exp}(\mathrm{ad}(X_{\beta}))(Y) =\mathrm{Ad}(\mathrm{exp}(X_{\beta}))(Y)$. 
Finally, we write $\mathfrak{g}^{X}$ for the centralizer of $X$ in $\mathfrak{g}$.

By \cite[\S 5.2]{Konst_Betti} there exists a unique element $H_0$ in $\mathfrak{h}$ 
such that $\alpha_i(H_0)=1$ for all $\alpha_i \in \Delta$. 
Therefore, for $\alpha \in \Phi$ we have $\alpha(H_0)=\mathrm{ht}(\alpha)$ 
and $\mathrm{ad} (H_0)(X_{\alpha})= \mathrm{ht}(\alpha)X_{\alpha}$. 
This yields a decomposition of $\mathfrak{g}$ into a direct sum of eigenspaces 
$\mathfrak{g}^{(j)}$ of 
$\mathrm{ad} (H_0)(X_{\alpha})$ with eigenvalues $j \in \mathbb{Z}$.
Clearly for $j\neq 0$ the eigenspace $\mathfrak{g}^{(j)}$ is  
 the direct sum of root spaces  of the same height $j$, i.e. we have 
 \begin{equation*}
 \mathfrak{g}^{(j)} = \sum_{\alpha \in \Phi , \ \mathrm{ht}(\alpha)=j} \mathfrak{g}_{\alpha}.
 \end{equation*}
 A proof of Lemma~\ref{lemma_eigenspace_decomp} below can be found in 
 \cite[page 369, Lemma 9]{KostantLieGrpRep}.
 \begin{lemma}\label{lemma_eigenspace_decomp}
 The maximal nilpotent subalgebras $\mathfrak{u}^+$, $\mathfrak{u}^-$
 and the Cartan algebra $\mathfrak{h}$ are the following direct sums of eigenspaces of $\mathrm{ad} (H_0)(X_{\alpha})$:
 \begin{equation*}
  \mathfrak{u}^+ = \sum_{j>0} \mathfrak{g}^{(j)}, \ \mathfrak{u}^- =\sum_{j<0} \mathfrak{g}^{(j)}, \  \mathfrak{h}= \mathfrak{g}^{(0)}.
 \end{equation*}
 Additionally, for two eigenvalues $i,j \in \mathbb{Z}$ we have the following relation:
  \begin{equation*}
 [\mathfrak{g}^{(i)}, \mathfrak{g}^{(j)} ]\subset \mathfrak{g}^{(i+j)}.
  \end{equation*}
 \end{lemma} 
 Let us consider the ring of polynomials on $\mathfrak{g}$ as a $G$-module in the 
 obvious way. Then by a theorem of Chevalley its ring of invariants is generated by $l$ homogeneous polynomials $u_i$ of degree $\mathrm{deg}(u_i)=m_i +1$. 
 The integers $m_i$ are such that 
 \begin{equation*}
  p_{G}(x)= \prod_{i=1}^l (1+ x^{2m_i +1})
 \end{equation*}
is the Poincare polynomial of $G$ and are called the exponents of $\mathfrak{g}$ (see \cite{Konst_Betti}). The exponents can also be recovered as the 
eigenvalues of a particular element in the Weyl group which is known as a Coxeter-Killing transformation (see \cite{Coxeter}).

Recall that $A_0^- = \sum_{i=1}^l X_{- \alpha_i}$ is the sum of basis elements for all negative simple roots and $A_0^+ = \sum_{i=1}^l X_{ \alpha_i}$
for all positive roots respectively. Then Theorem \ref{thm_basis_Z_i} below characterizes a basis of the centralizer $\mathfrak{g}^{A_0^+}$ in terms 
of exponents and the above eigenspace decomposition. For a proof see 
\cite[Theorem 5]{KostantLieGrpRep}.  
 \begin{theorem}\label{thm_basis_Z_i}
There exists a basis $\{ Z_i \mid i=1, \dots , l \}$ of $\mathfrak{g}^{A_0^+}$ 
such that $Z_i \in \mathfrak{g}^{(m_i)}$ where the integers $m_i$ are the 
exponents of $\mathfrak{g}$.
In particular $\mathfrak{g}^{A_0^+} \subset \mathfrak{b}^+$.
 \end{theorem}
Later we need the exponents $m_i$ of the Lie algebras   
$\mathfrak{sl}_{l+1}$, $\mathfrak{so}_{2l+1}$, $\mathfrak{sp}_{2l}$, $\mathfrak{so}_{2l}$ and $\mathfrak{g}_2$, i.e.~the Lie algebras 
of type  $A_l$, $B_l$, $C_l$, $D_l$ and $G_2$, for the explicit computation 
of the equations in Theorem~\ref{main_thm_intro}. We want to note that it is also possible to \textit{read off} the exponents
from the root system. The following procedure was discovered by A.~Shapiro and R.~Steinberg (see \cite{Steinberg}): For 
$k=1, \dots, \mathrm{ht}(\beta)$, where $\beta \in \Phi^+$ is the maximal root, let 
\begin{equation*}
 c_k = \lvert \{ \alpha \in \Phi^+ \mid \mathrm{ht}(\alpha)=k \} \rvert ,
\end{equation*}
i.e.~$c_k$ is the number of roots $\alpha \in \Phi^+$ such that $\mathrm{ht}(\alpha)=k$. Then $k$ is $c_k -c_{k+1}$ times an exponent of 
$\mathfrak{g}$. A proof of correctness of this empirical method follows from 
\cite[Corollary 8.7]{Konst_Betti}. 
 \begin{lemma}\label{lemma_direct_sum1}
The subalgebra $\mathfrak{b}^+$ is the direct sum of the centralizer $\mathfrak{g}^{A_0^+}$ and the image of $\mathfrak{u}^+$ under 
$\mathrm{ad}(A_0^-)$, that is
 \begin{equation*}
\mathfrak{b}^+= \mathfrak{g}^{A_0^+} + \mathrm{ad}(A_0^-)(\mathfrak{u}^+).
\end{equation*}
 \end{lemma}
 For a proof of Lemma~\ref{lemma_direct_sum1} above we refer to 
 \cite[Lemma 12]{KostantLieGrpRep}.

 We will now determine a basis of the subspace $\mathrm{ad}(A_0^-)(\mathfrak{u}^+)$ (see also \cite{KostantLieGrpRep}, the beginning of the 
 proof of Proposition 19). To this purpose, we number the basis elements 
 $\{ X_{\alpha} \mid \alpha \in \Phi^+ \}$
 of $\mathfrak{u}^+$ into $X_i$ with $i=1,\dots,r$ where $r= \lvert\Phi^+ \rvert$ 
 is the number of positive roots. 
 We denote by $r(i)$ the positive number such that 
 $X_i \in \mathfrak{g}^{(r(i))}$. We can now rearrange the numbering of the basis such that $r(i) \leq r(i+1)$ 
 for all $i=1,\dots ,r-1$. Further we define 
 \[
 W_i := [X_i, A_0^- ]
 \]
 and clearly we have $W_i \in \mathfrak{g}^{(r(i)-1)}$.
 When we interchange the role of the positive and negative roots in Theorem~\ref{thm_basis_Z_i} we obtain $\mathfrak{g}^{A_0^-} \subset \mathfrak{b}^-$ 
 and so $\mathfrak{g}^{A_0^-}\cap \mathfrak{u}^+ = (0)$. We conclude that
 the set $\{ W_i \mid i=1, \dots, r \}$ is a basis of $\mathrm{ad}(A_0^-)(\mathfrak{u}^+)$.
 \begin{lemma}\label{lemma_direct_sum2}
 There exist $l$ roots $\gamma_1,\dots, \gamma_l$ of $ \Phi^+$ such that
 $\mathfrak{b}^+$ is the direct sum 
 \begin{equation*}
 \mathfrak{b}^+= \mathrm{ad}(A_0^-)(\mathfrak{u}^+) + \sum_{i=1}^l \mathfrak{g}_{\gamma_i}   
 \end{equation*}
 and such that $\mathrm{ht}(\gamma_i)=m_i$ for $i=1,\dots,l$ where $m_i $ are the exponents of $\mathfrak{g}$.
\end{lemma}
\begin{proof}
 Theorem~\ref{thm_basis_Z_i} yields basis elements $Z_i$ of $\mathfrak{g}^{A_0^+}$ such that $Z_i \in \mathfrak{g}^{(m_i)}$ for $i=1,\dots, l$. 
 We choose now $i \in \{1, \dots,l\}$ and fix it. Let $m:=m_i$
 and denote by $i_1,\dots,i_s$ all elements of $\{1,\dots,l\}$ such that $m_{i_k} =m$. Then the set 
 \begin{equation*}
  \{Z_{i_1}, \dots,Z_{i_s} \} \cup \{ W_{h} \mid \ W_h \in \mathfrak{g}^{(m)} \} 
 \end{equation*}
is a basis of $\mathfrak{g}^{(m)}$ by Lemma~\ref{lemma_direct_sum1} where the first set is a basis of 
$\mathfrak{g}^{A_0^+} \cap \mathfrak{g}^{(m)}$ 
and the second one of $\mathrm{ad}(A_0^-)(\mathfrak{g}^{(m+1)})$. On the other hand
$\mathfrak{g}^{(m)}$ is the direct sum of all one-dimensional root spaces
$\mathfrak{g}_{\alpha}$ with $\mathrm{ht}(\alpha)=m$.
We denote the set of these roots by $\Lambda^{(m)}$. Then by basis extension there
exist $s$ roots 
$\gamma_{1}, \dots, \gamma_{s} \in \Lambda^{(m)}$ such that the set
\begin{equation*}
 \{ W_{h} \mid \ W_{h} \in \mathfrak{g}^{(m)}   \} \cup \{ X_{\gamma_k} \mid k=1, \dots, s \}
\end{equation*}
is a basis of $\mathfrak{g}^{(m)}$.
If we repeat this procedure for all distinct $m_i$, we obtain a basis for all eigenspaces $\mathfrak{g}^{(m_i)}$.  
A basis for all remaining eigenspaces $\mathfrak{g}^{(j)}$ in the direct sum 
\[
\mathfrak{b}^+ = \sum_{j \geq 0}  \mathfrak{g}^{(j)},
\]
that is for all $\mathfrak{g}^{(j)}$ such that $j\geq 0$ and $j\neq m_i$ for all $ i=1, \dots , l$, 
is simply given by all $W_{h} \in \mathfrak{g}^{(j)} $.
\end{proof}
\begin{definition}
 We call $l$ roots $\gamma_1, \dots, \gamma_l$ of $\Phi^+$, which satisfy the conditions in Lemma~\ref{lemma_direct_sum2}, 
 complementary roots of $\Phi^+$ and the subspace 
 \[
 \mathfrak{r}:=\sum_{i=1}^l \mathfrak{g}_{\gamma_i}
 \]
 a root space complement to $\mathrm{ad}(A_0^-)(\mathfrak{u}^+)$.
\end{definition}
Let $F$ be a differential field with constants $C$ and $\partial(\textit{\textbf{y}})=A \textit{\textbf{y}}$ be a matrix differential equation 
with $A \in F^{n \times n}$. 
\begin{definition}
 The map 
 \begin{equation*}
  \ell \delta: \mathrm{GL}_n(F) \rightarrow F^{n \times n}, g \mapsto \partial(g)g^{-1}
 \end{equation*}
is called the logarithmic derivative.
\end{definition}

\begin{proposition}\label{log_derivate}
 Let $G \subset \mathrm{GL}_n$ be a linear algebraic group. Then the restriction of
 $\ell \delta$ to $G$ maps $G$ to its Lie algebra $\mathfrak{g}$, that is
 \begin{equation*}
  \ell \delta \mid_{G}: G \rightarrow \mathfrak{g}
 \end{equation*} 
\end{proposition}
A proof for Proposition~\ref{log_derivate} can be found in \cite[page 585]{Kov}.

 \begin{lemma}[Transformation Lemma] \label{Prop_tans_lemma}
 Let $A \in A_0^- +  \mathfrak{b}^+(F)$. Then there exists $u \in U^+$ such that  
 \begin{equation*}
  \mathrm{Ad}(u)(A) + \ell \delta(u) \in A_0^- + \mathfrak{r}(F),
 \end{equation*}
 where $U^+$ is the maximal unipotent subgroup of $G$ with Lie algebra $\mathfrak{u}^+$
 and $\mathfrak{r}(F)$ is a root space complement to $\mathrm{ad}(A_0^-)(\mathfrak{u}^+)$. 
\end{lemma}
\begin{proof}
 Note that by Lemma~\ref{lemma_direct_sum2} we can express every matrix in $\mathfrak{b}^+(F)$ in terms of 
 the basis elements $\{ W_h , \ X_{\gamma_i} \mid 1 \leq h\leq r, \ 1 \leq i \leq l\}$.
This allows us to make the following inductive assumption on $k= 1, \dots, r$: For
 \begin{equation*}
   A_0:= A_0^- + \sum_{i=1}^{r} w_i W_i + \sum_{j=1}^l a_j X_{\gamma_j} 
 \end{equation*}
 there exists $u \in U^+$ such that 
\begin{equation*}
 A_k := \mathrm{Ad}(u)(A)+ \ell\delta(u)= A_0^- + \sum_{i=k+1}^{r} \bar{w}_i W_i +
 \sum_{j=1}^l \bar{a}_j X_{\gamma_j}. 
\end{equation*}
Let $k>1$ and assume the assumption holds for $k-1$, that is $A_0$ is gauge equivalent by an element of $U^+$ to a matrix of shape 
\begin{equation}\label{eq1_transformationLemma}
A_{k-1} = A_0^- + \sum_{i=k}^{r} \bar{w}_i W_i +
 \sum_{j=1}^l \bar{a}_j X_{\gamma_j} .
\end{equation}
We prove that there is a further element of $U^+$ such that $A_{k-1}$ is 
gauge equivalent to a matrix of shape $A_k$. 
We only have to consider the case when $\bar{w}_k \neq 0$ in $A_{k-1}$, otherwise 
there is nothing to show.  The element $X_k$ forms a basis of a one-dimensional 
root space. We denote the root which belongs to this root space 
by $\beta$. Now let $u_{\beta}(\rho)$ be a parametrized root group element 
 \begin{equation*}
 u_{\beta}(\rho)= \mathrm{exp} ( \rho X_k)
  = \sum_{j=0}^{\infty} \frac{1}{j!} \rho^j (X_k)^j.
 \end{equation*}
 We show that we can choose $\rho \in F$ such that the coefficient of $W_k$ in the expression 
 \[
 \mathrm{Ad}(u_{\beta}(\rho))(A_{k-1})+ \ell \delta(u_{\beta}(\rho)) 
 \]
 vanishes. Note that 
 \begin{equation*}
 \mathrm{Ad}(u_{\beta}(\rho))(A_{k-1}) =  \mathrm{exp} ( \rho \mathrm{ad}(X_k)) (A_{k-1})= \sum_{j\geq0}  \frac{1}
 {j!} \rho^j \mathrm{ad}^j (X_k) (A_{k-1})
 \end{equation*}
 and that this expression is linear in $A_{k-1}$. We compute the image of $A_{k-1}$ under 
 $\mathrm{Ad}(u_{\beta}(\rho))$ separately for the three main summands in \eqref{eq1_transformationLemma}.
 
We determine first the image of $A_0^-$ under $\mathrm{Ad}(u_{\beta}(\rho))$. 
For $j=1$ we have 
\[
\rho \mathrm{ad}(X_k)(A_0^-) =  \rho W_k \in \mathfrak{g}^{(r(k)-1)}
\]
by definition of the elements $W_i$ and for $j \geq 2$ the image $\mathrm{ad}^j (X_k)(A_0^-)$  lies in the eigenspace space $\mathfrak{g}^{(j r(k)-1)}$ by Lemma~\ref{lemma_eigenspace_decomp}, that is in a eigenspace with eigenvalue 
greater equal than $r(k)$. We conclude that
\begin{eqnarray*}
 \mathrm{Ad}(u_{\beta}(\rho))(A_0^-)&=&  \sum_{j \geq 0}^{} \frac{1}{j!} \rho^j \mathrm{ad}^j (X_k)(A_0^-)
 \in A_0^- + \rho W_k + \sum_{j\geq 0} \mathfrak{g}^{(r(k)+j)}.
\end{eqnarray*}

Next we consider the image of $\sum_{i=k}^{r} \bar{w}_i W_i $ under $\mathrm{Ad}(u_{\beta}(\rho))$. 
Let $p$ be the maximum of all indices $i$ with $r(i)=r(k)$. Clearly $p \geq k$ and for 
$i \geq p+1$ we have that $W_i$ lies in an eigenspace with eigenvalue greater equal than $r(k)$. 
This yields the representation  
\[
 \mathrm{ad}^0(X_k)( \sum_{i=k}^{r} \bar{w}_i W_i ) = \sum_{i=k}^{r} \bar{w}_i W_i 
\in \sum_{i=k}^{p} \bar{w}_i W_i + \sum_{j\geq 0} \mathfrak{g}^{(r(k)+j)} .
\]
For $j\geq 1$ and $W_i \in \mathfrak{g}^{(r(i)-1)}$ 
with $k \leq i \leq r$ the relation in Lemma~\ref{lemma_eigenspace_decomp} implies  
\[
\mathrm{ad}^j(X_k)(W_i) \in \mathfrak{g}^{(j r(k)+r(i)-1)}  ,
\]
that is the image lies in an eigenspace with eigenvalue greater equal than $r(k)$.
We conclude that  
\begin{eqnarray*}
 \mathrm{Ad}(u_{\beta}(\rho))(\sum_{i=k}^{r} \bar{w}_i W_i )
  \in  \sum_{i=k}^{p} \bar{w}_i W_i + \sum_{j\geq 0} \mathfrak{g}^{(r(k)+j)}.
\end{eqnarray*}
 
Now we determine the image of $\sum_{j=1}^l \bar{a}_j X_{\gamma_j}$ under $\mathrm{Ad}(u_{\beta}(\rho))$. 
Let $q$ be the maximum of all indices $j$ with $m_j \leq r(k)-1$. We can then write 
\[
 \mathrm{ad}^0(X_k)( \sum_{j=1}^l \bar{a}_j X_{\gamma_j} ) =  \sum_{j=1}^l \bar{a}_j X_{\gamma_j} \in  \sum_{j=1}^q \bar{a}_j X_{\gamma_j} + \sum_{j\geq 0} \mathfrak{g}^{(r(k)+j)}.
\]
Further the relation in Lemma~\ref{lemma_eigenspace_decomp} implies for $j\geq 1$ 
that $\mathrm{ad}^j(X_k)( X_{\gamma_i})$ lies in the eigenspace 
$\mathfrak{g}^{(jr(k)+m_i)}$, that is in eigenspaces with eigenvalue greater equal 
than $r(k)$. Thus we obtain 
\begin{eqnarray*}
 \mathrm{Ad}(u_{\beta}(\rho))( \sum_{j=1}^l \bar{a}_j X_{\gamma_j} ) \in  \sum_{j=1}^q \bar{a}_j X_{\gamma_j} + \sum_{j\geq 0} \mathfrak{g}^{(r(k)+j)}.
\end{eqnarray*}
Finally, we compute the image of $u_{\beta}(\rho)$ under the logarithmic derivative. 
By construction the Lie algebra of the root subgroup $U_{\beta}$ is 
$\mathfrak{g}_{\beta}$ and so by Proposition~\ref{log_derivate} we have 
$\ell\delta (u_{\beta}(\rho)) \in \mathfrak{g}_{\beta}$ which is contained in 
$\mathfrak{g}^{(r(k))}$. Summing up we obtain 
\begin{equation*}
 \mathrm{Ad}(u_{\beta}(\rho))(A_{k-1})+\ell \delta(u_{\beta}(\rho)) 
 \in  A_0^- + \rho W_k + \sum_{i=k}^{p} \bar{w}_i W_i + \sum_{j=1}^{q} \bar{a}_j X_{\gamma_j} + \sum_{j\geq 0} 
 \mathfrak{g}^{(r(k)+j)}. 
\end{equation*}
Hence, for $\rho=-\bar{w}_k$ the induction assumption follows for $k>1$. 
The same argumentation shows that the induction assumption also holds 
for $k=1$. The assertion of the lemma follows then for $k=r$.   
\end{proof}
\section{The equation for $\mathrm{SL}_{l+1}(C)$}\label{chap7}
Let $\epsilon_1, \dots , \epsilon_{l+1}$ be the standard orthonormal basis of 
$\mathbb{R}^{l+1}$ with respect to the usual inner product 
$(\cdot, \cdot)$ of $\mathbb{R}^{l+1}$. 
From \cite[VI, Planche I]{BourGroupesetAlgebresIV-VI} we obtain that the root system 
$\Phi$ of type $A_l$ consists of 
the vectors $ \epsilon_i - \epsilon_j $ with $1 \leq i,j \leq l+1$ and $ i \neq j$. 
Further the elements $\alpha_i := \epsilon_{i} - \epsilon_{i+1}$ with 
$1 \leq i \leq l$ form a basis $\Delta$ of $\Phi$ and with respect to $\Delta$ 
the root system of type $A_l$ is
 \begin{equation*}
 \Phi= \{ \pm \alpha_i \mid i =1, \dots,l\} \cup \{ \pm (\alpha_i + \alpha_{i+1} + \dots +  \alpha_{j})  \mid 1 \leq i < j  \leq l \} .
 \end{equation*}
A basis of a Cartan decomposition of $\mathfrak{sl}_{l+1}$ can be taken from \cite{BourGroupesetAlgebresVII-VIII}. 
More precisely, the discussion in \cite[VIII, \S 13.1]{BourGroupesetAlgebresVII-VIII}  yields that the set of matrices
 \begin{equation*}
  \{ E_{ij} \mid 1 \leq i,j \leq l+1, i \neq j \} \cup \{ E_{ii}- E_{i+1,i+1} \mid 1 \leq i \leq l \}
 \end{equation*}
is a basis of $\mathfrak{sl}_{l+1}$ and that the matrices $H_i := E_{ii}- E_{i+1,i+1}$ with $1 \leq i \leq l$ form a basis of a Cartan algebra 
$\mathfrak{h}$. With respect to $\mathfrak{h}$ the elements $E_{ij}$ generate the one-dimensional root spaces $(\mathfrak{sl}_{l+1})_{\alpha}$ 
of the corresponding Cartan decomposition where $\alpha = \epsilon_i - \epsilon_j$ for $1 \leq i,j\leq l+1$ and $i \neq j$.
\begin{lemma}\label{lemma_sl_roots}
The $l$ roots 
\begin{equation*}
\gamma_1:= \alpha_l, \ \gamma_{2}:=\alpha_{l-1} + \alpha_l, \ \dots \ , \ \gamma_l:= \alpha_1 + \dots + \alpha_l   
\end{equation*}
are complementary roots of $\Phi^+$.
\end{lemma}
\begin{proof}
From \cite[VI, Planche I]{BourGroupesetAlgebresIV-VI} it follows that the maximal root in $\Phi^+$ has height $l$ and that for 
$k=1, \dots,l$ the number of roots of height $k$ is $l+1-k$. We conclude with Shapiro's method that the exponents
of $\mathfrak{sl}_{l+1}$ are $m_k=k$ for $k=1 ,\dots, l$ (alternatively see 
\cite[VI, Planche I]{BourGroupesetAlgebresIV-VI}).  
Then  the roots $\gamma_k$ satisfy $\mathrm{ht}(\gamma_k)=m_k$.  
Following the proof of Lemma~\ref{lemma_direct_sum2}, we need to show that we can extend the set 
 $\{ W_h \mid W_h \in \mathfrak{sl}_{l+1}^{(k)} \}$ to a basis of $\mathfrak{sl}_{l+1}^{(k)}$ by adjoining 
 $X_{\gamma_k}$. We proceed by induction on $l$.\\ 
 In case $l=2$, the root system $\Phi$ consists of the roots $\pm \alpha_1$, $\pm \alpha_2$ and 
 $\pm(\alpha_1 + \alpha_2)$. Then the single element $W_3= [A_0^-, X_{\alpha_1 + \alpha_2}]$
forms a basis of 
\[
\mathrm{ad}(A_0^-)(\mathfrak{u}^+) \cap \mathfrak{sl}_3^{(1)} 
\]
and has a non-zero component 
in $(\mathfrak{sl}_3)_{\alpha_1}$ since the Lie bracket of $X_{-\alpha_2}$ and 
$X_{\alpha_1 + \alpha_2}$ lies in $(\mathfrak{sl}_{l+1})_{\alpha_1}$ and is nonzero.
We conclude that $\{W_3,  X_{\alpha_2}\}$ is a basis of 
$\mathfrak{sl}_3^{(1)}$. Finally, since $\alpha_1 + \alpha_2$ is the unique root 
of height $2$, the 
vector $X_{\alpha_1 + \alpha_2}$ forms a basis for $\mathfrak{sl}_{l+1}^{(2)}$.\\
Now assume the assertion is true for $l-1$.  
The Dynkin diagram of type $A_l$ shows that the subset 
$\Phi' = \{ \alpha \in \Phi \mid n_{\alpha_1}(\alpha) =0 \}$ of $\Phi $
is a root system of type $A_{l-1}$ which is generated by the simple roots 
$\alpha_2, \dots ,\alpha_l$. We identify $\mathfrak{sl}_l$ inside $\mathfrak{sl}_{l+1}$ with the 
Lie algebra generated by the matrices $X_{\pm \alpha_i} \in \mathfrak{sl}_{l+1}$ for $i=2,\dots,l$.
By the induction assumption $X_{\gamma_k}$ and all vectors of 
\[
 \{ W_h \mid W_h \in \mathfrak{sl}_{l }^{(k)} \} \subset \{ W_h \mid W_h \in \mathfrak{sl}_{l+1}^{(k)} \} 
\]
form a basis of $\mathfrak{sl}_l^{(k)}$ for $k=1,\dots, l-1$. 
The elements of $\Phi^+ \setminus \Phi'^+$ are the roots $\alpha_1 + \dots + \alpha_k$ for $k=1,\dots,l$
and for each $k$ there is a unique root in $\Phi^+ \setminus \Phi'^+$ of height $k$ which we denote by $\beta_k$.
Since $n_{\alpha_1}(\alpha)=0$ for all $\alpha \in \Phi'$ we conclude for $k=1,\dots,l-1$ that $\beta_{k+1}$
is the unique root of $\Phi$ which satisfies $\beta_{k+1}-\alpha_i=\beta_k$ for a simple root $\alpha_i$.
Hence the vector $W= [A_0^-, X_{\beta_{k +1}}]$ is unique
among the vectors $ W_h \in \mathfrak{sl}_{l+1}^{(k)}$ with a non-zero component in $(\mathfrak{sl}_{l+1})_{\beta_{k}}$. We conclude that $X_{\gamma_k}$ and  
\[
\{ W_h \mid W_h \in \mathfrak{sl}_{l+1}^{(k)} \} = \{ W \} \cup \{ W_h \mid W_h \in \mathfrak{sl}_{l }^{(k)} \} 
\]
form a basis of $\mathfrak{sl}_{l+1}^{(k)}$ for $k=1,\dots, l-1$.
Since $\gamma_l$ is the unique root of height $l$, it is clear that $  X_{\gamma_l}  $ forms a basis of $\mathfrak{sl}_{l+1}^{(l)}$. 
\end{proof}

 \begin{theorem}\label{theorem for SL_n}
  The linear parameter differential equation 
  \begin{equation*}
   L(y,\textbf{t})= y^{(l+1)} - \sum\nolimits_{i=1}^{l} t_i y^{(i-1)}=0
  \end{equation*}
has differential Galois group $\mathrm{SL}_{l+1}(C)$ over $F_1$.
 \end{theorem}
 \begin{proof}
We consider the differential equation 
$\partial(\boldsymbol{y})=A_{\mathrm{SL}_{l+1}}(\boldsymbol{t}) \boldsymbol{y}$ 
over $F_1$, where $A_{\mathrm{SL}_{l+1}}(\boldsymbol{t})$ 
is defined by 
\begin{equation*}
A_{\mathrm{SL}_{l+1}}(\textit{\textbf{t}}) := A_0^+ + \sum_{i=1}^l t_i X_{-\gamma_{l+1- i}} 
\end{equation*}
and we denote its differential Galois group by $H(C)$. 
We prove that $H(C)= \mathrm{SL}_{l+1}(C)$.\\
By construction $A_{\mathrm{SL}_{l+1}}(\boldsymbol{t}) \in \mathfrak{sl}_{l+1}(F_1)$ 
and so Proposition~\ref{Prop_upper} yields that $H(C)$ is a subgroup 
of $\mathrm{SL}_{l+1}(C)$.

Conversely, Proposition~\ref{mod_Mitch_Singer} implies that there exists 
a matrix differential equation $\partial(\boldsymbol{y})=A \boldsymbol{y}$ 
over $F_2$ such that $A \in  A_0^+  + \mathfrak{b}^-(C[z])$ and such that its 
differential Galois group is $\mathrm{SL}_{l+1}(C)$. 
We are going to apply Lemma~\ref{Prop_tans_lemma}. If we interchange the role of 
the positive and negative roots, Lemma~\ref{Prop_tans_lemma} yields
that $A$ is a gauge equivalent to a matrix in the plane $A_0^+ + \mathfrak{r}(C[z])$,
where by Lemma~\ref{lemma_sl_roots} the space 
$\mathfrak{r}=\sum_{i=1}^l \mathfrak{g}_{-\gamma_i}$ is a root space 
complement of $\mathrm{ad}(A_0^+)(\mathfrak{u}^-)$. Hence, there exists a surjective 
$C \{\boldsymbol{t} \}$-specialization 
$\sigma: C \{ \boldsymbol{t} \} \rightarrow C[z]$ such that the differential Galois group of
\[
\partial(\boldsymbol{y})= A_{\mathrm{SL}_{l+1}}(\sigma(\boldsymbol{t})) \boldsymbol{y} \]
over $F_2$ is $\mathrm{SL}_{l+1}(C)$.
Theorem~\ref{specialization_bound} then asserts that $\mathrm{SL}_{l+1}(C)$ is 
contained in $H(C)$. Combining this with the relation from above we obtain that $H(C) = \mathrm{SL}_{l+1}(C)$.

Finally, we show that the matrix differential equation defined by 
$A_{\mathrm{SL}_{l+1}}(\boldsymbol{t})$ is equivalent 
to the linear parameter differential equation in the statement of the theorem. 
But this is clearly satisfied (see for instance \cite[Chapter 1.2]{P/S}), 
since the defining matrix
\begin{equation*}
 A_{\mathrm{SL}_{l+1}}(\boldsymbol{t})= \sum_{i=1}^l E_{i,i+1} + \sum_{i=1}^l t_i E_{l+1,i}
\end{equation*}
has the shape of a companion matrix with trace zero.  
\end{proof}
\section{The equation for $\mathrm{SP}_{2l}(C)$}\label{chap8}
Let $\epsilon_1, \dots, \epsilon_l$ be the standard orthonormal basis of $\mathbb{R}^l$ with respect to the usual inner product $(\cdot,\cdot)$ of $\mathbb{R}^l$.  
Then from \cite[VI, Planche III]{BourGroupesetAlgebresIV-VI} we obtain that the root system $\Phi$ of type $C_l$ consists
of the vectors $\pm 2 \epsilon_i$ with $1 \leq i \leq l$ and $\pm \epsilon_i \pm \epsilon_j$ with $1 \leq i < j \leq l$. A basis $\Delta$ for $\Phi$ is given by the vectors
\begin{equation*}
 \alpha_1 = \epsilon_1 - \epsilon_2, \ \alpha_2 = \epsilon_2 - \epsilon_3, \ \dots , \ \alpha_{l-1}=\epsilon_{l-1} - \epsilon_l, \ \alpha_l = 2 \epsilon_l 
\end{equation*}
and with respect to $\Delta$ the positive roots are
\begin{align*}
 & \epsilon_i - \epsilon_j = \sum_{i \leq k < j} \alpha_k &&  ( 1 \leq i < j \leq l), \\
 & \epsilon_i + \epsilon_j =  \sum_{i \leq k < j} \alpha_k +  2 \sum_{j \leq k <l} \alpha_k + \alpha_l &&  (1 \leq i < j \leq l),   \\
 & 2 \epsilon_i \quad \; \ = 2 \sum_{i \leq k <l} \alpha_k + \alpha_l   && (1 \leq i \leq l) .
\end{align*}
The negative roots of $ \Phi$ are obtained by interchanging all signs in the expressions of the positive roots.

We take a Cartan decomposition of the Lie algebra $\mathfrak{sp}_{2l}$ from \cite[VIII, \S 13.3]{BourGroupesetAlgebresVII-VIII}. To this purpose we renumber the rows and columns of the matrices 
$E_{ij} \in C^{2l \times 2l}$ into $(1,2, \dots, l ,-l, \dots,-2,-1)$. Then by 
\cite[VIII, \S 13.3]{BourGroupesetAlgebresVII-VIII} a Cartan subalgebra 
$\mathfrak{h}$ is generated by the basis elements $H_i = E_{i,i} - E_{-i,-i}$ 
with $1\leq i \leq l$ and for $\alpha \in \Phi$ the root spaces 
$(\mathfrak{sp}_{2l})_{\alpha}$ with respect to $\mathfrak{h}$ are generated 
by the following matrices correspondingly:
\begin{align*}
 & X_{2\epsilon_i} \ \; = E_{i,-i} && X_{\epsilon_i - \epsilon_j} \ \; = E_{i,j}- E_{-j,-i} && X_{\epsilon_i + \epsilon_j} \ \; = E_{i,-j} + E_{j,-i} \\
 & X_{-2\epsilon_i} = -E_{-i,i} && X_{- \epsilon_i + \epsilon_j} = - E_{j,i} + E_{-i,-j} &&  X_{-\epsilon_i - \epsilon_j} = - E_{-i,j} - E_{-j,i} .
\end{align*}
\begin{lemma}\label{lemma_Sp2l_roots}
 The $l$ roots 
 \begin{equation*}
 \gamma_1 := \alpha_l, \gamma_2 := 2\alpha_{l-1} + \alpha_l , \gamma_3 := 2 \alpha_{l-2} + 2 \alpha_{l-1}+ \alpha_l, \dots , \gamma_l:= 2 \alpha_1 + \dots + 2\alpha_{l-1} + \alpha_l
\end{equation*}
are complementary roots of $\Phi$.
\end{lemma}
\begin{proof}
 By \cite[VI, Planche III]{BourGroupesetAlgebresIV-VI} 
 the exponents of a Lie algebra of type $C_l$ are
 \begin{equation*}
  m_1 = 1 , \ m_2 =3 , \ m_3 = 5, \dots, m_l = 2l-1.
 \end{equation*}
Hence, the roots $\gamma_i$ satisfy $\mathrm{ht}(\gamma_i)=m_i$ for all 
$1 \leq i \leq l$. In order to show that the roots $\gamma_i$ satisfy the second 
condition of Lemma~\ref{lemma_direct_sum2}, we need to prove that the set 
\[
\{X_{\gamma_i} \} \cup \{ W_h \mid W_h  \in \mathfrak{sp}_{2l}^{(m_i)} \} 
\]
is a basis of $\mathfrak{sp}_{2l}^{(m_i)}$ for $1 \leq i \leq l$. 
The proof is done by induction on $l$.

For $l=2$ the root system $\Phi$ consists of the roots 
\[
\pm \alpha_1, \ \pm \alpha_2, \ \pm (\alpha_1 + \alpha_2), \  \pm (2\alpha_1 + \alpha_2) 
\]
and the exponents are $m_1=1$ and $m_2=3$. 
Since the Lie bracket of $X_{-\alpha_2}$ with $X_{\alpha_1 + \alpha_2}$ is 
non-zero and lies in $(\mathfrak{sp}_{2l})_{\alpha_1}$,   
the vector $ W_3=[A_0^-, X_{\alpha_1 +\alpha_2}]$
has a non-zero component in $(\mathfrak{sp}_{4})_{\alpha_1}$. We conclude that  
$ W_3$ and $X_{\alpha_2}$ form a basis of $\mathfrak{sp}_{4}^{(1)}$. 
Since $2\alpha_1 + \alpha_2$ is the unique root of height three in $\Phi^+$, 
the vector $ X_{2\alpha_1 + \alpha_2}$ forms clearly a basis of 
$\mathfrak{sp}_{4}^{(3)}$.

Let $l>2$ and assume the assertion is true for $l-1$. We consider the subset  
\[
\Phi' :=\{ \alpha \mid n_{\alpha_1}(\alpha) =0\} 
\]
of the root system $\Phi$. Then the Dynkin diagram implies 
that $\Phi'$ is a root system of type $C_{l-1}$ and it is generated by the simple roots $\alpha_2,\dots, \alpha_l$. 
We identify in the following $\mathfrak{sp}_{2l-2}$ inside $\mathfrak{sp}_{2l}$ 
with the Lie subalgebra generated by $X_{\pm \alpha_i} \in \mathfrak{sp}_{2l}$ 
with $i=2,\dots,l$. By the induction assumption $X_{\gamma_i}$
and the vectors of the set
\[
\{ W_h \mid W_h \in \mathfrak{sp}_{2l-2}^{(m_i)} \} \subset \{ W_h \mid W_h \in \mathfrak{sp}_{2l}^{(m_i)} \} 
\]
form a basis of $\mathfrak{sp}_{2l-2}^{(m_i)}$ with $i=1,\dots,l-1$.
From the shapes of the roots in $\Phi^+$ 
we conclude that the set $\Phi^+\setminus \Phi'^+$ consists of the roots 
$\alpha_1 + \dots + \alpha_k$ for $k=1,\dots,l$ and 
\[
\alpha_1 + \dots + 2 \alpha_k+ \dots + 2 \alpha_{l-1} + \alpha_l
\]
for $k=l-1,\dots,1$.
Then for $k =1, \dots ,  2l-1$, where $1, \dots ,  2l-1$ are all possible heights 
of roots in $\Phi^+$, there exists a unique root in $\Phi^+\setminus \Phi'^+$ of 
height $k$ which we denote by $\beta_k$.
For $k=1,\dots,2l-2$ the root $\beta_{k+1}$ is the unique root of $\Phi$ which satisfies  
$\beta_{k+1} - \alpha_i = \beta_k$ for a simple root $\alpha_i$, since for all 
$\alpha \in \Phi'$ we have $n_{\alpha_1}(\alpha) = 0$.
We conclude that for $m_i \leq 2l-2$, i.e.~for $i=1,\dots,l-1$, the vector 
\[
W:=[A_0^-,X_{\beta_{m_i+1}}] 
\]
is unique among the vectors $W_h \in \mathfrak{sp}_{2l}^{(m_i)}$ with a non-zero 
component in the root space $(\mathfrak{sp}_{2l})_{\beta_{m_i}}$. Thus $X_{\gamma_i}$ and the vectors 
\[
\{ W_h \mid W_h \in \mathfrak{sp}_{2l}^{(m_i)} \} = \{ W \} \cup  \{ W_h \mid W_h \in \mathfrak{sp}_{2l-2}^{(m_i)} \}
\]
form a basis of $\mathfrak{sp}_{2l}^{(m_i)}$ for $i=1,\dots,l-1$.
Since $\gamma_l$ is the unique root of maximal height $m_l=2l-1$ in $\Phi^+$, the vector
$ X_{\gamma_l}$ forms clearly a basis of $(\mathfrak{sp}_{2l})_{\gamma_l} = \mathfrak{sp}_{2l}^{(m_l)}$. This completes the induction.
\end{proof}
\begin{lemma}\label{Lemma_sp_2l_equation}
 The matrix parameter differential equation $\partial(\textit{\textbf{y}})= A_{\mathrm{SP}_{2l}}(\textit{\textbf{t}}) \textit{\textbf{y}}$ 
 over $F_1$, where  
 \begin{equation*}
  A_{\mathrm{SP}_{2l}}(\textit{\textbf{t}}):= A_0^+ + \sum_{i=1}^l (-t_i)X_{- \gamma_i},
\end{equation*}
is equivalent to the linear parameter differential equation
\begin{equation*}
L(y,\textbf{t})= y^{(2l)} - \sum\nolimits_{i=1}^{l} (-1)^{i-1} (t_i y^{(l-i)})^{(l-i)}=0.
\end{equation*}
\end{lemma}
\begin{proof}
Let $\textit{\textbf{y}}=(y_1, \dots, y_l,y_{l+1}, \dots, y_{2l})^{tr}$. 
Then the matrix differential equation 
$\partial(\textit{\textbf{y}})= A_{\mathrm{SP}_{2l}}(\textit{\textbf{t}}) \textit{\textbf{y}}$, where 
\begin{equation*}
 A_{\mathrm{SP}_{2l}}(\textit{\textbf{t}})= (\sum_{i=1}^{l-1} E_{i,i+1}- E_{-i-1,-i}) + E_{l,-l} - \sum_{i=1}^l t_i E_{-l-1+i, l+1-i},
\end{equation*}
expands into the following system of differential equations:
\begin{align*}
 &y_1' = y_2 && y_{l+1}'=t_1 y_l - y_{l+2}\\
 &\quad \vdots && \quad \vdots\\
 &y_{l-1}' = y_{l} && y_{2l-1}' = t_{l-1} y_2 - y_{2l}\\
 &y_l' = y_{l+1} && y_{2l}' = t_l y_1
\end{align*}
For $1\leq k \leq l-1$ we consider the subsystem
\begin{gather*}
 y_{l-k+1}'= y_{l-k+2}, \dots, \ y_l' = y_{l+1}, \\  y_{l+1}'=t_1 y_l - y_{l+2}, \dots, \ y_{l+k}'=t_k y_{l+1-k} - y_{l+1+k}
\end{gather*}
and prove by induction on $k$ that this system yields the following linear 
differential equation:
\begin{equation*}
y_{l+1-k}^{(2k)} = \sum_{i=1}^k (-1)^{i-1} (t_i y_{l-k+1}^{(k-i)})^{(k-i)} + (-1)^k y_{l+1+k}. 
\end{equation*}
Letting $k=1$, the subsystem consists of the two equations $y_l' = y_{l+1}$ and  $y_{l+1}'=t_1 y_l - y_{l+2}$. 
We differentiate the first equation and we substitute in this expression the right hand side of the second 
equation for $y_{l+1}'$, i.e. we obtain $y_{l}''=t_1 y_l - y_{l+2}$.\\
Now let $k>1$. Then for $k-1$ the induction assumption, applied to the subsystem 
\begin{gather*}
 y_{l-k+2}'= y_{l-k+3}, \dots, \ y_l' = y_{l+1}, \\  y_{l+1}'=t_1 y_l - y_{l+2}, \dots, \ y_{l+k-1}'=t_k y_{l+2-k} - y_{l+k},
\end{gather*}
yields the linear differential equation
\begin{equation}
 y_{l-k+2}^{(2k-2)} = \sum_{i=1}^{k-1} (-1)^{i-1} (t_i y_{l-k+2}^{(k-1-i)})^{(k-1-i)} 
 + (-1)^{k-1} y_{l+k}. \label{C_l_eqn_1}
\end{equation}
The equation $y_{l-k+1}'= y_{l-k+2}$ implies that we can replace $y_{l-k+2}$ by $y_{l-k+1}'$ in Equation~(\ref{C_l_eqn_1}) and 
if we differentiate this expression, we obtain
\begin{equation}
y_{l-k+1}^{(2k)} = \sum_{i=1}^{k-1} (-1)^{i-1} (t_i y_{l-k+1}^{(k-i)})^{(k-i)} + (-1)^{k-1} y_{l+k}'. \label{C_l_eqn_2}
\end{equation}
Differentiating Equation~(\ref{C_l_eqn_2}) and making the substitution $y_{l+k}'=t_k y_{l+1-k} - y_{l+1+k}$ completes the induction.\\
Now we consider the full system of equations. Ignoring the first and last equation, we get from the induction for $k=l-1$ the following equation:
\begin{equation*}
y_{2}^{(2l-2)} = \sum_{i=1}^{l-1} (-1)^{i-1} (t_i y_{2}^{(l-1-i)})^{(l-1-i)} + (-1)^ {l-1} y_{2l}. 
\end{equation*}
As above, the first equation of the full system implies that we can replace $y_2$ by $y_1'$ which yields an expression in $y_{2l}$ and derivatives of $y_1$. 
Thus, if we differentiate this expression and make the substitution $y_{2l}' = t_l y_1$, which is given by the last equation of the full system, 
we obtain the linear differential equation in the assertion of the lemma.
\end{proof}

\begin{theorem}
The linear parameter differential equation 
\begin{equation*}
L(y,\textbf{t})= y^{(2l)} - \sum\nolimits_{i=1}^{l} (-1)^{i-1} (t_i y^{(l-i)})^{(l-i)}=0
\end{equation*}
has differential Galois group $\mathrm{SP}_{2l}(C)$ over $F_1$.
\end{theorem}
\begin{proof}
The proof just works as the proof of Theorem~\ref{theorem for SL_n}. 
For $ A_{\mathrm{SP}_{2l}}(\boldsymbol{t})$ as in Lemma~\ref{Lemma_sp_2l_equation} one shows 
with Proposition~\ref{mod_Mitch_Singer}, Lemma~\ref{lemma_Sp2l_roots} and 
Lemma~\ref{Prop_tans_lemma} (interchange 
the role of the positive and negative roots) that 
$\partial(\boldsymbol{y})= A_{\mathrm{SP}_{2l}}(\boldsymbol{t}) \boldsymbol{y}$ 
specializes to a differential equation over $F_2$ with group $\mathrm{SP}_{2l}(C)$. 
If one applies then Theorem~\ref{specialization_bound} and Proposition~\ref{Prop_upper} 
one obtains that the differential Galois group of the equation defined by 
$A_{\mathrm{SP}_{2l}}(\boldsymbol{t})$ 
is $\mathrm{SP}_{2l}(C)$. Finally one uses Lemma~\ref{Lemma_sp_2l_equation} to complete 
the proof.
\end{proof}
 \section{The equation for $\mathrm{SO}_{2l+1}(C)$} \label{chap9}
 Let $\epsilon_1, \dots , \epsilon_l$ be the standard orthonormal basis of $\mathbb{R}^{l}$ with respect to the standard inner product 
 on $\mathbb{R}^{l}$and let
 \begin{equation*}
\Phi = \{ \pm \epsilon_i \mid 1\leq i \leq l \} \cup \{ \pm \epsilon_i \pm \epsilon_j \mid 1 \leq i < j \leq l \}.
 \end{equation*}
 Then by \cite[VI, Planche II]{BourGroupesetAlgebresIV-VI} the set $\Phi$ is a root
 system of type $B_l$ and a basis $\Delta$ is 
 given by the vectors  $\alpha_1 := \epsilon_1 - \epsilon_2$, $\alpha_2 := \epsilon_2 - \epsilon_3,\dots ,  
 \alpha_{l-1} := \epsilon_{l-1} - \epsilon_l$ and $\alpha_l := \epsilon_l$.
 With respect to $\Delta$ the positive roots of $\Phi$ are
 \begin{equation*}
 \sum_{h \leq k \leq l} \alpha_k = \epsilon_h , \ \sum_{i \leq k < j} \alpha_k = \epsilon_i - \epsilon_j \
 \mathrm{and} \  \sum_{i \leq k < j} \alpha_k + 2 \sum_{j \leq k \leq l} \alpha_k = \epsilon_i + \epsilon_j,
 \end{equation*}
 where  $1 \leq h \leq l$ and $1 \leq i < j \leq l$. Note that we obtain all 
 negative roots of $\Phi$ in terms of the $\alpha_i$ 
 by interchanging all signs in the expressions for the positive roots.
 
Let us now renumber the rows and columns of the matrices 
$E_{ij} \in C^{2l+1 \times 2l+1}$ into $(1, \dots,l,0,-l, \dots,-1)$. 
 Then a basis of a Cartan decomposition for $\mathfrak{so}_{2l+1}$ is given by 
 the following matrices (see \cite[VIII, \S 13.2]{BourGroupesetAlgebresVII-VIII}):
The diagonal matrices $H_i := E_{i,i} - E_{-i,-i}$ with $1 \leq i \leq l$ 
generate a Cartan subalgebra, which we denote by $\mathfrak{h}$,  
and for a root $\alpha \in \Phi$ the root space $(\mathfrak{so}_{2l+1})_{\alpha}$ 
with respect to $\mathfrak{h}$ is generated correspondingly
by one of the following matrices:
\begin{align*}
X_{\epsilon_h} \quad \, &=  2 E_{h,0} + E_{0,-h}\quad&  X_{-\epsilon_h} \quad \, & =  -2 E_{-h,0} - E_{0,h}  \\
X_{\epsilon_i - \epsilon_j}& =  E_{i,j} - E_{-j,-i}&  X_{\epsilon_j-\epsilon_i} \ \ & =  - E_{j,i} + E_{-i,-j}\\
X_{\epsilon_i + \epsilon_j} &= E_{i,-j} - E_{j,-i}&  X_{-\epsilon_i-\epsilon_j}  &  =  - E_{-j,i} + E_{-i,j} 
\end{align*}
where $1 \leq h \leq l$ and $1 \leq i < j \leq l$.
\begin{lemma}\label{lemma_so2l+1_roots}
 Let $\gamma_1 = \alpha_l$ and $\gamma_{i+1}=\alpha_{l-i} + 2 \alpha_{l+1-i} + \dots +2 \alpha_l$ 
for $1 \leq i \leq l-1$. The $l$ roots $\gamma_1, \dots, \gamma_l$ are complementary roots of $\Phi$.
\end{lemma}
\begin{proof}
From \cite[VI, Planche II]{BourGroupesetAlgebresIV-VI} we obtain that the exponents of $\mathfrak{so}_{2l+1}$ are 
\begin{equation*}
m_1 = 1, \ m_2 = 3, \  m_3 = 5, \dots, \ m_l = 2l-1 
\end{equation*}
and it is easy to check that the roots $\gamma_i$ satisfy $\mathrm{ht}(\gamma_i)=m_i$. 
We make now the following inductive assumption on $l\geq 2$: For all $1 \leq i \leq l$ the set
\begin{equation*}
\{ W_h \mid W_h \in \mathfrak{so}_{2l+1}^{(m_i)} \} \cup \{ X_{\gamma_i} \}
\end{equation*}
is a basis of $\mathfrak{so}_{2l+1}^{(m_i)}$.

In case $l=2$, that is $\Phi$ is of type $B_2$, the positive roots are $\alpha_1$, 
$\alpha_2$, $\alpha_1+\alpha_2$ and $\alpha_1 +2 \alpha_2$ and the exponents are 
$m_1=1$ and $m_2=3$. The Lie bracket
of $X_{-\alpha_2}$ with $X_{\alpha_1+\alpha_2}$ is non-zero and lies in 
$(\mathfrak{so}_{5})_{\alpha_1}$ and so $W_3 = [A_0^-, X_{\alpha_1 + \alpha_2}]$ 
has a non-zero component in the root space $(\mathfrak{so}_{5})_{\alpha_1}$.
 Hence the vectors $W_3$ and $X_{\alpha_2}$ form a basis of $\mathfrak{so}_{5}^{(1)}$. Finally the vector
 $X_{\alpha_1 + 2\alpha_2}$ forms a basis of $\mathfrak{so}_{5}^{(3)}$, since $\alpha_1 + 2\alpha_2$ is the unique root of maximal height three.
 
Let $l>2$ and assume the claim is true for $l-1$. We consider the subset 
\[
\Phi'= \{ \alpha \in \Phi \mid n_{\alpha_1}(\alpha) = 0 \} \subset \Phi
\]
The Dynkin diagram shows that $\Phi'$ is a subsystem of type $B_{l-1}$ which is 
generated by the simple roots $\alpha_i$ with $i=2,\dots,l$. We identify in the 
following  $\mathfrak{so}_{2l-1}$ inside $\mathfrak{so}_{2l+1}$
with the Lie subalgebra generated by $X_{\pm \alpha_i} \in \mathfrak{so}_{2l+1}$ 
for $i=2,\dots,l$. From the 
induction assumption we obtain for $i=1,\dots,l-1$ that $X_{\gamma_i}$ and the vectors 
\[
\{ W_h \mid W_h \in \mathfrak{so}_{2l-1}^{(m_i)} \} \subset \{ W_h \mid W_h \in \mathfrak{so}_{2l+1}^{(m_i)} \} 
\]
form a basis of $ \mathfrak{so}_{2l+1}^{(m_i)}$.
Further it follows from the shape of the roots in $\Phi^+$ that $\Phi^+ \setminus \Phi'^+$ consists of the roots
 $\alpha_1+ \dots + \alpha_k$ with $k=1, \dots , l$ and of the roots 
\begin{equation*} 
 \alpha_1 + \dots +   2\alpha_k + \dots + 2 \alpha_l  
\end{equation*}
with $k=l,\dots,2$. For $k=1,\dots, 2l-1$, where $1,\dots,2l-1$ are all possible heights of roots in $\Phi^+$,
there exist a unique root in $\Phi^+ \setminus \Phi'^+$ of height $k$. We denote this root by $\beta_k$.
Then for $k=1,\dots,2l-2$ the root $\beta_{k+1}$ is the unique root in $\Phi$ which fulfills the equation
$\beta_{k+1}-\alpha_i=\beta_k$ for some simple root $\alpha_i$, since the coefficient of $\alpha_1$ in 
the representation of $\beta_k$ has to be non-zero.
This implies for $m_i \leq 2l-2$, that is for $i=1,\dots,l-1$, that 
\[
W:= [A_0^-, X_{\beta_{m_i+1}}] 
\]
is unique among the vectors $ W_h \in \mathfrak{so}_{2l+1}^{(m_i)}$ which have a non-zero component in the 
root space $(\mathfrak{so}_{2l+1})_{\beta_{m_i}}$. We conclude that $X_{\gamma_i}$ and 
the vectors of
\[
\{ W_h \mid W_h \in \mathfrak{so}_{2l+1}^{(m_i)} \} = \{ W \} \cup \{ W_h \mid W_h \in \mathfrak{so}_{2l-1}^{(m_i)} \}
\]
form a basis of $\mathfrak{so}_{2l+1}^{(m_i)}$ for $i=1,\dots,l-1$. 
For $m_l=2l-1$ the vector $X_{\gamma_l}$ forms a basis of $\mathfrak{so}_{2l+1}^{(m_l)}$, since 
$\gamma_l$ is the unique root of maximal height $2l-1$.  
 \end{proof}
\begin{lemma}\label{Lemma_SO_2l+1_equation}
 The matrix parameter differential equation $\partial(\textit{\textbf{y}})=A_{\mathrm{SO}_{2l+1}}(\textit{\textbf{t}})\textit{\textbf{y}}$ 
 over $F_1$ is equivalent to the linear parameter differential equation
\begin{equation*}
L(y,\textbf{t})= y^{(2l+1)} - \sum\nolimits_{i=1}^{l} (-1)^{i-1} ((t_i y^{(l+1-i)})^{(l-i)}+(t_i y^{(l-i)})^{(l+1-i)})= 0
\end{equation*}
where the matrix $A_{\mathrm{SO}_{2l+1}}(\textit{\textbf{t}})$ is defined by 
 \begin{equation*}
A_{\mathrm{SO}_{2l+1}}(\textit{\textbf{t}}):= A_0^+ + \frac{1}{2} \sum_{i=1}^l t_i X_{-\gamma_i}. 
\end{equation*}
\end{lemma}
\begin{proof} 
With respect to the explicit basis of $\mathfrak{so}_{2l+1}$ the matrix 
differential equation $\partial(\textit{\textbf{y}})=A_{\mathrm{SO}_{2l+1}}(\textit{\textbf{t}})\textit{\textbf{y}}$ is
equivalent to the following system of equations:
\begin{align*}
& y_{k}'  =y_{k+1}& (1 \leq k \leq l-1)  \\
& y_l ' = 2 y_0 , \quad
 y_0 ' = \frac{1}{2} t_1 y_l + y_{-l}, \quad
 y_{-l}' = - \frac{1}{2} t_2 y_{l-1} + t_1 y_0 - y_{-l+1}&\\
& y_{-l+k}' = -\frac{1}{2} t_{k+2} y_{l-1-k} + \frac{1}{2} t_{k+1} y_{l+1-k} - y_{-l+1+k} &  (1 \leq k \leq l-2) \\
& y_{-1}'  = \frac{1}{2} t_l y_2 &
\end{align*}
We consider first the case $l \geq 3$ and prove by induction on $1 \leq k \leq l-2$ that the subsystem defined by the equations 
\begin{align}
& y_{l-j}'  =y_{l-j+1}& (1 \leq j \leq k)  \\
& y_l ' = 2 y_0 , \ 
 y_0 ' = \frac{1}{2} t_1 y_l + y_{-l}, \ 
 y_{-l}' = - \frac{1}{2} t_2 y_{l-1} + t_1 y_0 - y_{-l+1}& \label{B_l_line_1} \\
& y_{-l+j}' = -\frac{1}{2} t_{j+2} y_{l-1-j} + \frac{1}{2} t_{j+1} y_{l+1-j} - y_{-l+1+j} &  (1 \leq j \leq k)
\label{B_l_eqn_2} 
\end{align}
is equivalent to the single differential equation
\begin{eqnarray*}
 y_{l-k}^{(2(k+1)+1)} &=& \sum_{i=1}^{k+1} (-1)^{i-1} \left( (t_i y_{l-k}^{(k+1-i)})^{(k+2-i)} + (t_i y_{l-k}^{(k+2-i)})^{(k+1-i)}  \right)\\
 &&+ (-1)^{k+1} \left( t_{k+2} y_{l-k-1} + 2 y_{-l+k+1} \right).
\end{eqnarray*}
Letting $k=1$, the equations $y_{l-1}'=y_l$ and $y_l'=2 y_0$ imply $y_{l-1}''= 2y_0$. We differentiate
again and we plug in $\frac{1}{2} t_1 y_{l-1}' + y_{-l}$ for $y_0'$, where we made the substitution $y_l =y_{l-1}'$.  
We obtain $y_{l-1}^{(3)}=t_1 y_{l-1}' + 2y_{-l}$. Differentiating again and replacing $y_{-l}'$ by the right hand side of the last equation
in (\ref{B_l_line_1}) yields
 \begin{equation}
  y_{l-1}^{(4)} = (t_1 y_{l-1}')' - t_2 y_{l-1} + t_1 y_{l-1}'' - 2 y_{-l+1}, \label{B_l_eqn_3}
 \end{equation}
 where we substituted $y_0$ by $\frac{1}{2}y_{l-1}''$.
Differentiating Equation~(\ref{B_l_eqn_3}), replacing $ y_{-l+1}'$ by the right hand side of Equation~(\ref{B_l_eqn_2}) 
and plugging in $y_{l-1}'$ for $y_l$ yields
\begin{equation*}
  y_{l-1}^{(5)} = (t_1 y_{l-1}')'' + (t_1 y_{l-1}'')' - (t_2 y_{l-1}') -(t_2 y_{l-1})'  + t_3 y_{l-2} + 2 y_{-l+2}.
\end{equation*}
We note that in case $l=2$, we have to omit Equation~(\ref{B_l_eqn_2}) and consider instead $y_{-1}' = \frac{1}{2} t_2 y_2$.  
More precisely, differentiating Equation~(\ref{B_l_eqn_3}) and making the substitution implied by 
 $y_{-1}' = \frac{1}{2} t_2 y_2$ proves the lemma for $l=2$.\\
We come back to the case $l \geq 3$ and assume now $k>1$. Then the induction assumption yields
for the subsystem defined by the integer $k-1$ the differential equation
\begin{eqnarray*}
 y_{l-k+1}^{(2k+1)} &=& \sum_{i=1}^{k} (-1)^{i-1} \left( (t_i y_{l-k+1}^{(k-i)})^{(k+1-i)} + (t_i y_{l-k+1}^{(k+1-i)})^{(k-i)}  \right) \\
 &&+ (-1)^{k} \left( t_{k+1} y_{l-k} + 2 y_{-l+k} \right).
\end{eqnarray*}
The subsystem defined by the integer $k$ contains the additional differential equation $y_{l-k}'= y_{l-k+1}$. We replace now $ y_{l-k+1}$ by $y_{l-k}'$ and obtain
\begin{eqnarray*}
 y_{l-k}^{(2k+2)} &=& \sum_{i=1}^{k} (-1)^{i-1} \left( (t_i y_{l-k}^{(k+1-i)})^{(k+1-i)} + (t_i y_{l-k}^{(k+2-i)})^{(k-i)}  \right)\\
 &&+ (-1)^{k} \left( t_{k+1} y_{l-k} + 2 y_{-l+k} \right).
\end{eqnarray*}
Differentiating and replacing $ y_{-l+k}'$ by the right hand side of Equation~(\ref{B_l_eqn_2}) 
with $j=k$ proves the induction assumption.\\
Finally, if we consider the subsystem of the full system of equations defined by the integer $k=l-2$, we obtain the differential equation
\begin{eqnarray}\label{B_l_eqn_4}
 y_{2}^{(2l-1)} &=& \sum_{i=1}^{l-1} (-1)^{i-1} \left( (t_i y_{2}^{(l-1-i)})^{(l-i)} + (t_i y_{2}^{(l-i)})^{(l-1-i)}  \right)\\
 &&+ (-1)^{l-1} \left( t_{l} y_{1} + 2 y_{-1} \right). \nonumber
\end{eqnarray}
The first equation of the full system implies the substitution $y_2=y_1 '$ in Equation~(\ref{B_l_eqn_4}). We use this  
substitution also for the last equation of the full system and if we  differentiate and replace $y_{-1}'$ by the right hand side of this equation, 
it follows that the full system of equations is equivalent to the linear parameter differential equation of the lemma.
 \end{proof}
 \begin{theorem}
The linear parameter differential equation 
\begin{equation*}
L(y,\textbf{t})= y^{(2l+1)} - \sum\nolimits_{i=1}^{l} (-1)^{i-1} ((t_i y^{(l+1-i)})^{(l-i)}+(t_i y^{(l-i)})^{(l+1-i)})= 0
\end{equation*}
has differential Galois group $\mathrm{SO}_{2l+1}(C)$ over $F_1$.
\end{theorem}
 \begin{proof}
 The proof works as the proof of Theorem~\ref{theorem for SL_n}. Let 
 $A_{\mathrm{SO}_{2l+1}}(\boldsymbol{t})$ be as in Lemma~\ref{Lemma_SO_2l+1_equation}.
 To prove that the equation 
 $\partial(\boldsymbol{y})=A_{\mathrm{SO}_{2l+1}}(\boldsymbol{t})\boldsymbol{y}$ 
 specializes to a differential equation over $F_2$ with group $\mathrm{SO}_{2l+1}(C)$ 
 one uses Proposition~\ref{mod_Mitch_Singer}, Lemma~\ref{lemma_so2l+1_roots} 
 and Lemma~\ref{Prop_tans_lemma} (interchange the role of the positive and negative roots). 
 Then Theorem~\ref{specialization_bound} and Proposition~\ref{Prop_upper} imply that 
 the differential Galois group of 
 $\partial(\boldsymbol{y})=A_{\mathrm{SO}_{2l+1}}(\boldsymbol{t})\boldsymbol{y}$ is 
 $\mathrm{SO}_{2l+1}(C)$. Finally one applies Lemma~\ref{Lemma_SO_2l+1_equation}.
\end{proof}
\section{The equation for $\mathrm{SO}_{2l}(C)$}\label{chap10}
We denote the standard orthonormal basis of $\mathbb{R}^l$ with respect to the standard inner product 
$( \cdot , \cdot)$ by $\epsilon_1,\dots , \epsilon_l$. Let us assume for the rest of this section that $l \geq 3$.
Then by \cite[VI, Planche IV]{BourGroupesetAlgebresIV-VI} the 
vectors $\pm \epsilon_i \pm \epsilon_j$ with $1\leq i < j \leq l$ form a root system $\Phi$ of type $D_l$ and    
a basis $\Delta$ of $\Phi$ is given by the $l$ elements $\alpha_1= \epsilon_1 - \epsilon_2$, 
$\alpha_2 = \epsilon_2- \epsilon_3, \ \dots, \alpha_{l-1}=\epsilon_{l-1} -\epsilon_l$, $\alpha_l = \epsilon_{l-1} + \epsilon_l$. 
Then the positive roots can be expressed in terms of these basis elements in the following way:
\begin{align*}
 & \epsilon_i- \epsilon_j = \sum_{i \leq k < j}\alpha_k  &&  (1 \leq i < j \leq l), \\
 & \epsilon_i + \epsilon_j  = \sum_{i \leq k <j} \alpha_k + 2 \sum_{j \leq k < l-1} \alpha_k + \alpha_{l-1} + \alpha_l
 && ( 1 \leq i < j < l), \\ 
 & \epsilon_i + \epsilon_l = \sum_{i \leq k \leq l-2} \alpha_k + \alpha_l && (1\leq i < l).
\end{align*}
 The negative roots are obtained by simply interchanging all signs in the above expressions. 
 We take a Cartan decomposition of $\mathfrak{so}_{2l}$ from 
 \cite[VIII, \S 13.4]{BourGroupesetAlgebresVII-VIII}. Renumbering the rows 
 and columns of the matrices $E_{ij} \in C^{2l \times 2l}$ into 
 $1,\dots , l,-l, \dots ,-1$, 
 \cite[VIII, \S 13.4]{BourGroupesetAlgebresVII-VIII} yields that the matrices $H_i = E_{i,i} - E_{-i,-i}$ with $1\leq i \leq l$ generate 
 a Cartan subalgebra $\mathfrak{h}$  of $\mathfrak{so}_{2l}$ and, for the roots 
 $\alpha= \pm \epsilon_i \pm \epsilon_j$ with $1\leq i < j \leq l$ the root spaces $(\mathfrak{so}_{2l})_{\alpha}$ 
 with respect to $\mathfrak{h}$ are generated by the following matrices:
 \begin{align*}
&  X_{\epsilon_i - \epsilon_j} = E_{i,j} - E_{-j,-i}, && X_{-\epsilon_i + \epsilon_j} = - E_{j,i} + E_{-i,-j}, \ \\
&  X_{\epsilon_i + \epsilon_j} = E_{i,-j} - E_{j,-i}, && X_{-\epsilon_i - \epsilon_j} = - E_{-j,i} + E_{-i,j}.
 \end{align*}
Let now $\gamma_1 := \alpha_l$, $\gamma_2 := \alpha_{l-2} + \alpha_{l-1} + \alpha_l$ and for $3 \leq i \leq l-1$ let 
\begin{equation*}
\gamma_i := \alpha_{l-i} + 2 \sum_{j=l+1-i}^{l-2} \alpha_j + \alpha_{l-1} + \alpha_l.
\end{equation*}
 Finally, we define $\bar{\gamma}:=\alpha_1 + \dots + \alpha_{l-2} + \alpha_l$. 
 The following lemma shows that for these $l$ roots 
 Lemma~\ref{lemma_direct_sum2} and Lemma~\ref{Prop_tans_lemma} hold.
 \begin{lemma}\label{lemma_so2l_roots}
The $l$ roots $\gamma_1, \dots, \gamma_{l-1}$ and $\bar{\gamma}$ are complementary roots of $\Phi$. 
\end{lemma}
\begin{proof}
 From \cite[VI, Planche IV]{BourGroupesetAlgebresIV-VI} we obtain that the 
 exponents of the root system of type $D_l$ are 
 \[
 m_1=1, \ m_2=3, \ m_3=5, \dots , m_{l-2}=2l-5, \ m_{l-1}=2l-3 
 \]
 and $\bar{m}=l-1$, that is for $i=1,\dots, l-1$ we have the formula $m_i = 2i-1$. 
 Note that in case $l$ is even, the value $l-1$ occurs 
 twice, that is we have $\bar{m}=l-1=m_{l/2}$.
For $i=1,\dots, l-1$ one easily checks that the roots $\gamma_i$ satisfy $\mathrm{ht}(\gamma_i)=m_i$ 
and the root $\bar{\gamma}$ fulfills $\mathrm{ht}(\bar{\gamma})= \bar{m}$. 

To shorten notation we define for $\mathfrak{so}_{2l}$ and for a positive height $m$ of the root system $\Phi$ of type $D_l$ the set of vectors  
\begin{equation*}
\mathcal{B}(l,m) := \{ W_h \mid W_h \in \mathfrak{so}_{2l}^{(m)} \}.
\end{equation*}
Recall that $\mathcal{B}(l,m) $ is a basis of 
$\mathrm{ad}(A_0^-)(\mathfrak{u}^+)^{(m)}$ and in case $m$ is not an exponent it is 
even a basis of $\mathfrak{so}_{2l}^{(m)}$. 
In order to prove the claim we show that we can extend the basis $\mathcal{B}(l,m)$ 
to a basis of $\mathfrak{so}_{2l}^{(m)}$ by adjoining the respective vectors  
$X_{\gamma_i}$ and $X_{\bar{\gamma}}$. We use induction on $l$ where  
we distinguish between an odd and even $l$. We make the following inductive 
assumptions on $l \geq 4$: 
     If $l$ is odd, then the vectors of $\mathcal{B}(l,m_i)$ and 
    $ X_{\gamma_i}$ form a basis of $\mathfrak{so}_{2l}^{(m_i)}$ for $i=1,\dots, l-1$. 
    Further the vectors of $\mathcal{B}(l,\bar{m})$ and $X_{\bar{\gamma}}$ form a 
    basis of $\mathfrak{so}_{2l}^{(\bar{m})}$.
    If $l$ is even, then the vectors of $\mathcal{B}(l,m_i)$ and $ X_{\gamma_i}$ 
    form a basis of $\mathfrak{so}_{2l}^{(m_i)}$ for $i=1, \dots,l-1$ and $i \neq l/2$. 
    For $\bar{m}=m_{l/2}$ the vectors of $\mathcal{B}(l,\bar{m}) $ together with 
    $ X_{\bar{\gamma}}$ and $X_{\gamma_{l/2}}$ form a basis of 
    $\mathfrak{so}_{2l}^{(\bar{m})}$.
 
Let $l=4$. The exponents are $m_1=1$, $m_2=\bar{m}=3$ and $m_3=5$. 
We start with $m_1$. The roots of height two are 
$\beta_5 = \alpha_1 +  \alpha_2$, $\beta_6 = \alpha_2 + \alpha_3$ and 
$\beta_7 = \alpha_2 + \alpha_4$ and so $\mathcal{B}(4,1)$ consists of three elements 
$W_h =[A_0^-, X_{\beta_h}] $ with $h=5,6,7$. 
We represent each $W_h$ as a linear combination of basis elements corresponding to the 
simple roots. The coefficient of $X_{\alpha_s}$ in the corresponding representation is
nonzero if and only if there is a simple root $\alpha_j$ such that 
$\beta_h - \alpha_j=\alpha_s$. It is easy to check that the matrix whose 
columns represent the coordinates of $W_h$ and $X_{\alpha_4}$ is  
regular. Hence the vectors of $\mathcal{B}(4,1)$ and $ X_{\alpha_4}$ form a basis of 
$\mathfrak{so}_8^{(1)}$.
We consider now the exponent $m_2 = \bar{m} =3 $. The roots of height three are 
\[
\beta_8 = \alpha_1 + \alpha_2 + \alpha_3, \ \beta_9 = \alpha_1 + \alpha_2 + \alpha_4, \ 
\beta_{10} = \alpha_2 + \alpha_3 + \alpha_4
\]
and the only root of height four is 
$\beta_{11} = \alpha_1 + \alpha_2 + \alpha_3 + \alpha_4$. Thus $\mathcal{B}(4,3)$ 
consists only of the vector $W_{11}=[A_0^-, X_{\beta_{11}}]$. 
As above we represent $W_{11}$ as a linear combination of the basis elements 
$X_{\beta_h}$ of $\mathfrak{so}_{8}^{(3)}$ with $h=8,9,10$. 
All three coordinates of $W_{11}$ are nonzero and therefore the vector of 
$\mathcal{B}(4,3)$ together with $X_{\beta_{9}}$ and $ X_{\beta_{10}}$ 
form a basis of $\mathfrak{so}_8^{(3)}$.
In case of the exponent $m_3 =5$ the root 
$\gamma_{3}$ is the unique root of height five and 
so the vector $ X_{\gamma_{3}}$ forms trivially a basis 
of $\mathfrak{so}_8^{(5)}$. 

Let $l>4$ and assume that the claim is true for $l-1$. The Dynkin diagram shows that 
\[
\Phi'= \{ \alpha \in \Phi \mid n_{\alpha_1}(\alpha) =0 \}
\]
is a root subsystem of type $D_{l-1}$ and so we can identify $\mathfrak{so}_{2l-2}$
 inside $\mathfrak{so}_{2l}$ with the sualgebra generated by the vectors 
 $X_{\pm \alpha_i} \in \mathfrak{so}_{2l}$ for $i=2,\dots,l$. 
 From the shape of the roots of $\Phi$ we obtain that the set 
 $\Phi^+ \setminus \Phi'^+$ consists of the roots
 $\alpha_1+\dots + \alpha_k $ with $k=1,\dots,l-2$, and
 \[
\delta= \alpha_1+\dots +\alpha_{l-1} \quad \mathrm{and} \quad 
\bar{\gamma}=\alpha_1+\dots + \alpha_{l-2} + \alpha_l
\]
 as well as
 \[
 \alpha_1 + \dots + 2 \alpha_k + \dots + \alpha_{l-1} + \alpha_l
 \]
with $k=2,\dots,l-2$. For $k=1,\dots,2l-3$ one easily checks that if 
$k\neq l-1$ there is a 
unique root in $\Phi^+\setminus \Phi'^+$ of height $k$ which we denote by $\beta_k$
and if $k=l-1$ there are two roots  
of height $k$ in $\Phi^+\setminus \Phi'^+$ namely $\delta$ and $\bar{\gamma}$.
We will apply the 
 induction assumption to the system $\Phi'$ of rank $l-1$. For $i=1,\dots,l-2$ 
 we denote the corresponding 
 exponents and complementary roots by $m_i'$ and $\bar{m}'$ as well as by 
 $\gamma_i'$ and $\bar{\gamma}'$. It is easy to check that for $i=1,\dots,l-2$ we have 
 the equalities $m_i'=m_i$ and $\gamma_i'=\gamma_i$.

Let $l$ be odd. First we prove the claim for the exponents $m_i$ for $i=1,\dots,l-2$ 
with $i\neq (l-1)/2$. Since $l-1$ is even we obtain from the second part of the 
induction assumption that $\mathcal{B}(l-1,m_i')$ and $X_{\gamma_i'}$ form a basis of 
$\mathfrak{so}_{2l-2}^{(m_i')}$. For $r=(l-1)/2$ the 
exponent is $m_r=m_r'=\bar{m}'=l-2$ and since the exponents $m_i$ differ by two 
we conclude that $m_i\leq l-4$ for $i=1,\dots,r-1$ and 
$m_i \geq l$ for $i=r+1,\dots,l-2$. Hence in both cases $m_i$ and $m_i+1$ are not 
equal to $l-1$. As a consequence $\beta_{m_i}$ and $\beta_{m_i+1}$ are unique in 
$\Phi^+ \setminus \Phi'^+$ with height $m_i$ and $m_i+1$ respectively. This implies 
that we have to adjoin a single linear independent vector to the basis of $\mathfrak{so}_{2l-2}^{(m_i)}$ to obtain a basis of $\mathfrak{so}_{2l}^{(m_i)}$ 
and that 
\[
W_{\beta_{m_i+1}}=[A_0^-,X_{\beta_{m_i+1}}]
\]
together with the vectors of $\mathcal{B}(l-1,m_i)$ form precisely the set $\mathcal{B}(l,m_i)$.
Since 
$n_{\alpha_1}(\beta_{m_i})$ is nonzero we conclude that $\beta_{m_i +1}$ is the 
unique root in $\Phi^+$ such that $\beta_{m_i +1} -\alpha_j=\beta_{m_i}$ for 
some simple root $\alpha_j$. Thus $W_{\beta_{m_i+1}}$
is the unique basis element in $\mathcal{B}(l,m_i)$ which has a nonzero component in 
$(\mathfrak{so}_{2l})_{\beta_{m_i}}$ and so $X_{\gamma_i}$ and the vectors 
of $\mathcal{B}(l,m_i)$ form a basis of $\mathfrak{so}_{2l}^{(m_i)}$. It is left 
to prove the claim for the three remaining exponents $m_r=l-2$, $\bar{m}=l-1$ and 
$m_{l-1}=2l-3$. 

We consider the exponent $m_r=m_r'=\bar{m}'=l-2$. The second part of the induction 
assumption yields that the vectors of $\mathcal{B}(l-1,\bar{m}')$ together 
with $X_{\bar{\gamma}'}$ and $X_{\gamma_r'}$ form a basis of 
$\mathfrak{so}_{2l-2}^{(\bar{m}')}$. The root $\beta_{l-2}$ is the unique root 
in $\Phi^+ \setminus \Phi'^+$ of height $l-2$ and $\bar{\gamma}$ and $\delta$ are 
the only roots of height $l-1$ in $\Phi^+\setminus \Phi'^+$. We conclude that
we have to adjoin a single linear independent 
vector to the above basis to get a basis of $\mathfrak{so}_{2l}^{(m_r)}$ and that 
the vectors of $\mathcal{B}(l-1,\bar{m}')$ together with 
\[
W_{\bar{\gamma}}=[A_0^-, X_{\bar{\gamma}}] \quad \mathrm{and} \quad
W_{\delta}=[A_0^-, X_{\delta}] 
\]
form the set $\mathcal{B}(l,m_r)$. Since $n_{\alpha_1}(\beta_{l-2})$ is nonzero 
the roots $\delta$ and $\bar{\gamma}$ are the only roots of $\Phi$ from which 
we can subtract a simple root to obtain $\beta_{l-2}$. It follows 
that $W_{\bar{\gamma}}$ and $W_{\delta}$
are the only elements in $\mathcal{B}(l,m_r)$ which have a nonzero component in 
$(\mathfrak{so}_{2l})_{\beta_{l-2}}$.
Further it is possible to subtract a simple root from $\bar{\gamma}$ to obtain 
$\bar{\gamma}'$ but not from $\delta$ and so from the two vectors 
$W_{\bar{\gamma}}$ and $W_{\delta}$ only $W_{\bar{\gamma}}$ has a 
nonzero component in $(\mathfrak{so}_{2l-2})_{\bar{\gamma}'}$. 
We conclude that adjoining $W_{\delta}$ to the above basis 
and then substituting $X_{\bar{\gamma}'}$ by $W_{\bar{\gamma}}$ yields 
a basis of $\mathfrak{so}_{2l}^{(m_r)}$, that is the basis is formed by the vectors 
of $\mathcal{B}(l,m_r)$ and $X_{\gamma_r}$. 

Next we consider $\bar{m}=l-1$. Since $l-1$ is not an exponent of $\Phi'$ we
have that $\mathcal{B}(l-1,\bar{m})$ is a basis of $\mathfrak{so}_{2l-2}^{(\bar{m})}$.
The roots $\delta$ and $\bar{\gamma}$ are the only roots in $\Phi^+ \setminus \Phi'^+$
of height $l-1$ and $\beta_l$ is the unique root of height $l$ in $\Phi^+ \setminus \Phi'^+$. We conclude that we have to adjoin two linear independent vectors to 
$\mathcal{B}(l-1,\bar{m})$ to obtain a basis of $\mathfrak{so}_{2l}^{(\bar{m})}$ 
and that the vectors of $\mathcal{B}(l-1,\bar{m})$ together with 
\[
W_{\beta_l} = [A_0^-,X_{\beta_l}]
\]
form the set $\mathcal{B}(l,\bar{m})$. Since the coefficient of $\alpha_1$ in the 
representation of $\delta$ and $\bar{\gamma}$ is nonzero $\beta_l$ is the 
unique root of $\Phi$ from which it is possible to subtract a simple root to obtain 
$\delta$ or $\bar{\gamma}$. Hence the vector $W_{\beta_l}$
is the unique element of $\mathcal{B}(l,\bar{m})$ with a nonzero component in the 
root spaces of $\delta$ and $\bar{\gamma}$. It follows that 
$\mathcal{B}(l,\bar{m})$ together 
with $X_{\bar{\gamma}}$ forms a basis of $\mathfrak{so}_{2l}^{(\bar{m})}$.

In case for the exponent $m_{l-1}=2l-3$ the root $\gamma_{l-1}$ is the unique 
root of maximal height $m_{l-1}$ and so the vector $X_{\gamma_{l-1}}$ clearly forms a 
basis of $\mathfrak{so}_{2l}^{(m_{l-1})}$.

Let $l$ be even. Similar as above we first prove the claim for the exponents 
$m_i$ for $i=1,\dots,l-2$ with $i\neq l/2$. Applying the induction assumption
to the root system $\Phi'$ of odd rank $l-1$ we obtain that the vectors of  $\mathcal{B}(l-1,m_i')$
and $X_{\gamma_i'}$ form a basis of $\mathfrak{so}_{2l-2}^{(m_i')}$. For $r=l/2$ we have $m_r=\bar{m}=l-1$ and since the exponents $m_i$ differ by two it follows that 
$m_i \leq l-3$ for $i=1,\dots, r-1$ and $m_i \geq l+1$ for $i=r+1,\dots,l-2$. 
We conclude that $m_i$ and $m_i+1$ are not equal to $l-1$ and so $\beta_{m_i}$ and 
$\beta_{m_i+1}$ are the 
unique roots in $\Phi^+\setminus \Phi'^+$ of their respective heights. 
Hence we need to adjoin a single linear independent vector to the above basis to 
obtain a basis of $\mathfrak{so}_{2l}^{(m_i)}$ and the set formed by
\[
W_{\beta_{m_i+1}}=[A_0^-,X_{\beta_{m_i+1}}]
\] 
and the vectors of $\mathcal{B}(l-1,m_i)$ is precisely the set $\mathcal{B}(l,m_i)$. Since $\beta_{m_i+1}$ is the unique root in $\Phi$ such that  $\beta_{m_i+1}-\alpha_j=\beta_{m_i}$ for some simple root $\alpha_j$ we obtain that 
 $W_{\beta_{m_i+1}}$ is the unique vector in
 $\mathcal{B}(l,m_i)$ which has a nonzero component in the root space of 
 $\mathfrak{so}_{2l}$ for $\beta_{m_i}$ and so the vectors of $\mathcal{B}(l,m_i)$ and 
$X_{\gamma_i}$ form a basis of $\mathfrak{so}_{2l}^{(m_i)}$.
It is left to prove the claim for the two remaining exponents $m_{l/2}=\bar{m}=l-1$ and 
$m_{l-1}=2l-3$. 

We consider first the exponent $m_{l/2}=\bar{m}=m_{l/2}'=l-1$. The first part of the induction assumption applied to the 
system $\Phi'$ of odd rank $l-1$ yields for the exponent $m'_{l/2}$
that the vectors of $\mathcal{B}(l-1,m'_{l/2})$ and $X_{\gamma_{l/2}'}$ form a 
basis of $\mathfrak{so}_{2l-2}^{(m'_{l/2})}$. Since $\delta$ and $\bar{\gamma}$ are
the only roots of height $l-1$ in $\Phi^+ \setminus \Phi'^+$ we have to adjoin 
two linear independent vectors to the above basis to obtain a basis of 
$\mathfrak{so}_{2l}^{(\bar{m})}$.
The root $\beta_l$ is the unique root in $\Phi^+ \setminus \Phi'^+$ of height $l$ 
and so the vectors of $\mathcal{B}(l-1,m'_{l/2})$ and 
\[
W_{\beta_l}=[A_0^-,X_{\beta_l}] 
\]
are precisely the elements of the set 
$\mathcal{B}(l,\bar{m})$. Since among the vectors of $\mathcal{B}(l,\bar{m})$ only 
$W_{\beta_l}$ has nonzero components in the root spaces of $\mathfrak{so}_{2l}$ 
for $\delta$ and $\bar{\gamma}$ the vectors of
$\mathcal{B}(l,\bar{m})$ as well as $X_{\gamma_{l/2}}$ and $X_{\bar{\gamma}}$
form a basis of $\mathfrak{so}_{2l}^{(\bar{m})}$.

In case of the exponent $m_{l-1}=2l-3$ the root $\gamma_{l-1}$ is the unique 
root of height $m_{l-1}$ and so the vector $X_{\gamma_{l-1}}$ trivially forms a 
basis of $\mathfrak{so}_{2l}^{(m_{l-1})}$.
\end{proof} 
\begin{lemma}
\label{Lemma_SO_2l_equation}
The matrix parameter differential equation 
$\partial(\textit{\textbf{y}})=A_{\mathrm{SO}_{2l}}(\textit{\textbf{t}})\textit{\textbf{y}}$ over $F_1$, where 
\begin{equation*}
 A_{\mathrm{SO}_{2l}}(\textit{\textbf{t}}): = A_0^+ - t_1 X_{-\bar{\gamma}} + \sum_{i=1}^{l-1} - t_{i+1} X_{-\gamma_i}
\end{equation*}
and the roots $\gamma_i$ and $\bar{\gamma}$ are as in Lemma~\ref{lemma_so2l_roots},
is equivalent to the linear parameter differential equation
\begin{eqnarray}
L(y,t_1,...,t_l)&=& y^{(2l)}- 2 \sum_{i=3}^{l} (-1)^{i} ( (t_i y^{(l-i)})^{(l+2-i)} + (t_i y^{(l+1-i)})^{(l+1-i)} ) \nonumber \\
&&- (t_2 y^{(l-2)}+ t_1 y)^{(l)} - ((-1)^l t_1 z_1 + z_2) - \sum_{i=0}^{l-2} (t_2^{(l-2-i)} z_1)^{(i)}.
\nonumber
\end{eqnarray}
The coefficients $z_1$ and $z_2$ are given by 
\begin{eqnarray}
z_1 &=& y^{(l)}- t_2 y^{(l-2)} - t_1 y \nonumber \\ 
z_2 &=& \frac{(t_2^{(l-2)}+(-1)^{l-2}t_1)^{(1)}}{t_2^{(l-2)}+(-1)^{l-2}t_1} \cdot \left(  y^{(2l-1)}- 2 \sum_{i=3}^{l} (-1)^{i} ( (t_i y^{(l-i)})^{(l+1-i)}  \right. \nonumber \\
&& + (t_i y^{(l+1-i)})^{(l-i)} ) - \left.(t_2 y^{(l-2)} + t_1 y)^{(l-1)}- \sum_{i=0}^{l-3} (t_2^{(l-3-i)} z_1)^{(i)} \right).
\nonumber
\end{eqnarray}
\end{lemma}
\begin{proof}
Using the representation of $\mathfrak{so}_{2l}$ given at the beginning of this section, we obtain for $\textit{\textbf{y}}=(y_1, \dots, y_{2l})^{tr}$
that the matrix differential equation $\partial(\textit{\textbf{y}})=A_{\mathrm{SO}_{2l}}(\textit{\textbf{t}})\textit{\textbf{y}}$ is 
equivalent to the following system of linear differential equations:
\begin{align}
y'_{k} &= y_{k+1}  & \; (1 \leq k \leq l-2) \tag{k} \label{eqn_diff_main_eq_SO_k} \nonumber \\
y'_{l-1} &=y_l  + y_{l+1}  &\tag{$\ell$-1} \label{eqn_diff_main_eq_SO_l-1} \nonumber \\
y'_{l} &= - y_{l+2} & \tag{$\ell$} \label{eqn_diff_main_eq_SO_l} \nonumber \\
y'_{l+1} &= t_1 y_1 + t_2 y_{l-1} - y_{l+2} &\tag{$\ell$+1} \label{eqn_diff_main_eq_SO_l+1} \nonumber \\
y'_{l+2} &= t_3 y_{l-2}-t_2 y_l - y_{l+3}  &\tag{$\ell$+2} \label{eqn_diff_main_eq_SO_l+2} \nonumber \\
y'_{l+k} &= t_{k+1} y_{l-k} - t_k y_{l-k+2}- y_{l+k+1}  & \; (3 \leq k \leq l-1) \tag{$\ell$+k} \label{eqn_diff_main_eq_SO_l+k} \nonumber \\
y'_{2l} &= -t_l y_{2}-t_1 y_{l}.  & \tag{2$\ell$} \label{eqn_diff_main_eq_SO_2l} \nonumber 
\end{align}
We show that $y_1$ is a cyclic vector. From the Equations~(1) - ($\ell$-2) we deduce that
\begin{equation}
y_1^{(i-1)}= y_i \quad \mathrm{for} \; 1 \leq i \leq l-1. \nonumber
\end{equation}
In particular, we have $y_1^{(l-2)}=y_{l-1}$. Differentiating $y_1^{(l-2)}=y_{l-1}$ and substituting $y'_{l-1}$ 
by the right hand side of Equation~(\ref{eqn_diff_main_eq_SO_l-1}) yields $y_1^{(l-1)}=y_l + y_{l+1}$. 
We differentiate the last expression.  
Using Equation~(\ref{eqn_diff_main_eq_SO_l}) and~(\ref{eqn_diff_main_eq_SO_l+1}) we obtain that  
\begin{equation}
y_1^{(l)}=t_1 y_1 + t_2 y_{l-1} - 2 y_{l+2}. \nonumber
\end{equation}
Thus, we have $y_1^{(l)}=t_1 y_1 + t_2 y_1^{(l-2)} - 2 y_{l+2}$ and if we solve for $- 2 y_{l+2}$, we get that  
\begin{equation*}
-2 y_{l+2}=y_1^{(l)} - t_1 y_1 - t_2 y_1^{(l-2)}=:z_1
\end{equation*} 
We differentiate $y_1^{(l)}=t_1 y_1 + t_2 y_1^{(l-2)} - 2 y_{l+2}$ and substitute $y'_{l+2}$ by the right hand side of Equation~(\ref{eqn_diff_main_eq_SO_l+2}). 
This yields
\begin{equation}
 y_1^{(l+1)} = (t_1 y_1 + t_2 y_1^{(l-2)})' - 2 t_3 y_1^{(l-3)} + 2 t_2 y_l + 2 y_{l+3}.
 \tag{*} \label{eqn_diff_main_eq_SO_*}
\end{equation}
In case $l\geq 4$, we prove the following claim: For $1\leq k \leq l-3$ the system
\begin{align}
y'_{l+3} &= t_4 y_{l-3}-t_3 y_{l-1} - y_{l+4}  \tag{1} \label{eqn_diff_eq_SO_1} \nonumber \\
&\; \; \vdots \nonumber \\
y'_{l+2+k} &= t_{k+3} y_{l-k-2} - t_{k+2} y_{l-k}- y_{l+k+3} \tag{k} \label{eqn_diff_eq_SO_k} \nonumber
\end{align}
together with the equations
\begin{align} 
y_1^{(l+1)} &= \bar{z}' - 2 t_3 y_1^{(l-3)} + 2 t_2 y_l + 2 y_{l+3} &\tag{a} \label{eqn_diff_eq_SO_A}\nonumber \\
y_1^{(i-1)}&=  y_i & \ (l-k-2 \leq i \leq l-1) \tag{b} \label{eqn_diff_eq_SO_B} \nonumber \\
y'_l &= - y_{l+2} &\tag{c} \label{eqn_diff_eq_SO_C} \nonumber \\
-2 y_{l+2} &= y_1^{(l)} - t_1 y_1 - t_2 y_1^{(l-2)}=:z_1 & \tag{d} \label{eqn_diff_eq_SO_D} \nonumber 
\end{align}
yields the differential equation
\begin{eqnarray}
y_1^{(l+k+1)} &=& \bar{z}^{(k+1)} + 2 \sum_{i=3}^{k+2} (-1)^i ( (t_i y_1^{(l-i)})^{(k+3-i)}+(t_i y_1^{(l-i+1)})^{(k+2-i)} ) \nonumber \\
&& + 2 ( (-1)^{k+1} t_{k+3}y_1^{(l-k-3)} + (-1)^k y_{l+k+3} + t_2^k y_l) + \sum_{i=0}^{k-1} (t_2^{(k-i-1)} z_1)^{(i)}, \nonumber
\end{eqnarray} 
where we denote in the following by $\bar{z}^{(m)}$ the term $(t_1 y_1 + t_2 y_1^{(l-2)})^{(m)}$ for $m \in \mathbb{N}$. 
The proof is done by induction on $1\leq k \leq l-3$.\\
Let $k=1$. We differentiate $y_1^{(l+1)} = \bar{z}' - 2 t_3 y_1^{(l-3)} + 2 t_2 y_l + 2 y_{l+3}$ and substitute $y'_{l+3}$ 
by the right hand side of Equation~(\ref{eqn_diff_eq_SO_1}). We obtain the equation 
\begin{eqnarray}
y_1^{(l+2)} &=& \bar{z}^{(2)} - 2 (t_3 y_1^{(l-3)})^{(1)}-2 t_3 y_{l-1} + 2 t_4 y_{l-3} + 2 t'_2 y_l + 2 t_2 y'_l - 2 y_{l+4} .\nonumber
\end{eqnarray}
Using the Equations~(\ref{eqn_diff_eq_SO_B}) and (\ref{eqn_diff_eq_SO_C}) for the substitution of $y_{l-1}$, $y_{l-3}$ and $y_l '$ and afterwards
Equation~(\ref{eqn_diff_eq_SO_D}) for the substitution of $y_{l+2}$, we get that 
\begin{eqnarray}
y_1^{(l+2)} 
&=& \bar{z}^{(2)} - 2 (t_3 y_1^{(l-3)})^{(1)} - 2 t_3 y_1^{(l-2)} + 2 t_4 y_1^{(l-4)} +  2 t'_2 y_l + t_2 z_1 -2y_{l+4} .\nonumber 
\end{eqnarray}
Now let $1< k \leq l-3$. For $k-1$ we obtain a subsystem of the above system formed by the corresponding equations (1)-(k-1) and (a), (b'), (c) and (d).
Then the induction assumption yields for this subsystem the differential equation
\begin{eqnarray*}
y_1^{(l+k)} &=& \bar{z}^{(k)} + 2 \sum_{i=3}^{k+1} (-1)^i ( (t_i y_1^{(l-i)})^{(k+2-i)} + (t_i y_1^{(l-i+1)})^{(k+1-i)} ) \\
&& + 2 ( (-1)^k t_{k+2} y_1^{(l-k-2)}  + (-1)^{k-1} y_{l+k+2} + t_2^{(k-1)} y_l) + \sum_{i=0}^{k-2} (t_2^{(k-2-i)} z_1)^{(i)} . 
\end{eqnarray*}
We differentiate this equation and substitute $y'_{l+k+2}$ by the right hand side of 
Equation~(k-1). It follows that 
\begin{eqnarray}
y_1^{(l+k+1)} &=& \bar{z}^{(k+1)} + 2 \sum_{i=3}^{k+1} (-1)^i ( (t_i y_1^{(l-i)})^{(k+3-i)} + (t_i y_1^{(l-i+1)})^{(k+2-i)} ) \nonumber \\
&& + 2 ( (-1)^k (t_{k+2} y_1^{(l-k-2)})'  + (-1)^{k-1} (t_{k+3} y_{l-k-2}- t_{k+2} y_{l-k}- y_{l+k+3}) \nonumber \\ 
&& + t_2^{(k)} y_l + t_2^{(k-1)} y'_l ) + \sum_{i=0}^{k-2} (t_2^{(k-2-i)} z_1)^{(i+1)}. \nonumber
\end{eqnarray}
The equations in~(\ref{eqn_diff_eq_SO_B}) show that we can replace $y_{l-k-2}$ by $y_1^{l-k-3}$ and
$y_{l-k}$ by $y_1^{l-k-1}$ in the above expression. 
The Equations~(\ref{eqn_diff_eq_SO_C}) and (\ref{eqn_diff_eq_SO_D}) imply 
that $2 y'_l = -2 y_{l+2}=z_1$.   
Using these substitutions, the above equation simplifies to
\begin{eqnarray*}
y_1^{(l+k+1)} 
&=& \bar{z}^{(k+1)} + 2 \sum_{i=3}^{k+2} (-1)^i ( (t_i y_1^{(l-i)})^{(k+3-i)} + (t_i y_1^{(l-i+1)})^{(k+2-i)} ) \\
&& + 2 ( (-1)^{k+1} t_{k+3} y_1^{l-k-3} + (-1)^{k} y_{l+k+3} + t_2^{(k)} y_l) + \sum_{i=0}^{k-1} (t_2^{(k-1-i)} z_1)^{(i)}    
\end{eqnarray*}
and the induction is complete.\\  
Now, in case $l \geq 4$ the claim yields for $k=l-3$ the  differential equation
\begin{eqnarray}
y_1^{(2l-2)} &=& \bar{z}^{(l-2)} + 2 \sum_{i=3}^{l-1} (-1)^i ( (t_i y_1^{(l-i)})^{(l-i)} + (t_i y_1^{(l-i+1)})^{(l-1-i)} )
\nonumber \\
&& + 2 ( (-1)^{l-2} t_{l} y_1 + (-1)^{l-3} y_{2l} + t_2^{(l-3)} y_l) + \sum_{i=0}^{l-4} (t_2^{(l-4-i)} z_1)^{(i)}. \nonumber 
\end{eqnarray}
Note that for $l=3$ this equation is equal to Equation~(\ref{eqn_diff_main_eq_SO_*}) and the proof in case $l=3$ continues here.
We differentiate the last equation and in the new expression we plug in 
 $-t_ly_1' - t_1y_l$ for $y'_{2l}$ (see Equation~(\ref{eqn_diff_main_eq_SO_2l})). 
Afterwards, we simplify the so obtained equation. Summing up, we get that  
\begin{gather*}
y_1^{(2l-1)} = \bar{z}^{(l-1)} + 
2 \sum_{i=3}^{l} (-1)^i ( (t_i y_1^{(l-i)})^{(l-i+1)} + (t_i y_1^{(l-i+1)})^{(l-i)} )\\
+ 2 ( (-1)^{l-2} t_1+ t_2^{(l-2)}) y_l + \sum_{i=0}^{l-3} (t_2^{(l-3-i)} z_1)^{(i)}.  
\end{gather*} 
In order to obtain a suitable expression of $y_l$ in terms of derivatives of $y_1$ for the last step, 
we solve the above equation for $2y_l$ and multiply with $((-1)^{l-2}t_1 + t_2^{(l-2)})'$, i.e.~we obtain the
following relation:
\begin{gather*}
2 ((-1)^{l-2}t_1 + t_2^{(l-2)})' y_l = \frac{((-1)^{l-2}t_1 + t_2^{(l-2)})'}{(-1)^{l-2} t_1+ t_2^{(l-2)}} \cdot \big( y_1^{(2l-1)} - \bar{z}^{(l-1)}  \\
- 2 \sum_{i=3}^{l} (-1)^i ( (t_i y_1^{(l-i)})^{(l-i+1)} + (t_i y_1^{(l-i+1)})^{(l-i)} \big) 
- \sum_{i=0}^{l-3} (t_2^{(l-3-i)} z_1)^{(i)} ) 
=: z_2 . 
\end{gather*}
Finally, we differentiate the penultimate equation and replace $2 ((-1)^{l-2}t_1 + t_2^{(l-2)})' y_l$ by $z_2$ and
$2 y'_l $ by $z_1$ (see above). This yields the equation in the assertion of the lemma.
\end{proof} 
\begin{theorem}
The linear parameter differential equation
  \begin{eqnarray*}
L(y,\textbf{t})&=&  y^{(2l)} -2 \sum\nolimits_{i=3}^{l} (-1)^{i} ((t_i y^{(l-i)})^{(l+2-i)}+(t_i y^{(l+1-i)})^{(l+1-i)})\\
&&  - (t_2 y^{(l-2)} + t_1 y)^{(l)} - ((-1)^l t_1 z_1 + z_2) -\sum\nolimits_{i=0}^{l-2} (t_2^{l-2-i} z_1)^{(i)} =0
\end{eqnarray*}
has $\mathrm{SO}_{2l}(C)$ as differential Galois group over $F_1$ where the functions $z_1$ and $z_2$ are as in 
Lemma~\ref{Lemma_SO_2l_equation}.
\end{theorem}
\begin{proof}
The proof is very similar to the proof of Theorem~\ref{theorem for SL_n}. 
Let $A_{\mathrm{SO}_{2l}}(\textit{\textbf{t}})$ be the matrix of Lemma~\ref{Lemma_SO_2l_equation}.
One uses Proposition~\ref{mod_Mitch_Singer}, Lemma~\ref{lemma_so2l_roots} 
and Lemma~\ref{Prop_tans_lemma} (interchange the role of the 
positive and negative roots) to show that 
$\partial(\boldsymbol{y})=A_{\mathrm{SO}_{2l}}(\boldsymbol{t}) \boldsymbol{y}$ 
specializes to an equation over $F_2$ with group $\mathrm{SO}_{2l}(C)$. 
It follows then from Proposition~\ref{Prop_upper} and 
Theorem~\ref{specialization_bound} that the differential Galois group of 
$\partial(\boldsymbol{y})=A_{\mathrm{SO}_{2l}}(\boldsymbol{t}) \boldsymbol{y}$ 
is $\mathrm{SO}_{2l}(C)$. To complete the proof one applies Lemma~\ref{Lemma_SO_2l_equation}.
\end{proof}
\section{The equation for $\mathrm{G}_2$}\label{chap11}
Let $\epsilon_1, \ \epsilon_2, \ \epsilon_3$ be the standard orthonormal basis of 
$\mathbb{R}^3$ with respect to the standard inner product $(\cdot, \cdot)$. Then
from \cite[VI, Planche IX]{BourGroupesetAlgebresIV-VI}
we obtain that the root system $\Phi$ of type $G_2$ consists 
of the twelve vectors 
$\pm (\epsilon_1 - \epsilon_2)$, $\pm (\epsilon_1 - \epsilon_3)$, $\pm (\epsilon_2 - \epsilon_3)$, 
$\pm (2 \epsilon_1 - \epsilon_2 -\epsilon_3)$, $\pm (2 \epsilon_2 - \epsilon_1 -\epsilon_3)$ and 
$\pm (2 \epsilon_3 - \epsilon_1 -\epsilon_2)$. Further, as a basis $\Delta$ of 
$\Phi$ we can take the two vectors
$\alpha_1 = \epsilon_1 - \epsilon_2$ and 
$\alpha_2 =-2 \epsilon_1 + \epsilon_2 + \epsilon_3$. 
Then, with respect to $\Delta$ the roots in $\Phi$ have the following shapes: 
\begin{align*}
&\pm \alpha_1 = \pm (\epsilon_1 - \epsilon_2), &
&\pm \alpha_2 =\pm (-2 \epsilon_1 + \epsilon_2 +\epsilon_3),\\
&\pm (\alpha_1 + \alpha_2) = \pm (- \epsilon_1 + \epsilon_3),&
&\pm (2 \alpha_1 + \alpha_2) = \pm (-\epsilon_2 + \epsilon_3),\\
&\pm(3 \alpha_1 + \alpha_2) = \pm (-2 \epsilon_2 + \epsilon_1 +\epsilon_3),&
& \pm (3 \alpha_1 + 2 \alpha_2) = \pm (2 \epsilon_3 - \epsilon_1 -\epsilon_2). 
\end{align*}  
In \cite[Chapter 19.3]{HumLie} J.~Humphreys presents an irreducible representation of the Lie algebra $\mathfrak{g}_2$ of type $G_2$ 
as a subalgebra of $\mathfrak{gl}_7$.
Following his argumentation, we obtain that a Cartan subalgebra $\mathfrak{h}$ of $\mathfrak{g}_2$ is generated by the two diagonal matrices   
\begin{equation*}
 H_1=- E_{2,2} + E_{5,5}+2 E_{3,3} -2 E_{6,6} - E_{4,4} + E_{7,7}, \ 
 H_2=  E_{3,3} -E_{6,6} - E_{4,4} + E_{7,7}
\end{equation*}
and that with respect to $\mathfrak{h}$ we can choose the root vectors corresponding to the positive roots as 
\begin{align*}
& X_{\alpha_1} \quad \ \ \, = \sqrt{2}(E_{1,3} - E_{6,1}) + (E_{2,7}- E_{4,5}), && 
X_{\alpha_2} \quad \quad \  = E_{3,4}- E_{7,6}, \\
& X_{\alpha_1+ \alpha_2} \ = \sqrt{2}(E_{1,4} - E_{7,1}) - (E_{2,6}- E_{3,5}),  && X_{3\alpha_1 + \alpha_2} \  = E_{2,3}- E_{6,5}, \\
& X_{2\alpha_1+ \alpha_2} = \sqrt{2}(E_{2,1} - E_{1,5}) - (E_{7,3}- E_{6,4}), &&X_{3\alpha_1+ 2\alpha_2} = E_{2,4}- E_{7,5}.
\end{align*}
 The root vectors corresponding to the negative roots are 
\begin{align*}
& X_{-\alpha_1}\quad \ \ \ = X_{\alpha_1}^{tr}, 
&& X_{-\alpha_2} \quad \ \ \ = X_{\alpha_2}^{tr} ,
&& X_{-\alpha_1-\alpha_2} \ \ = X_{\alpha_1 + \alpha_2}^{tr}, \\
& X_{-2\alpha_1-\alpha_2}= X_{2\alpha_1 + \alpha_2}^{tr}, 
&& X_{-3\alpha_1 - \alpha_2} = X_{3 \alpha_1 + \alpha_2}^{tr}, 
&& X_{-3\alpha_1-2\alpha_2}= X_{3\alpha_1 + 2\alpha_2}^{tr}.
\end{align*}
In \cite[Chapter 19.3]{HumLie} the corresponding matrices $g_k$, $g_{i,j}$ are not assigned to the roots $\alpha \in \Phi$. To obtain such an assignment
as above one can use the relations (1)-(5) given there.
Summing up, we have a Cartan decomposition 
\begin{equation*}
 \mathfrak{g}_2 = \mathfrak{h} \oplus \bigoplus_{\alpha \in \Phi} (\mathfrak{g}_2)_{\alpha} = \mathfrak{h} \oplus \bigoplus_{\alpha \in \Phi} \langle X_{\alpha} \rangle 
\end{equation*}
given by the explicit matrices $X_{\alpha}$ from above for $\alpha \in \Phi$.
\begin{lemma}\label{parameter_roots_g2}
The two roots $\gamma_1 = \alpha_2$ and $\gamma_2 = 3\alpha_1 + 2 \alpha_2$ of the root system $\Phi$ of 
type $G_2$ are complementary roots.
\end{lemma}
\begin{proof}
The exponents of the root system of type $G_2$ are $m_1 = 1$ and $m_2 = 5$ 
(see \cite[VI, Planche IX]{BourGroupesetAlgebresIV-VI}). Thus the roots $\gamma_i$ 
 satisfy $\mathrm{ht}(\gamma_i)=m_i$ with $i=1,2$. We have that $\alpha_1 + \alpha_2$ 
 is the only root in $\Phi^+$ of height two and therefore 
 $W_3=[A_0^-,X_{\alpha_1 + \alpha_2}]$ forms a basis of 
 \[
 \mathrm{ad}(A_0^-)(\mathfrak{u}^+) \cap \mathfrak{g}_{2}^{(m_1)}.
 \]
 Since $(\alpha_1 + \alpha_2)- \alpha_2= \alpha_1$, $W_3$ has a non-zero component 
 in the root space $(\mathfrak{g}_{2})_{\alpha_1}$. 
 We conclude that the set $\{W_3, X_{\alpha_2} \}$ is a basis of 
 $\mathfrak{g}_{2}^{(m_1)}$. Finally, since the root $\gamma_2$
 is the maximal root, it follows that $X_{\gamma_2}$ forms a basis of 
 $\mathfrak{g}_{2}^{(m_2)}$.
\end{proof}
\begin{lemma}\label{linear_equation_G2}
The matrix parameter differential equation $\partial(\textit{\textbf{y}})=A_{G_2}(t_1,t_2) \textit{\textbf{y}}$ 
over $C\langle t_1,t_2 \rangle$, where  
\begin{equation*}
A_{G_2}(t_1,t_2):=A_0^+ +t_1 X_{-\gamma_1} + t_2 X_{-\gamma_2},
\end{equation*}
is equivalent to the linear parameter differential equation
 \begin{equation*}
y^{(7)} = 2 t_2 y' + 2 (t_2 y )' + ( t_1 y^{(4)})' + (t_1 y')^{(4)} - (t_1 (t_1 y')')'.
\end{equation*}
\end{lemma}
\begin{proof}
With respect to the above representation of $\mathfrak{g}_2$ the matrix differential equation $\partial(\textit{\textbf{y}})=A_{G_2}(t_1,t_2) \textit{\textbf{y}}$ 
is equivalent to the following system of differential equations:
\begin{align*}
& y_1^{(1)}  =  \sqrt{2} y_3 ,&&
y_2^{(1)}  =  y_7, &&
y_3^{(1)}  =  y_4, &&
y_4^{(1)}  = t_2 y_2 + t_1 y_3 -y_5\\
& y_5^{(1)}  =  -t_2 y_7, &&
y_6^{(1)}  =  -\sqrt{2} y_1 - t_1 y_7, &&
y_7^{(1)}  =  - y_6.
\end{align*}
We take $y_2$ as a cyclic vector. 
More precisely, we differentiate successively the equation $y_2^{(1)}  =  y_7$ 
and using the above equations we substitute accordingly until we get 
an expression only in derivatives of $y_2$.
After the first step we obtain $y_2^{(2)}  =  -y_6$, where we substituted $y_7^{(1)}$ by $-y_6$. 
Differentiating this equation and using the relation
\begin{equation*}
 y_6^{(1)}  =  -\sqrt{2} y_1 - t_1 y_7   =  -\sqrt{2} y_1 - t_1 y_2^{(1)}
\end{equation*}
we get $y_2^{(3)}  =  \sqrt{2} y_1 +t_1 y_2^{(1)}$. Differentiating this expression and plugging in 
$ y_1^{(1)}  =  \sqrt{2} y_3$ shows that $y_2^{(4)}  =  (t_1 y_2^{(1)})^{(1)}+ 2 y_3$. 
If we differentiate again and substitute $y_3^{(1)} $ by $y_4$, we obtain 
$y_2^{(5)}  =  (t_1 y_2^{(1)})^{(2)} + 2 y_4$. Differentiating the last equation yields 
\begin{equation*}
 y_2^{(6)} =  (t_1 y_2^{(1)})^{(3)} + 2 t_2 y_2 +  t_1 y_2^{(4)}-  t_1 (t_1 y_2^{(1)})^{(1)} - 2 y_5,
\end{equation*}
where we used $y_4^{(1)}  = t_2 y_2 + t_1 y_3 -y_5$ and $2 y_3 = y_2^{(4)} - (t_1 y_2^{(1)})^{(1)}$. 
With the relation 
$y_5^{(1)}  =  -t_2 y_7 =-t_2 y_2^{(1)}$ we finally get that
\begin{equation*}
 y_2^{(7)}  =  (t_1 y_2^{(1)})^{(4)} + 2 (t_2 y_2)^{(1)} + (t_1 y_2^{(4)})^{(1)} - (t_1 (t_1 y_2^{(1)})^{(1)})^{(1)}
  + 2t_2 y_2^{(1)}.
\end{equation*}
\end{proof}
\begin{theorem}
The differential Galois group of the linear parameter differential equation 
\begin{equation*}
L(y,t_1,t_2)= y^{(7)} - 2 t_2 y' - 2 (t_2 y )' - ( t_1 y^{(4)})' - (t_1 y')^{(4)} + (t_1(t_1 y')')'=0
\end{equation*}
is $\mathrm{G}_2(C)$ over $C\langle t_1,   t_2 \rangle$.
\end{theorem}
\begin{proof}
The proof works as the proof of Theorem~\ref{theorem for SL_n}. 
Let $A_{\mathrm{G}_2}(t_1,t_2)$ be
the matrix of Lemma~\ref{linear_equation_G2}. Using Proposition~\ref{mod_Mitch_Singer}, 
Lemma~\ref{parameter_roots_g2} and Lemma~\ref{Prop_tans_lemma} (interchange the role 
of the positive and negative roots) one shows that there exists a 
$C\{ t_1,t_2 \}$-specialization 
of the differential equation 
$\partial(\boldsymbol{y})=A_{\mathrm{G}_2}(t_1,t_2) \boldsymbol{y}$ to an
equation over $F_2$ with differential Galois group $\mathrm{G}_2(C)$.
It follows then from Theorem~\ref{specialization_bound} and 
Proposition~\ref{Prop_upper} that the differential Galois group of 
$\partial(\boldsymbol{y})=A_{\mathrm{G}_2}(t_1,t_2) \boldsymbol{y}$ is 
$\mathrm{G}_2(C)$. Finally, one applies Lemma~\ref{linear_equation_G2}.
\end{proof}


\bibliographystyle{plain}
\bibliography{main}

\begin{thebibliography}{10}

\bibitem{BourGroupesetAlgebresIV-VI}
N.~Bourbaki.
\newblock {\em \'{E}l\'{e}ments de math\'{e}matique. {F}asc. {XXXIV}. {G}roupes
  et alg\`ebres de {L}ie. {C}hapitre {IV}: {G}roupes de {C}oxeter et syst\`emes
  de {T}its. {C}hapitre {V}: {G}roupes engendr\'{e}s par des r\'{e}flexions.
  {C}hapitre {VI}: syst\`emes de racines}.
\newblock Actualit\'{e}s Scientifiques et Industrielles, No. 1337. Hermann,
  Paris, 1968.

\bibitem{BourGroupesetAlgebresVII-VIII}
N.~Bourbaki.
\newblock {\em \'{E}l\'{e}ments de math\'{e}matique. {F}asc. {XXXVIII}:
  {G}roupes et alg\`ebres de {L}ie. {C}hapitre {VII}: {S}ous-alg\`ebres de
  {C}artan, \'{e}l\'{e}ments r\'{e}guliers. {C}hapitre {VIII}: {A}lg\`ebres de
  {L}ie semi-simples d\'{e}ploy\'{e}es}.
\newblock Actualit\'{e}s Scientifiques et Industrielles, No. 1364. Hermann,
  Paris, 1975.

\bibitem{Coxeter}
H.~S.~M. Coxeter.
\newblock The product of the generators of a finite group generated by
  reflections.
\newblock {\em Duke Math. J.}, 18:765--782, 1951.

\bibitem{Dett_Reiter}
Michael Dettweiler and Stefan Reiter.
\newblock The classification of orthogonally rigid {$G_2$}-local systems and
  related differential operators.
\newblock {\em Trans. Amer. Math. Soc.}, 366(11):5821--5851, 2014.

\bibitem{DyckerhoffInverseProb}
Tobias Dyckerhoff.
\newblock The inverse problem of differential galois theory over the field
  $\mathbb{R}(z)$.
\newblock {\em arXiv.org}, arXiv:0802.2897, 2008.

\bibitem{Frenkel/Gross}
Edward Frenkel and Benedict Gross.
\newblock A rigid irregular connection on the projective line.
\newblock {\em Ann. of Math. (2)}, 170(3):1469--1512, 2009.

\bibitem{Gold}
Lawrence {Goldman}.
\newblock {Specialization and Picard-Vessiot theory.}
\newblock {\em {Trans. Am. Math. Soc.}}, 85:327--356, 1957.

\bibitem{Hart}
Julia Hartmann.
\newblock On the inverse problem in differential {G}alois theory.
\newblock {\em J. Reine Angew. Math.}, 586:21--44, 2005.

\bibitem{HumLie}
James~E. Humphreys.
\newblock {\em Introduction to {L}ie algebras and representation theory},
  volume~9 of {\em Graduate Texts in Mathematics}.
\newblock Springer-Verlag, New York-Berlin, 1978.
\newblock Second printing, revised.

\bibitem{LJAL}
Lourdes Juan and Arne Ledet.
\newblock On generic differential {${\rm SO}_n$}-extensions.
\newblock {\em Proc. Amer. Math. Soc.}, 136(4):1145--1153, 2008.

\bibitem{Katz}
Nicholas~M. Katz.
\newblock {\em Exponential sums and differential equations}, volume 124 of {\em
  Annals of Mathematics Studies}.
\newblock Princeton University Press, Princeton, NJ, 1990.

\bibitem{Konst_Betti}
Bertram Kostant.
\newblock The principal three-dimensional subgroup and the {B}etti numbers of a
  complex simple {L}ie group.
\newblock {\em Amer. J. Math.}, 81:973--1032, 1959.

\bibitem{KostantLieGrpRep}
Bertram Kostant.
\newblock Lie group representations on polynomial rings.
\newblock {\em Amer. J. Math.}, 85:327--404, 1963.

\bibitem{Kov}
J.~Kovacic.
\newblock The inverse problem in the {G}alois theory of differential fields.
\newblock {\em Ann. of Math. (2)}, 89:583--608, 1969.

\bibitem{KovacicCyclicVector}
J.~Kovacic.
\newblock Cyclic vectors and picard-vessiot extensions.
\newblock {\em technical report, Prolifics Inc.}, 1996.

\bibitem{Lipshitz&Rubel}
Leonard Lipshitz and Lee~A. Rubel.
\newblock A gap theorem for power series solutions of algebraic differential
  equations.
\newblock {\em Amer. J. Math.}, 108(5):1193--1213, 1986.

\bibitem{Maurischat}
Andreas Maurischat.
\newblock Galois theory for iterative connections and nonreduced {G}alois
  groups.
\newblock {\em Trans. Amer. Math. Soc.}, 362(10):5411--5453, 2010.

\bibitem{Maurischat_simple}
Andreas Maurischat.
\newblock Picard-{V}essiot theory of differentially simple rings.
\newblock {\em J. Algebra}, 409:162--181, 2014.

\bibitem{MitschiSinger}
C.~Mitschi and M.~F. Singer.
\newblock Connected linear groups as differential {G}alois groups.
\newblock {\em J. Algebra}, 184(1):333--361, 1996.

\bibitem{Seiss}
Matthias Sei\ss.
\newblock Root parametrized differential equations for the classical groups.
\newblock {\em arXiv.org}, arXiv:1609.05535.

\bibitem{SerreGaloisCohomology}
Jean-Pierre Serre.
\newblock {\em Galois cohomology}.
\newblock Springer-Verlag, Berlin, 1997.
\newblock Translated from the French by Patrick Ion and revised by the author.

\bibitem{Steinberg}
Robert Steinberg.
\newblock Finite reflection groups.
\newblock {\em Trans. Amer. Math. Soc.}, 91:493--504, 1959.

\bibitem{P/S}
Marius van~der Put and Michael~F. Singer.
\newblock {\em Galois theory of linear differential equations}, volume 328 of
  {\em Grundlehren der Mathematischen Wissenschaften [Fundamental Principles of
  Mathematical Sciences]}.
\newblock Springer-Verlag, Berlin, 2003.

\end{thebibliography}

\end{document}